\documentclass[10pt, leqno]{amsart}

\usepackage[utf8]{inputenc}

\usepackage[square, numbers, comma]{natbib} 

\usepackage[usenames,dvipsnames,svgnames,table]{xcolor}
\usepackage{amsmath,amsfonts,amsthm,amssymb,amssymb,esint,verbatim,tabularx,graphicx,stackrel}
\usepackage{enumerate}
\usepackage{fancyhdr}
\usepackage{epic}
\usepackage{pgf,tikz}
\usepackage{bbm}
\usetikzlibrary{arrows}
\usepackage{microtype}

\usepackage{color}

\usepackage{pdfpages}
\usepackage{mathtools}

\numberwithin{equation}{section}

\usepackage[hypertexnames=false]{hyperref}
\hypersetup{urlcolor=blue, colorlinks=true} 
\usepackage[inner=3cm,outer=3cm,bottom=4cm,top=4cm]{geometry}

\newtheorem{theorem}{Theorem}[section]
\newtheorem{proposition}[theorem]{Proposition}

\newtheorem{lemma}[theorem]{Lemma}
\newtheorem{corollary}[theorem]{Corollary}

\theoremstyle{definition}
\newtheorem{definition}[theorem]{Definition}

\theoremstyle{remark}
\newtheorem{remark}[theorem]{Remark}
\theoremstyle{remark}

\newcommand{\sgn}{\textup{sgn}}

\newcommand{\R}{\mathbb{R}}

\renewcommand{\L}{\mathcal{L}}
\newcommand{\B}{\mathcal{B}}
\newcommand{\F}{\mathcal{F}}
\newcommand{\z}{\zeta}

\renewcommand{\k}{\kappa}
\newcommand{\x}{\xi}
\newcommand{\bu}{\overline{u}^c}
\newcommand{\bv}{\overline{v}^c}
\renewcommand{\a}{\alpha}
\newcommand{\e}{\varepsilon}

\newcommand{\m}{\mu}
\newcommand{\cha}{\chi_\Omega}
\newcommand{\po}{\partial\Omega}
\renewcommand{\z}{\zeta}

\newcommand{\dell}{\partial}
\newcommand{\norm}[2]{\|#1\|_{{#2}}}
\newcommand{\loc}{\mathrm{loc}}

\renewcommand{\b}{\beta}
\renewcommand{\d}{\delta}

\newcommand{\dd}{\,\mathrm{d}}

\DeclareMathOperator{\supp}{supp}

\DeclareMathOperator*{\esssup}{ess\,sup}
\DeclareMathOperator*{\essinf}{ess\,inf}
\DeclareMathOperator*{\esslim}{ess\,lim}
\DeclareMathOperator*{\esslimsup}{ess\,lim\,sup}
\DeclareMathOperator*{\essliminf}{ess\,lim\,inf}

\newcommand{\ol}{\overline}

\newcommand{\nc}{\normalcolor}
\newcommand{\cb}{\color{blue}}

\begin{document}

\title[Nonlocal degenerate parabolic hyperbolic equations on bounded domains]{
Nonlocal degenerate parabolic hyperbolic equations on bounded domains}




\author[N.~Alibaud]{Natha\"el Alibaud}
\address[N.~Alibaud]{LmB (UMR 6623)\\ Universit\'e Marie et Louis Pasteur\\ CNRS, M\'etrologie\\ 16 Rte de Gray\\ 25000 Besan\c{c}on, France\\
and\\
SUPMICROTECH-ENSMM, 
26 Chemin de l'Epitaphe\\ 25030 Besan\c{c}on, France}
\email{nathael.alibaud\@@{}ens2m.fr}

\author[J.~Endal]{J{\o}rgen Endal}
\address[J. Endal]{Department of Mathematical Sciences\\
Norwegian University of Science and Technology (NTNU)\\
N-7491 Trondheim, Norway}
\email[]{jorgen.endal\@@{}ntnu.no}
\urladdr{http://folk.ntnu.no/jorgeen/}

\author[E.~R.~Jakobsen]{Espen R. Jakobsen}
\address[E.~R.~Jakobsen]{Department of Mathematical Sciences\\
Norwegian University of Science and Technology (NTNU)\\
N-7491 Trondheim, Norway}
\email[]{espen.jakobsen\@@{}ntnu.no}
\urladdr{http://folk.ntnu.no/erj/}

\author[O.~M{\ae}hlen]{Ola M{\ae}hlen}
\address[O.~M{\ae}hlen]{Mathematical Institute of Orsay\\
Paris-Saclay University\\
91400 Orsay, France}
\email[]{ola.maehlen\@@{}universite-paris-saclay.fr}



\begin{abstract}
We study well-posedness of degenerate mixed-type parabolic-hyperbolic equations
$$
\partial_tu+\text{div}\big(f(u)\big)=\mathcal{L}[b(u)]
$$
on bounded domains with 
general Dirichlet boundary/exterior conditions. 
The nonlocal diffusion operator $\mathcal{L}$ can be any symmetric Lévy operator (e.g. fractional Laplacians) and $b$ is nondecreasing and allowed to have degenerate regions ($b'=0$). We propose an entropy solution formulation for the problem and show uniqueness of bounded entropy solutions under general assumptions. Existence of solutions is shown in a separate paper.
%
%
The uniqueness proof is based on the Kru\v{z}kov doubling of variables technique and incorporates several a priori results derived from our entropy formulation: an $L^\infty$-bound, an energy estimate, strong initial trace, weak boundary traces, and a \textit{nonlocal} boundary condition. 
%
%
Our work can be seen as both extending nonlocal theories from the whole space to domains and local theories on domains to the nonlocal case. Unlike local theories our formulation does not assume energy estimates. They are now a consequence of the formulation, but as opposed to previous nonlocal theories, they play an essential role in our arguments. 
\end{abstract}

\maketitle

\tableofcontents


\section{Introduction}

In this paper we study well-posedness 
on bounded domains $\Omega\subset \R^d$ for mixed type hyperbolic-parabolic equations with nonlinear and nonlocal diffusion:
%
\begin{equation}\label{E}
\begin{cases} 
\dell_tu+\textup{div}\big(f(u)\big)=\L[b(u)] \qquad&\text{in}\qquad Q\coloneqq(0,T)\times\Omega,\\
u=u^c \qquad&\text{in}\qquad Q^c\coloneqq(0,T)\times\Omega^c,\\
u(0,\cdot)=u_0 \qquad&\text{on}\qquad \Omega,\\
\end{cases}
\end{equation}
where $T>0$, the initial/boundary (exterior) data $u_0,u^c$ are bounded, the ``fluxes" $f,b$ are (at least) locally Lipschitz and $b$ is nondecreasing (possibly degenerate),\footnote{Both $f$ and $b$ may be strongly degenerate in the sense that $f'$ or $b'$ is $0$ on some interval.}  
``$\textup{div}$'' is the $x$-divergence, and
the nonlocal diffusion operator $\L$ 
is 
defined for $\phi \in C_c^\infty(\R^d)$ as a singular integral:
\begin{equation}\label{deflevy}
   \L [\phi](x):=P.V.\int_{|z|>0} \big(\phi(x+z)-\phi(x)\big) \dd \mu(z) \coloneqq\lim_{\epsilon\to0}\int_{|z|>\epsilon} \big(\phi(x+z)-\phi(x)\big) \dd\m(z),
\end{equation}
where the (L\'evy) measure $\mu\geq0$ is symmetric and $\int_{|z|>0}|z|^2\wedge 1 \dd \mu(z)<\infty$.\footnote{Hence $-\L$ is a nonnegative, self-adjoint (on $L^2$), pseudo-differential operator of order $\alpha\in(0,2)$.} This class of anomalous diffusion operators coincides with the generators of the symmetric pure-jump Lévy processes \cite{Ber96, Sat99, Sch03, App09}, including $\alpha$-stable, tempered, relativistic,
and compound Poisson processes. The corresponding generators include the fractional Laplacians \cite{App09,Sat99}
$$\L=-(-\Delta)^{\frac{\alpha}{2}},\quad \alpha\in(0,2),\qquad \text{where}\qquad \dd \mu_\alpha(z)\coloneqq C_{d,\alpha}|z|^{-(d+\alpha)}\dd z,$$ 
anisotropic operators $-\sum_{i=1}^N(-\dell_{x_ix_i}^2)^{\frac{\alpha_i}{2}}$, $\alpha_i\in(0,2)$, relativistic Schr\"odinger 
operators $m^\alpha I-(m^2I-\Delta)^{\frac{\alpha}{2}}$, degenerate and $0$-order operators, and numerical discretizations \cite{DTEnJa18b,DTEnJa19} of these. Since the operators $\L$ are naturally defined on the whole space, they require Dirichlet data on $\Omega^c$  
for \eqref{E} to be well-defined. In addition, inflow conditions are needed in hyperbolic regions. We will discuss this in more detail below.

Equation \eqref{E} is a possibly degenerate nonlinear convection-diffusion equation. 
Convection and diffusion phenomena are ubiquitous in sciences, economics, and engineering, their history dating back centuries. The study of nonlocal models 
is a strong trend across disciplines,
motivated by applications \cite{CoTa04,MeKl04,Woy01}, but also by deep connections between different fields and theoretical breakthroughs within mathematics \cite{Lan72,BoBuCh03,CaSi07,Caf12,Vaz12}.
Of the many applications of equations like \eqref{E},
we mention reservoir simulation \cite{EsKa00}, sedimentation processes \cite{BuCoBuTo99}, and traffic flow \cite{Whi74} in the local 
case (e.g. $\L=\Delta$); detonation in gases \cite{Cla02}, radiation hydrodynamics \cite{Ros89, RoYo07}, and semiconductor growth \cite{Woy01} in the nonlocal 
case; and porous media flow \cite{Vaz07, NC11, DPQuRoVa12} and mathematical finance \cite{Po03,CoTa04} in both cases. 

The solution structure of equation \eqref{E} is very rich. Special cases include scalar conservation laws ($b=0$) \cite{MaNeRoRu96,HoRi02,Daf10}, fractional conservation laws ($b(u)=u$) \cite{BiFuWo98,DrIm06,Ali07}, nonlinear fractional diffusion equations ($f=0$), degenerate porous medium type equations \cite{DPQuRoVa11,DPQuRoVa12,DPQuRoVa14,DPQuRoVa17}, strongly degenerate Stefan type problems \cite{DTEnVa20a,DTEnVa20b}, and problems of mixed hyperbolic-parabolic type  \cite{CiJa11,AlCiJa14,EnJa14}.  
Solutions may have fundamentally different behaviour in different regions:
The problem is hyperbolic with possible discontinuous shock solutions where the diffusion degenerates ($b'(u)=0$).
 In diffusive regions ($b'(u)>0$), the problem is `parabolic' and the behavior of solutions is determined by the order $\alpha$ of $\L$: Smooth solutions in the 
 subcritical/diffusion dominated case $\alpha\in(1,2)$ \cite{DrIm06}, possibly discontinuous solutions in the supercritical/convection dominated case $\alpha\in(0,1)$ \cite{Ali07,AlAn10}. The critical case $\alpha=1$ is delicate, but can sometimes give regular solutions \cite{KiNaSh08,CoVi12}. In subcritical diffusive regions boundary conditions are expected to be satisfied pointwise, while in other regions, there may be sharp boundary layers and loss of boundary conditions (at `outflow').

To accommodate nonsmooth solutions, a weak/distributional solutions concept is needed, and then solutions are nonunique unless additional (entropy) conditions are imposed \cite{Daf10,AlAn10}. In \cite{Kru70} Kru\v{z}kov developed an entropy solution theory and well-posedness for the Cauchy problem for scalar conservation laws \eqref{E} with $b=0$. 
After a long time and attempts by different authors, Carrillo came in \cite{Car99} up with a highly nontrivial extension  
to local (degenerate) convection-diffusion equations \eqref{E} with $\L=\Delta$. 
Inspired by ideas from nonlocal viscosity solution theories (see e.g. \cite{JaKa06,BaIm08}), Alibaud  in \cite{Ali07} introduced an entropy solution theory and well-posedness for the Cauchy problem for \eqref{E} with linear diffusion $b(u)=u$ and $\L=-(-\Delta)^{\frac{\alpha}2}$. Cifani and Jakobsen in \cite{CiJa11} then extended these ideas to cover the Cauchy problem for \eqref{E} in full generality, including
nonlinear possibly degenerate diffusions and arbitrary pure-jump L\'evy diffusion operators. In nonlocal entropy solution theories, a key point is to split the nonlocal operator in a singular and a nonlingular part, where the singular part vanishes after taking limits, and the nonsingular part leads to beneficial cancellations in the Kru\v{z}kov doubling of variables argument. This idea has no local analogue, and the theory is quite different from  local theories \cite{Car99,KaRi01,ChPe03}. We refer to \cite{KaUl11,AlCiJa12,AlCiJa14,CiJa14,EnJa14,IgSt18,BKV20,AAO20} and references therein for more results on Cauchy problems and recent extensions.

\subsubsection*{Hyperbolic and mixed type boundary value problems}
In hyperbolic problems, boundary conditions can be imposed only at inflow, where characteristics go into the domain.
When $b=0$ and \eqref{E} is a nonlinear scalar conservation law, 
\begin{equation}\label{eq:SCL}
\begin{cases} 
\dell_tu+\textup{div}\big(f(u)\big)=0 \qquad&\text{in}\qquad Q,\\
u=u^c \qquad&\text{on}\qquad \Gamma\coloneqq(0,T)\times\dell\Omega,\\
u(0,\cdot)=u_0 \qquad&\text{on}\qquad \Omega,\\
\end{cases}
\end{equation}
characteristics depend on $u$ itself and the correct way to impose boundary conditions only at inflow is
\begin{equation}\label{eq:BLN}
\sgn(u-u^c)\big(f(u)-f(k)\big)\cdot\hat{n}\geq 0 \qquad  \forall k: u\wedge u^c\leq k\leq u\vee u^c \qquad \text{on $(0,T)\times\dell\Omega$},
\end{equation}
where $\hat{n}$ is the outward unit normal. This strong formulation of Bardos, Le Roux, and Nedelec \cite{BaLRNe79} (the BLN condition) requires $BV$ regularity and strong boundary traces for $u$.
 Otto introduced in \cite{Ott96} (see also \cite{MaNeRoRu96}) an $L^\infty$-theory using weak boundary traces and a weak entropy formulation on the boundary, and 
 then this theory was simplified by  the semi-Kru\v{z}kov boundary entropy formulation of Vovelle \cite{Vov02} --  
 see also \cite{EyGaHe00}. 
 In \cite{ChFr99} divergence-measure fields were analyzed to obtain weak traces and applied to \eqref{eq:SCL}. The divergence structure of the equation makes entropy solutions regular at the boundaries, and $L^\infty$-solutions of \eqref{eq:SCL} are shown to have strong traces in \cite{Vas01,Pan05,KwVa07}. The weak trace approach of Otto is in a sense bypassed, but the case of general fluxes $f=f(t,x,u)$ is not well-understood. In the BV setting, an extensive overview of well-posedness for \eqref{eq:SCL} is given in \cite{Ros18}, see also \cite{Ser96,Daf10}. Measure-valued solutions are considered in  \cite{Sze89,Val00}, renormalized solutions in \cite{CaWi02, PoVo03, AmWiCa06}, kinetic solutions in \cite{ImVo04}, saturated solutions in \cite{LiSo18}, and Robin and Neumann conditions in \cite{AnSb15}.

When both convection and nonlinear local diffusion ($\L=\Delta$) are present in \eqref{E}, the problem is of mixed type and will be hyperbolic in regions where the diffusion degenerates. A correct formulation taking into account inflow conditions are needed, and the first result was given in  \cite{Car99} for homogeneous Dirichlet conditions. The case of $u^c\not\equiv0$ is more delicate in the sense that the question of traces becomes more visible. Well-posedness is proved by Mascia, Poretta, and Terracina \cite{MaPoTe02} using a careful parabolic extension of the Otto-formulation. Michel and Vovelle \cite{MiVo03} recasted the problem as a single integral inequality  using semi-Kru\v{z}kov entropies, and also proved existence and convergence of finite volume schemes. Further improvements can be found in \cite{Val05}. Strong traces exist in some cases \cite{Kwo09, FrLi17}, strong entropy solutions can be found in \cite{RoGa01, AnGa16} and kinetic solutions in \cite{Kob06, Kob07}. We also mention work on related boundary conditions \cite{AnBo04, BuFrKa07}, on anisotropic diffusion \cite{KoOh12, LiWa12, Wan16}, and on doubly and triply nonlinear degenerate parabolic equations \cite{BeWi96, BlMu97, CaWi99, AmWi03, BlPo05, BeKa05, AnIg07, AnBeKaOu09, Amm10, DrEyTa16}.

\subsubsection*{Nonlocal boundary value problems}
Our nonlocal operators $\L$ are naturally defined on the whole space. For problem \eqref{E} to be well-defined, we need to restrict $\L$ to the bounded domain $\Omega$.
Likely the most classical way of doing this is to impose Dirichlet data on the complement $\Omega^c$ as in \eqref{E}. In probability theory this corresponds to stopping or killing the underlying (particle) processes the first time it exits from $\Omega$ \cite{Dy65}. 
This approach 
is by far the most popular one in the PDE setting, with a large literature including important contributions of Caffarelli and Silvestre -- see e.g. \cite{CaSi07,Caf12,Vaz12,R-Ot16,BuVa16}. As opposed to the local case, there are many other ways to restrict $\L$ to a bounded domain and impose Dirichlet conditions. 
Equivalent definitions 
of nonlocal operators 
in $\R^d$ \cite{Kwa17} via PDEs, probability, calculus of variations, spectral theory etc.,  may differ on 
domains and lead to different 
interpretations of Dirichlet conditions. Prominent examples are approaches based on censoring (regional fractional Laplcians)  or spectral theory (spectral fractional Laplcians) \cite{BoBuCh03,Gr18} where data is only required on the boundary $\partial \Omega$. But the resulting operators $\L$ are no longer translation invariant, and their integral representations depend explicitly on $\Omega$. See e.g. \cite{Gr18} for a comparison of the three approaches discussed above and \cite{BoFiVa18a, BoFiVa18b} for a discussion in a nonlinear PDE setting.

Problem \eqref{E} is a nonlocal problem of mixed type, and any formulation of it needs to take into account inflow conditions in hyperbolic regions. 
Up to now the literature on such problems is very limited. For purely parabolic equations ($f\equiv0$), boundary value problems have been studied in e.g. \cite{BoVa15, BoVa15b, BoFiVa18a, BoFiVa18b}. Here there are no hyperbolic regions and inflow conditions are not needed.
When the diffusion is linear and non-degenerate ($b=\textup{Id}$), $\L$ is the spectral fractional Laplacian, and $u^c\equiv0$, there are results 
in \cite{Bra16, Kan18, Kan20}. 
Here the boundary condition is satisfied in a pointwise sense when $\alpha\in(1,2)$. Recently, Huaroto and Neves \cite{HuNe22} showed 
existence 
of $L^\infty$ entropy solutions of a problem similar to \eqref{E}, but with local boundary conditions and $\L=-(-\Delta)^{\frac\alpha2}_{\Omega}$, the regional fractional Laplacian of order $\alpha\in(1,2)$.\footnote{For functions $\phi\in C^2(\R^d)$ supported in $\Omega$ and $x\in\Omega$,  from the definitions we immediately have that
$$-(-\Delta)^{\frac\alpha2}_{\Omega}\phi(x)=-(-\Delta)^{\frac\alpha2}\phi(x)+\kappa(x)\phi(x)\qquad\text{where}\qquad \kappa(x)=C_{d,\alpha}\int_{\Omega^c}|z|^{-(d+\alpha)}\dd z \sim d_{\partial\Omega}(x)^{-\alpha}$$ 
is singular on the boundary. This operator does not belong to the class of operators considered in this paper.}  This setting is well adapted to analysis in Sobolev spaces $H^s(\Omega)$ with strong boundary traces, and the proofs use the (local) vanishing viscosity method, a priori estimates on smooth solutions, and compactness arguments to pass to the limit. From a nice $L^1$-contraction type of estimate, they also obtain uniqueness of the vanishing viscosity limits.
However as explained in \cite[page 3]{HuNe22}, there is currently no general uniqueness result for entropy solutions of this model.
\footnote{In no non-trivial case does the theory of  \cite{HuNe22} coincide with ours, even for existence. The proofs are also quite different from ours, e.g. we do not use smooth approximations, but rather the entropy formulation directly.
} 



\subsubsection*{Main contributions of the paper:} 
\medskip
\begin{itemize}
    \item [1.] An appropriate entropy solution formulation for 
    problem \eqref{E}.
    \medskip
    \item[2.] A uniqueness result for entropy solutions of \eqref{E}.
\end{itemize}
\medskip
The most difficult part of well-posedness for \eqref{E} is uniqueness.
Existence of entropy solutions of \eqref{E} holds under mild additional assumptions (see Theorem \ref{thm:ExistenceEntropyProcessSolution2} below), but due to the length of the current paper and additional technicalities and difficulties, the existence theory is given in a separate paper \cite{AlEnJaMa23b}. 

Entropy solutions of \eqref{E} are defined to be $L^\infty$ functions satisfying pointwise exterior conditions and a family of weak semi-Kru\v{z}kov entropy inequalities incorporating both the equation in the interior and inflow conditions in hyperbolic parts of the boundary. The definition includes splitting the nonlocal operator into singular and nonsingular parts as in \cite{Ali07,CiJa11} and a version of the entropy formulation on the boundary given in \cite{MaPoTe02,MiVo03}. See Definition \ref{def: entropySolutionVovelleMethod} for the precise statement. The entropy inequalities include initial and boundary terms that impose strong initial traces (Lemma \ref{lem: timeContinuityAtZeroOfEntropySolution}), hyperbolic weak traces on the boundary (Proposition \ref{prop: weakTrace}) and a boundary condition (Proposition \ref{prop: boundaryConditionEntropySolution}) that coincides with the BLN-condition \eqref{eq:BLN} whenever strong traces exist \cite{Vov02}. 

The most important contribution of the paper is the uniqueness result for our class of entropy solutions (Theorem \ref{thm: uniquenessEntropySolutions}). As in the local case,\footnote{Problem \eqref{E} with $\L=\Delta$, see e.g. \cite{Car99,MaPoTe02,MiVo03} for uniqueness results.} it takes the form of a partial
$L^1$-contraction estimate, $\|u(t,\cdot)-v(t,\cdot)\|_{L^1}\leq\|u_0-v_0\|_{L^1}$ for solutions $u$ and $v$ coinciding on $\Omega^c$. This result seems to be the first uniqueness result for degenerate parabolic hyperbolic
equations that are nonlocal. 
The proof is based on a highly non-trivial adaptation of the Kru\v{z}kov doubling of variables technique and combine nonlocal elements from \cite{Ali07,CiJa11} with arguments to handle boundary conditions from \cite{MiVo03}. 
%
%
Working in a domain requires a lot of extra ingredients and is technically  more involved than working on the whole space.
To handle nonlocal problems on domains,
we develop new and useful identities and bounds for Lévy operators in Section \ref{sec:Preliminaries}, we prove nonlocal versions of weak trace, boundary conditions, energy estimates, a purely new result on boundary integrability in Section \ref{sec: propertiesOfEntropySolutions}, and finally, we change the procedure of the uniqueness argument itself in Section \ref{sec: uniqueness}. In the local case \cite{MiVo03}, the doubling of variables arguments roughly speaking use carefully chosen boundary layer sequences and passage to the limit (i) to undo the doubling and then (ii) to see the boundary. To construct the inner boundary layer sequences, rather nice and explicit solutions of simple Poisson problems are then used. 
Such arguments are difficult to generalize to our family of nonlocal operators which may not be uniformly elliptic and have integral representations that are more difficult to handle. We overcome this problem by taking cruder `outer' boundary layer sequences and reverse the order of the limits (i) and (ii) mentioned above. The price for this approach is more regularity requirements on $b$ and $\overline u^c$ in problem \eqref{E}, see Remark \ref{rem: explainingHowTheProofOfTheMotherLemmaDiffersFromTheOneInMiVO} for more details.



 One key part of uniqueness proofs on domains is making sense of and controlling fluxes at the boundary. To give some insight into our new nonlocal formulation, consider e.g. the the in-flux of $|u-\overline{u}|$ at the boundary. In the local case ($\L=\Delta$) it is formally given as the measure 
$$
\Big(\sgn(u-\overline{u}^c)\big(f(u)-f(\overline{u}^c)\big)-\nabla\big|b(u)-b(\overline{u}^c)\big|\Big)\cdot\nabla\cha\qquad\text{on}\qquad \Gamma=(0,T)\times \partial\Omega,
$$
where $\chi_\Omega$ is the characteristic function on $\Omega$. The existence and control of such a (weak) flux-trace is part of the solution-concept in \cite{MaPoTe02}, but not in \cite{MiVo03} where it instead follows from the entropy inequalities and the assumption of finite energy. Here we work with a nonlocal weak trace given by
$$
\sgn(u-\overline{u}^c)\big(f(u)-f(\overline{u}^c)\big)\cdot\nabla\cha-\B\big(|b(u)-b(\overline{u}^c)|,\cha\big)\qquad\text{in} \qquad M\coloneqq(0,T)\times\R^d,
$$
where $\B$ is the bilinear operator whose integral is the bilinear form associated to $\L$. The latter measure has a singular part on $\Gamma$ and an absolutely continuous part in $M\setminus \Gamma$ which is due to the nonlocal $\B$-term. Like in \cite{MiVo03}, we also prove the necessary properties of this trace (Proposition \ref{prop: weakTrace}) from the entropy formulation using a (nonlocal) energy estimate. Contrary to the local case, this energy estimate is not a part of the formulation, but follows as a consequence of it (see Proposition \ref{eq: finiteEnergy}).\footnote{Energy estimates play no role in nonlocal problems on the whole space \cite{CiJa11}, but are still necessary for local ones \cite{KaRi03}.}

\subsubsection*{Outline of paper}Section \ref{sec:AssumptionsConceptMain} is reserved for assumptions, concept of solution, and main results, and in Section \ref{sec:Preliminaries} we provide preliminary results on nonlocal operators. We collect important properties of entropy solutions, like energy estimates, boundary conditions, and weak traces, in Section \ref{sec: propertiesOfEntropySolutions}. These properties are then used to prove uniqueness in Section \ref{sec: uniqueness}. The Appendices contain auxiliary results and calculations.


\section{Assumptions, concept of solution, and main results}
\label{sec:AssumptionsConceptMain}
In this section we state the notation and assumptions, present the (entropy) solution concept we will use, and give the main results -- including uniqueness, a priori estimates, and boundary trace/condition results. We also state the existence result from \cite{AlEnJaMa23b}.

\subsubsection*{Notation}
Let
$\phi \vee^{+} \psi\coloneqq \max\{\phi,\psi\}=\phi \vee \psi$, $\phi\vee^{-} \psi \coloneqq \min\{\phi,\psi\}=\phi\wedge \psi$, and
$\sgn^\pm(a)=\pm1$ if $\pm a >0$ and zero otherwise. In addition to the open and bounded set $\Omega\subset \R^d$, we will be working with the sets 
$$
Q=(0,T)\times\Omega,\quad \Gamma=(0,T)\times\dell\Omega,\quad \overline{\Omega}=\Omega\cup \dell \Omega,
$$
$$
Q^c=(0,T)\times\Omega^c, \quad \ol{\Gamma}= [0,T]\times\dell\Omega, \quad M=(0,T)\times\R^d. 
$$
We will write $|\Omega|$ to mean the $d$-dimensional Lebesgue measure of $\Omega$, and $|Q|\coloneqq T|\Omega|$. The $(d-1)$-dimensional Hausdorff measure on $\po$ is denoted by $\dd \sigma(x)$. In the double variable scenario $(t,x),(s,y)\in \R\times\R^d$, we will for brevity often omit the differentials $\dd t\dd x \dd s\dd y$ from integrals featuring all four variables. Instead, we specify the relevant variables for each domain by writing
\begin{align*}
    \Omega_x,\qquad\Omega_y,\qquad Q_x,\qquad Q_y,\qquad \Gamma_x,\qquad \Gamma_y,\qquad M_x,\qquad 
    M_y,
\end{align*}
where $x$ indicates that the domain of integration is over $(t,x)$, while $y$ indicates integration over $(s,y)$.

 The Lévy operator $\L$ is as defined in \eqref{deflevy}, and we define its corresponding bilinear operator by
\begin{align}\label{def:blform}
   \B[\phi,\psi](x) \coloneqq\lim_{\epsilon\to0}\frac{1}{2}\int_{|z|>\epsilon}  \big(\phi(x+z)-\phi(x)\big)\big(\psi(x+z)-\psi(x)\big)\dd\m(z),  
\end{align}
where the limit will typically be taken in $L^1(\R^d)$ or $L^1_{\loc}(\R^d)$; Proposition \ref{prop: BCompatiblePairs} lists conditions on $\phi,\psi$ ensuring that this limit exists. 
We denote truncated operators by
\begin{equation}\label{eq:DifferentTruncatedOperators}
    \L^{\geq r}, \qquad\L^{<r},\qquad\L^{r'> \cdots \geq r},\qquad\qquad
    \B^{\geq r}, \qquad\B^{<r},\qquad\B^{r'> \cdots \geq r},
\end{equation}
where the domains of integration are restricted to $\{|z|\geq r\}$, $\{|z|<r\}$, and $\{r'> |z|\geq r\}$; observe that $ \L^{\geq r}, \L^{<r},\L^{r'> \cdots \geq r}$ are themselves Lévy operators. When the functions $\phi,\psi$ depend on two space variables $(x,y)\in \R^d\times \R^d$, we shall write $\L_x$, $\L_y$, $\B_x$, $\B_y$ to specify in which of the variables the operators act. And acting in both variables, is the frequently occurring Lévy operator 
\begin{equation}\label{eq: theXPlussYOperators}
    \begin{split}
        \L_{x+y}[\phi](x,y)\coloneqq&\, \lim_{\epsilon\to 0}\int_{|z|\geq \epsilon}\big(\phi(x+z,y+z)-\phi(x,y)\big)\dd \mu(z),\\
\B_{x+y}[\phi,\psi](x,y)\coloneqq&\, \frac{1}{2}\lim_{\epsilon\to 0}\int_{|z|\geq \epsilon}\big(\phi(x+z,y+z)-\phi(x,y)\big)\big(\psi(x+z,y+z)-\psi(x,y)\big)\dd \mu(z),
    \end{split}
\end{equation}
and the (less frequent) mixed bilinear operators 
\begin{equation}\label{eq: theMoreMessyXPlussYOperators}
    \begin{split}
\B_{x,y}[\phi,\psi](x,y)\coloneqq&\, \frac{1}{2}\lim_{\epsilon\to 0}\int_{|z|\geq \epsilon}\big(\phi(x+z,y)-\phi(x,y)\big)\big(\psi(x,y+z)-\psi(x,y)\big)\dd \mu(z),\\
\B_{x,x+y}[\phi,\psi](x,y)\coloneqq&\, \frac{1}{2}\lim_{\epsilon\to 0}\int_{|z|\geq \epsilon}\big(\phi(x+z,y)-\phi(x,y)\big)\big(\psi(x+z,y+z)-\psi(x,y)\big)\dd \mu(z),
 \end{split}
\end{equation}
with $\B_{y,x}$ and $\B_{y,x+y}$ defined similarly. For all these operators, we denote their truncated versions by notation analogous to that in \eqref{eq:DifferentTruncatedOperators}.

We identify $L^2(\Omega)$ as a subspace of $L^2(\R^d)$ through zero extensions and define $H^{\L}_0(\Omega)\subset L^2(\Omega)$ to be the Hilbert space with norm
\begin{align}\label{eq: definitionOfLevySpaceNorm}
     \norm{\phi}{H^{\L}_0(\Omega)}^2\coloneqq\norm{\phi}{L^2(\Omega)}^2 + \int_{\R^d}\B[\phi,\phi]\dd x.
\end{align}
The Lipschitz constant of a function $g$ on a 
set $K$ is denoted $L_{g,K}$ or simply $L_g$ if the set is clear from the context.

\subsubsection*{Assumptions} In this paper we will use the following assumptions:\medskip
\begin{align}
&\textup{$\Omega\subset \R^d$ is open, bounded with $C^2$-boundary $\dell\Omega$, and outward pointing normal $\hat{n}$}.
\tag{$\textup{A}_\Omega$}
\label{Omegaassumption}\\[0.2cm]
&f=(f_1,f_2,\ldots,f_d)\in W_{\textup{loc}}^{1,\infty}(\mathbb{R};\mathbb{R}^d).
\tag{$\textup{A}_f$}
\label{fassumption}\\[0.2cm]
&\text{$b\in W_{\textup{loc}}^{1,\infty}(\R;\R)$ is non-decreasing, and the weak derivative $ b'\in{TV_{\textup{loc}}(\R)}$.}
\tag{$\textup{A}_b$}
\label{bassumption}\\[0.2cm]
\tag{$\textup{A}_{u^c}$}&u^c\in (C^2\cap L^\infty)(Q^c)\textup{ and has an extension $\overline{u}^c\in (C^2\cap L^\infty)([0,T]\times\R^d)$}.
\label{u^cassumption}\\[0.2cm]
&u_0\in L^{\infty}(\Omega).
\label{u_0assumption}
\tag{$\textup{A}_{u_0}$}\\[0.2cm]
&\textup{$\mu$ is a symmetric Radon measure on $\R^d\setminus\{0\}$, $\mu\geq0$, and $\textstyle\int_{\R^d\setminus\{0\}}\big(|z|^2\wedge 1\big)\dd \mu(z)<\infty$.
}
\label{muassumption}
\tag{$\textup{A}_{\mu}$}
\end{align}
We extend $\mu$ to $\R^d$  by setting $\mu(\{0\})=0$: This is for notational brevity. On $\R^d$, $\mu$ is a locally unbounded Borel measure, but still $\sigma$-finite, and so the coming product measures will be unambiguous.
\begin{remark}\label{assumptionremark}\leavevmode
\begin{enumerate}[{\rm (a)}]
\item In \eqref{fassumption} and \eqref{bassumption}, we can assume without loss of generality that $f(0)=0$ and $b(0)=0$ (add constants to $f$ and $b$), and that $f$ and $b$ are globally Lipschitz and $b'$ has bounded total variation (solutions are uniformly bounded by Lemma \ref{lem: maximumsPrinciple}). 
\item In \eqref{bassumption} 
the condition $b'\in TV_{\textup{loc}}(\R)$ implies $b\in W^{1,\infty}_{\textup{loc}}(\R)$, the standard assumption for the Cauchy problem. Our stronger assumption still allows for classical power type and strongly degenerate nonlinearities  $b(r)=r^m$ for  $m>1$ (porous medium) and  $b(r)=\max\{r-L,0\}$ 
(one-phase Stefan). 
\item 
Under our assumptions, 
 $\L[b(\overline{u}^c)]\in L^1(Q)$ (Corollary \ref{cor: L1LocControlOfLevyOnComposition}) and 
$\dell_t\ol{u}^c+\textup{div}_x\big(f(\ol{u}^c)\big)-\L[b(\ol{u}^c)]\in L^1(Q)$,
which will be important in many calculations and results concerning $u-\ol{u}^c$.
\end{enumerate}
\end{remark}

\subsubsection*{Entropy solutions} For the definition of entropy solutions we introduce the semi-Kru\v{z}kov entropy-entropy flux pairs
\begin{equation*}
(u-k)^\pm,\qquad
F^{\pm}(u,k)\coloneqq\sgn^\pm(u-k)(f(u)-f(k))
\qquad  \text{for}\qquad u,k\in\mathbb{R},
\end{equation*}
where $(\cdot)^+\coloneqq\max\{\cdot,0\}$ and $(\cdot)^-\coloneqq(-\cdot)^+$, and the
usual splitting of the nonlocal operator,
\begin{equation*}
\L[\phi](x)=\L^{<r}[\phi](x)+\L^{\geq r}[\phi](x) \qquad\text{for}\qquad \phi\in C_c^{\infty}(\R^d),\ r>0,\ x\in \mathbb{R}^d,
\end{equation*}
where $\L^{<r}$ and $\L^{\geq r}$ are defined in \eqref{eq:DifferentTruncatedOperators}.
To motivate the definition (see also Appendix \ref{sec:ClassicalSolutionsAreEntropySolutions}), assume $u$ is a smooth solution of \eqref{E}, $0\leq \varphi\in C_c^\infty([0,T)\times\R^d)$, and $k\in\R$. Replace $(u,f(u),b(u))$ by $(u-k,f(u)-f(k),b(u)-b(k))$ in \eqref{E}, multiply the resulting PDE by $\sgn^{\pm}(u-k)\varphi$, and integrate over $Q$. After several integrations by parts and chain rules, 
we get the {\em equality}
\begin{align*}
&\,-\int_{Q}\Big((u-k)^{\pm} \dell_t\varphi +F^{\pm}(u,k)\cdot\nabla \varphi\Big)\dd x\dd t \\
&\,-\int_M \L^{\geq r}[b(u)-b(k)]\sgn^{\pm}(u-k)\varphi\cha\dd x\dd t -\int_M (b(u)-b(k))^{\pm}\L^{<r}[\varphi]\dd x \dd t\\
=&\, \int_\Omega(u_0-k)^{\pm}\varphi(0,\cdot)\dd x + L_f\int_\Gamma (\overline{u}^c-k)^\pm\varphi\dd\sigma(x)\dd t - \int_\Gamma \varphi \dd \mu_{\textup{c}}^{\pm}- \int_M \varphi \dd \mu_{\textup{d}}^{\pm,r},
\end{align*}
where the $\mu$-measures are absolutely continuous with respect to $\dd\sigma \dd t$ and $\dd x \dd t$ with densities
\begin{align*}
&
 F^{\pm}(u,k)\cdot\hat{n}+L_f(u-k)^{\pm}\qquad\text{and}\qquad
\L^{<r}[b(u)-b(k)]\sgn^\pm(u-k)\cha-\L^{<r}[(b(u)-b(k))^\pm].
\end{align*}
For our definition we want these measures to be nonnegative. For the classical measure $\mu_{\textup c}^{\pm}$, this follows by definition of $F^{\pm}$ and Lipschitz conitnuity of $f$. The new measure $\mu_{\textup d}^{\pm,r}$ however, requires an additional assumption. Since $b'\geq0$ and $\sgn^\pm(u-k)(b(u)-b(k))=\sgn^\pm(b(u)-b(k))(b(u)-b(k))$, if 
$$
\sgn^\pm(b(u^c)-b(k))\varphi=0 \quad \text{a.e.~in }Q^c,
$$
then we get
\begin{equation*}
\begin{split}
&\,\int_M \Big(\L^{<r}[b(u)-b(k)]\sgn^\pm(b(u)-b(k))\cha-\L^{<r}[(b(u)-b(k))^\pm]\Big)\varphi\dd x\dd t\\
= &\,\int_M \Big(\L^{< r}[b(u)-b(k)]\sgn^\pm(b(u)-b(k))-\L^{<r}[(b(u)-b(k))^\pm]\Big)\varphi\dd x\dd t.
\end{split}
\end{equation*}
The last integral is nonpositive by the convex inequality $\sgn^+ (v)\, \L^{<r}[v]\leq \L^{<r}[v^+]$ (cf. Corollary \ref{cor:ConvexInequalityNonlocalOperators}), and hence we can conclude that also $\mu_{\textup d}^{\pm,r}$ is nonnegative.

Our definition is then:

\begin{definition}[Entropy solution]\label{def: entropySolutionVovelleMethod}
A function $u\in L^\infty(M)$ is an \emph{entropy solution} of \eqref{E}
if:
\begin{enumerate}[(a)]
\item\label{def: entropyInequalityVovelleMethod-1} \textup{(Entropy inequalities in $\overline Q$)} For all $r>0$, and all $k\in\R$ and $0\leq \varphi\in C^\infty_c([0,T)\times\R^d)$ satisfying 
\begin{align}\label{eq: admissibilityConditionOnConstantAndTestfunction}
(b(u^c)-b(k))^\pm\varphi = 0\qquad \text{in $Q^c$},
\end{align}
the following inequality holds
\begin{equation}\label{eq: entropyInequalityVovelleMethod}
\begin{split}
&\,-\int_Q \Big( (u-k)^\pm\dell_t\varphi + F^\pm(u,k)\cdot\nabla\varphi \Big)\dd x\dd t\\
&\,-\int_Q\L^{\geq r}[b(u)]\sgn^\pm(u-k)\varphi\dd x\dd t - \int_M(b(u)-b(k))^\pm\L^{< r}[\varphi]\dd x\dd t\\
\leq&\,\int_\Omega (u_0-k)^\pm\varphi(0,\cdot) \dd x + L_f\int_\Gamma (\overline{u}^{c}-k)^\pm\varphi \dd\sigma(x)\dd t.
\end{split}
\end{equation}
\item\label{def: entropyInequalityVovelleMethod-2} \textup{(Data in $Q^c$)} $u=u^c$ a.e.~in $Q^c$.
\end{enumerate}
\end{definition}
This definition is an extension of both \cite[Definition 2.1]{CiJa11} (see also \cite{Ali07}) for nonlocal problems in the whole space and \cite[Definition 2.1]{MiVo03} (see also \cite{MaPoTe02}) for local problems on bounded domains. In the hyperbolic case ($b'\equiv0$), it is equivalent \cite{Vov02} to the original definition of Otto \cite{Ott96}. By \eqref{Omegaassumption}--\eqref{muassumption}, all the integrals in \eqref{eq: entropyInequalityVovelleMethod} are well-defined.
\begin{remark}\label{rem: defentsolnremark}\leavevmode
\begin{enumerate}[(a)]
\item The condition $(b(u^c)-b(k))^+\varphi = 0$ is weaker than $|b(u^c)-b(k)|\varphi = 0$. 
When $b(u^c)< b(k)$, the first condition holds for any $\varphi$ while the second then implies that $\varphi=0$.
This gives a hint to why standard Kru\v{z}kov entropy-entropy flux pairs are too restrictive in this setting, see \cite{Vov02} for counterexamples to uniqueness in the hyperbolic case when $b=0$.
\item \label{rem: defentsolnremark(b)} The term $\mathcal{L}^{\geq r}[b(u)]$ is a convolution of a bounded Lebesgue-measurable function $b(u)$ 
and a finite Borel measure $\mu$. It is well-defined in the sense that $\L^{\geq r}$ maps all Borel representatives of $b(u)$ to the same element in $L^\infty(M)$; see \cite[Proposition 8.49]{Fo:book} for details.

\end{enumerate}
\end{remark}

\subsubsection*{A priori results and uniqueness} A first consequence of our definition is the  $L^\infty$ a priori estimate for solutions.
\begin{lemma}[$L^\infty$-bound
]\label{lem: maximumsPrinciple}
Assume \eqref{Omegaassumption}--\eqref{muassumption} and  $u$ is an entropy solution of \eqref{E}. Then
\begin{align*}
   \min\Big\{\essinf_{\Omega} u_0, \essinf_{Q^c} u^c\Big\} \leq u(t,x)\leq  \max\Big\{\esssup_{\Omega} u_0,  \esssup_{ Q^c} u^c\Big\}\qquad\text{for a.e.~$(t,x)\in Q$.}
\end{align*}
\end{lemma}

The weakly posed initial conditions $u_0$ are in fact assumed in the strong $L^1$-sense.

\begin{lemma}[Time continuity at $t=0$]\label{lem: timeContinuityAtZeroOfEntropySolution}
Assume \eqref{Omegaassumption}--\eqref{muassumption} and  $u$ is an entropy solution of \eqref{E}. 
Then 
\begin{align*}
    \esslim_{t\to0+}\norm{u(t,\cdot)-u_0}{L^1(\Omega)} =0.
\end{align*}
\end{lemma}

The rather standard proofs of these two results can be found in Appendix \ref{app:div_pfs}. We continue by listing additional a priori estimates that are needed for the uniqueness proof. The proofs are given in Section \ref{sec: propertiesOfEntropySolutions}. 
Recall that $\B$ is the bilinear operator associated with $\L$ defined in \eqref{def:blform}, and define
$$
F(u,k)\coloneqq (F^++F^-)(u,k)=\sgn(u-k)(f(u)-f(k)).
$$
{
\def\thetheorem{\ref{prop: finiteEnergy}}
\begin{proposition}[Energy estimate]
Assume \eqref{Omegaassumption}--\eqref{muassumption} and $u$ is an entropy solution of \eqref{E}.
Then 
{
\def\theequation{\ref{eq: finiteEnergy}}
\begin{align}\nonumber
&\,\int_M \B\big[b(u)-b(\bu),b(u)-b(\bu)\big] \dd x\dd t\\
    \leq &\,\int_{\Omega} H(u_0,\bu(0))\dd x - \int_{Q} \Big[(u-\bu) \bu_t + F (u,\bu)\cdot\nabla\bu\Big]b'(\bu)\dd x\dd t \\
   &\,+ \int_{Q} \L[b(\bu)](b(u)-b(\bu))\dd x\dd t,
    \nonumber
\end{align}
}\addtocounter{equation}{-1}
where  
$
H(u,k):=\int_k^u \big(b(\xi)-b(k)\big)\dd \xi \geq 0.$ 
 Moreover, the right-hand side of \eqref{eq: finiteEnergy} is finite.
\end{proposition}
}
\addtocounter{theorem}{-1}

\begin{remark}\label{remark: energyEstimateImpliesLocalBoundedEnergyOfBU}
This global energy estimate for $b(u)-b(\bu)$ implies 
a local energy estimate for $b(u)$: By the inequality $x^2\leq 2(x-y)^2 + 2y^2$ we have
$$
\int_M \B\big[b(u),b(u)\big]\varphi \dd x\dd t\leq 2\int_M \B\big[b(u)-b(\bu),b(u)-b(\bu)\big]\varphi \dd x\dd t+2\int_M \B\big[b(\bu),b(\bu)\big] \varphi\dd x\dd t
$$
for any $0\leq \varphi \in C_c^\infty(M)$, where 
$\B\big[b(\bu),b(\bu)\big]$ is locally bounded since $b(\bu)$ is bounded and locally Lipschitz. 
If $\int_M \B[b(\bu),b(\bu)] \dd x\dd t<\infty$, we can send $\varphi\to 1$ to get a global energy estimate for $b(u)$. 
\end{remark}

To write down our results regarding boundary integrability, weak trace, and boundary condition, we need different sequences of test functions to approximate the (characteristic 
function of the) domain. We say $(\z_\d)_{\d>0}\subset C^\infty_c(\R^d)$ is a \textit{boundary layer sequence} if, for some Borel function $\z\colon \po\to[0,1]$,
\begin{equation}\label{eq: generalizedBoundaryLayerSequence}
0\leq \z_\d(x)\leq 1,\quad x\in\R^d, \quad\forall \delta>0,\qquad \lim_{\d\to0}\z_\d(x)=\begin{cases}
    1,\quad\quad\quad x\in \Omega,\\
        \z(x),\quad\hspace{7pt} x\in\po,\\
        0,\quad\quad\quad x\in\R^d\setminus\Omega.
    \end{cases}
\end{equation}
Boundary layer sequences $(\overline{\z}_\delta)_{\d>0}$ and $(\z_\delta)_{\d>0}$ 
are called \textit{outer} and \textit{inner boundary layer sequences} 
if
\begin{align}\label{exteriorBoundaryLayerSequence}
&\overline{\z}_\d=1 
\quad\text{in}\quad \overline{\Omega}\qquad\text{and}\qquad 
\lim_{\d\to0}\overline{\z}_\d(x)=\chi_{\overline{\Omega}}(x), \quad x\in\R^d,\\
&  \label{interiorBoundaryLayerSequence}
\supp \z_\d\subset \Omega
\qquad\text{and}\qquad\lim_{\d\to0}\z_\d(x)=\chi_{\Omega}(x),\quad x\in\R^d.
\end{align} 
Typically, we will require of these sequences to satisfy 
$\limsup_{\d\to0}\norm{\nabla\z_\d}{L^1(\R^d)}<\infty$.  All boundary layer sequences \cb  satisfy the weak limit $-\nabla\z_\d\to\hat{n}\delta_{\partial \Omega}$: \nc 
 For any $\psi\in C^1(\R^d)$,
\[
\lim_{\d\to0}\int_{\R^d}\psi(-\nabla\z_\d) \dd x=\lim_{\d\to0}\int_{\R^d}\textup{div}(\psi)\z_\d\dd x=\int_{\Omega}\textup{div}(\psi)\dd x=\int_{\dell\Omega}\psi\cdot\hat{n}\dd \sigma(x).
\]
We note that outer and inner boundary layer sequences are easily constructed by setting
\begin{align*}
   \overline{\z}_{\d}\coloneqq \rho_\delta\ast\chi_{\Omega_{\delta}},\qquad \z_{\d}\coloneqq \rho_\delta\ast \chi_{\Omega_{-\delta}},  
\end{align*}
where $\rho_\delta(x)= \frac{1}{\d^d}\rho(\frac x\delta)$ is a standard mollifier with support in $\{|x|<\delta\}$, 
$\Omega_{\d}=\{x\colon \mathrm{dist}(x,\Omega)<\d\}$, and $\Omega_{-\d}=\{x\colon \mathrm{dist}(x,\Omega^c)>\d\}$. By \eqref{Omegaassumption}, $\partial\Omega$ and $\partial \Omega_{\pm\d}$ are $C^2$ with finite Hausdorff $\mathcal{H}^{d-1}$-measure when $\delta$ is small. 
Moreover, 
$\lim_{\d\to0}\norm{\nabla\z_\d}{L^1(\R^d)}=\lim_{\d\to0}\norm{\nabla\overline{\z}_\d}{L^1(\R^d)}=|\partial\Omega|_{\mathcal{H}^{d-1}}<\infty$. 
%

The next result restricts how much $b(u)$ can differ from $b(\bu)$ close to the boundary $\Gamma$. It is a purely nonlocal result and prevents blow-up of certain integrals.

{\def\thetheorem{\ref{prop: boundaryIntegrability}}
\begin{proposition}[Boundary integrability]
Assume \eqref{Omegaassumption}--\eqref{muassumption} and $u$ is an entropy solution of \eqref{E}.
Then {
\def\theequation{\ref{eq: theEquationOfBoundaryIntegrability}}
\begin{equation}
    \begin{split}
            \,\int_{0}^T\int_{\substack{x\in\Omega\\ x+z\in\Omega^c}}|b(u(t,x))-b(\bu(t,x))|\dd \mu(z)\dd x\dd t
            \leq&\,
    C,
    \end{split}
\end{equation}
}\addtocounter{equation}{-1}
where $C=C(\Omega,T,f,b,\mu,\bu,u_0)$ is finite and given by \eqref{eq: constantBoundingMotherLemma}.
\end{proposition}
}\addtocounter{theorem}{-1}
The two next results have local analogues and are essential for the uniqueness proof.

{\def\thetheorem{\ref{prop: weakTrace}}
\begin{proposition}[Weak trace]
 Assume \eqref{Omegaassumption}--\eqref{muassumption} and $u$ is an entropy solution of \eqref{E}.
 Let $\nu$ be the finite Borel measure from \eqref{eq: theMeasureForTheWeakTrace}. Then, for any boundary layer sequence $(\z_\d)_{\d>0}$, $r>0$, and $\varphi\in C^\infty_b([0,T]\times \R^d)$,  we have
{\def\theequation{\ref{eq: weakTrace}}
\begin{equation}
\begin{split}
       &\,\lim_{\d\to0}\Bigg(\int_Q F(u,\overline{u}^c)\cdot (\nabla \z_\d) \varphi \dd x\dd t-\int_{M}\B^{<r}[ |b(u)-b(\overline{u}^c)|,\z_\d]\varphi \dd x \dd t\Bigg)\\
     = &\,\int_{\overline{\Gamma}}(1-\z)\varphi \dd \nu  - \int_M \B^{<r}\big[|b(u)-b(\overline{u}^c)|,\cha\big]\varphi \dd x \dd t,
\end{split}
\end{equation}
}\addtocounter{equation}{-1}where $\overline{\Gamma}= [0,T]\times\partial\Omega$, $F(u,\bu)=\sgn(u-\bu)(f(u)-f(\bu))$, and $\z$ is the limit of $\z_\d$ on $\po$. The right-hand side of \eqref{eq: weakTrace} side is non-positive, and the expression inside the $\lim$ on the left-hand side 
is
bounded by $\tilde{C}\norm{\varphi}{C^1(M)}$ for a $\tilde{C}>0$ independent of $\d,\varphi$.
\end{proposition}
}\addtocounter{theorem}{-1}

\begin{remark} The $\B^{<r}$-term on the right-hand side is finite by boundary integrability Proposition \ref{prop: boundaryIntegrability}.
\end{remark}

The $\B^{<r}$-terms are new and represent the contribution from the nonlocal diffusion, while the hyperbolic part 
(the $F$ and $\nu$ terms) coincides with 
the weak trace of Otto\footnote{It is an equivalent restatement in terms of semi-Kru\v{z}kov entropy-entropy fluxes due to Vovelle \cite{Vov02} of 
the result of Otto (see Definition 7.2 and Lemma 7.34 in \cite{MaNeRoRu96}).}  for scalar conservation laws \cite{Vov02,MaNeRoRu96}. 
The `flux density' $F(u,\overline{u}^c)\cdot\nabla\cha - \B[|b(u)-b(\overline{u}^c)|,\cha]$  
measures the inflow of $|u-\bu|$ into $Q$. The convective part is a singular measure concentrated on $\Gamma$ while the nonlocal diffusive part is absolutely continuous and supported on $M$. The term $-\B[|b(u)-b(\overline{u}^c)|,\cha](t,x)$ (see \eqref{def:blform}) represents the inflow density from $(t,x)\in Q^c$ into $Q$ or minus the outflow density from $(t,x)\in Q$ to $Q^c$.

The weak trace becomes important in the uniqueness proof during the Kru\v{z}kov doubling of variables technique where we want to replace `difficult' terms by `easy' terms using the boundary condition below. Define the Kru\v{z}kov boundary entropy-entropy flux pairs as
\begin{equation*}
\begin{split}
\mathcal{F}= \mathcal{F}(u,\overline{u}^c,k) &\coloneqq F(u,\overline{u}^c)+F(u,k)-F(\overline{u}^c,k),\\
\Sigma=\Sigma(u,\overline{u}^c,k) &\coloneqq |b(u)-b(\overline{u}^c)| + |b(u)-b(k)|-|b(\overline{u}^c)-b(k)|.
\end{split}
\end{equation*}

{
\def\thetheorem{\ref{prop: boundaryConditionEntropySolution}}
\begin{proposition}[Boundary condition]
Assume \eqref{Omegaassumption}--\eqref{muassumption} and $u$ is an entropy solution of \eqref{E}. 
Then, for any inner boundary layer sequence $(\z_\d)_{\d>0}$, $k\in\R$, $r>0$, and $0\leq \varphi\in C^\infty_b([0,T]\times\R^d)$, 
{
\def\theequation{\ref{eq: theUsefulBoundaryCondition}}
\begin{align}
    \lim_{\d\to 0}&\Bigg( \int_Q \big(\mathcal{F}\cdot\nabla\z_\d\big)\varphi \dd x\dd t- \int_M \B^{<r}\big[\Sigma,\z_\d\big]\varphi \dd x\dd t\Bigg)\leq 0.
\end{align}
}Moreover, the 
expression inside the $\lim$ 
is bounded by 
$\tilde{C}\norm{\varphi}{C^1(M)}$ for a $\tilde{C}>0$ independent of $\d,k,\varphi$.
\addtocounter{equation}{-1}
\end{proposition}
} \addtocounter{theorem}{-1}
At this point we could formulate an equivalent definition of entropy solutions in the spirit of Mascia, Porretta, and Terracina \cite[Definition 1.2]{MaPoTe02}, by combining  Lemma \ref{lem: timeContinuityAtZeroOfEntropySolution}, Propositions \ref{prop: boundaryIntegrability}, \ref{prop: weakTrace}, \ref{prop: boundaryConditionEntropySolution}, and the classical Kru\v{z}kov entropy inequalities in $Q$.
However, the way the 
Definition \ref{def: entropySolutionVovelleMethod} is formulated now, makes it easier to use in the existence proof where we take limits of approximate solutions (see \cite{AlEnJaMa23b}). 

We now state the uniqueness result, the main contribution of this paper.

{\def\thetheorem{\ref{thm: uniquenessEntropySolutions}}
\begin{theorem}[Uniqueness]
Assume \eqref{Omegaassumption}--\eqref{muassumption}. Let $u,v$ be entropy solutions of \eqref{E} in the sense of Definition \ref{def: entropySolutionVovelleMethod} with initial data $u_0,v_0$ and exterior data $u^c=v^c$. Then for a.e.~$t\in (0,T)$ we have
\begin{align*}
    \int_{\Omega}|u(t,x)-v(t,x)|\dd x \leq \int_{\Omega}|u_0(x)-v_0(x)|\dd x.
\end{align*}
In particular, if also $u_0=v_0$, then $u=v$ a.e.~in $(0,T)\times\R^d$.
\end{theorem}
}\addtocounter{theorem}{-1}

Finally, we state the existence result from \cite{AlEnJaMa23b}. Here we need slightly stronger assumptions on $\L/\mu$:
\begin{align}
&\textup{$\mu$ is such that the symbol  $m(\xi)\coloneqq\textstyle\int_{\R^d}\big(1-\cos(\xi\cdot z)\big)\dd\mu(z)\to\infty$ \ as \ $|\xi|\to\infty$.}
\label{muassumption2}
\tag{$\textup{A}_{\mu}'$}
\end{align}
The symbol $m$ is the multiplier of $-\L$, and the condition is sufficient (and necessary) for certain strong $L^2$-compactness to follow from $\L$-energy estimates. This condition requires a non-finite $\mu$ and is implied if the absolutely continuous (with respect to the Lebesgue measure) part of $\mu$ is non-finite; see \cite{AlEnJaMa23b} for details. In particular, it is satisfied for fractional Laplacians where the symbol is $m_\alpha(\xi)= |\x|^{\alpha}$.

\begin{theorem}[Existence]\label{thm:ExistenceEntropyProcessSolution2}
Assume \eqref{Omegaassumption}--\eqref{muassumption} and either \eqref{muassumption2} or $\mu(\R^d)<\infty$. Then there exists an entropy solution $u$ of \eqref{E}. 
\end{theorem}

The proof in \cite{AlEnJaMa23b} uses a series of approximations of the nonlocal operators, entropy process solutions, nonlinear weak-$\star$ compactness techniques, and strong $L^2$ compactness inherited from the energy estimate and novel arguments to control time-translations. We also show that \eqref{muassumption2} is sufficient and necessary for $\L$-energy estimates to induce strong spatial compactness in $L^2$.

\begin{remark}\leavevmode 
\begin{enumerate}[{\rm (a)}]
\item We extend the uniqueness result Theorem \ref{thm: uniquenessEntropySolutions} to entropy-process solutions in \cite{AlEnJaMa23b}.
\item We have to work in $L^\infty$, and not in $C([0,T];L^1)\cap L^\infty$, since the $L^1$-contraction in Theorem \ref{thm: uniquenessEntropySolutions} is not strong enough to control space-translations of solutions in term of space-translations of the data. Our existence proof is therefore based on weak compactness arguments that does not preserve the $C([0,T];L^1(\Omega))$-regularity of approximate solutions. 
\end{enumerate}
\end{remark}




\section{Preliminaries on the Lévy operator}
\label{sec:Preliminaries}

This section is devoted to proving various identities and control estimates for the operators $\L$ and $\B$. We will here say that the expressions $\L[\phi]$ and $\B[\psi,\varphi]$ are well-defined in $L^1(\R^d)$ or $L^1_{\loc}(\R^d)$ when the limits $\lim_{r\searrow 0}\L^{\geq r}[\phi]$ and $\lim_{r\searrow 0}\B^{\geq r}[\psi,\varphi]$ exist in $L^1(\R^d)$ or $L^1_{\loc}(\R^d)$ respectively.

\subsection{Identities for $\L$ and $\B$}
We begin by establishing a connection between $\L$ and $\B$ and follow up with a product rule for $\B$. These results resemble well-known identities for $\Delta$ and $\nabla$, and the proofs boil down to truncations of the operators $\L$ and $\B$, interchanging integrals, and change of variables. To ensure that Fubini's theorem can be applied, some `absolute integrabilty' of the relevant components is required.  For this purpose an assumption of compact support is included in the propositions (though various other assumptions would suffice) as we shall frequently encounter compactly supported functions. 

\begin{proposition} [Integration by parts formulas]\label{prop: integrationByPartsFormula} 
Assume \eqref{muassumption}, $\phi,\psi\in L^\infty(\R^d)$, and that $\phi$ or $\psi$ has compact support. If both $\L[\phi]$ and $\B[\phi,\psi]$ are well-defined in $L^1(\R^d)$, then
\begin{align}\label{eq: fromLevyToBilinearForm}
    \int_{\R^d}\L[\phi]\psi\dd x = -\int_{\R^d}\B[\phi,\psi]\dd x.
\end{align}
If both  $\L[\phi]$ and $\L[\psi]$ are well-defined in $L^1(\R^d)$, then
\begin{align}\label{eq: fromLevyToLevy}
    \int_{\R^d}\L[\phi]\psi\dd x = \int_{\R^d}\phi\L[\psi]\dd x.
\end{align}
\end{proposition}
\begin{proof}
We prove only \eqref{eq: fromLevyToBilinearForm} as the proof of \eqref{eq: fromLevyToLevy} is similar. Fix $r>0$, and note that $(z,x)\mapsto \big(\phi(x+z)-\phi(x)\big)\psi(x)$ is absolutely integrable on the set $\{(z,x)\in\R^d\times\R^d:|z|\geq r\}$ with respect to $\dd\mu(z)\dd x$ (see the proof of Proposition \ref{prop: BCompatiblePairs} for similar expressions treated in detail).We may then use Fubini's theorem to compute 
\begin{align*}
   \int_{\R^d} \L^{\geq r}[\phi]\psi\dd x = &\, \int_{\R^d}\int_{
   |z|\geq r} \big(\phi(x+z)-\phi(x)\big)\psi(x)\dd\mu(z)\dd x\\
   =&\,\int_{
   |z|\geq r} \int_{\R^d}\big(\phi(x)-\phi(x-z)\big)\psi(x-z)\dd x\dd\mu(z)\\
     =&\, -\int_{\R^d}\int_{
   |z|\geq r} \big(\phi(x+z)-\phi(x)\big)\psi(x+z)\dd\mu(z)\dd x,
\end{align*}
where we made the substitutions $x\mapsto x-z$ and then $z\mapsto -z$ exploiting the symmetry of $\mu$.
Thus
\begin{align*}
   \int_{\R^d} \L^{\geq r}[\phi]\psi\dd x = &\, \frac{1}{2}\int_{\R^d}\int_{
   |z|\geq r} \big(\phi(x+z)-\phi(x)\big)\big(\psi(x)-\psi(x+z)\big)\dd\mu(z)\dd x= -\int_{\R^d}\B^{\geq r}[\phi,\psi]\dd x.
\end{align*}
The result follows by letting $r\to 0$.
\end{proof}
\begin{remark}
If $\psi$ in the previous proposition has compact support, the result holds true even when $\L[\phi]$ is merely well-defined in the $L^1_{\loc}$-sense; the proof is identical.
\end{remark}

The next formula is a nonlocal analogue of the product rule $\nabla\phi\cdot \nabla(\psi\varphi)=(\nabla\phi\cdot \nabla\psi)\varphi + (\nabla\phi\cdot \nabla\varphi)\psi$. 

\begin{proposition}[Product rule]\label{prop: productRuleForB}
Assume \eqref{muassumption}, $\phi,\psi,\varphi\in L^\infty(\R^d)$, and either $\phi$, $\psi$ or $\varphi$, has compact support. Then, if $\B[\phi,\psi\varphi]$, $\B[\phi,\psi]$, and $\B[\phi,\varphi]$, 
are well-defined in $L^1(\R^d)$, we have
\begin{align}\label{eq: theProductRuleOfB}
\int_{\R^d}\B[\phi,\psi\varphi]\dd x = \int_{\R^d}\B[\phi,\psi]\varphi\dd x + \int_{\R^d}\B[\phi,\varphi]\psi\dd x.
\end{align}
\end{proposition}

\begin{proof}
Fix $r>0$. Similar to the previous proof, the compact support of one of $\phi,\psi,\varphi$ will guarantee that the following integrands are absolutely integrable with respect to $\dd\mu(z)\dd x$ on the given domain. By Fubini's theorem we compute
\begin{align*}
     \int_{\R^d}\big(\B^{\geq r}[\phi,\psi\varphi] - \B^{\geq r}[\phi,\psi]\varphi\big) \dd x
     = &\, \frac{1}{2}\int_{\R^d}\int_{|z|\geq r} \big(\phi(x+z)-\phi(x)\big)\big(\varphi(x+z)-\varphi(x)\big)\psi(x+z) \dd \mu(z)\dd x\\
     = &\, \frac{1}{2}\int_{|z|\geq r}\int_{\R^d} \big(\phi(x)-\phi(x-z)\big)\big(\varphi(x)-\varphi(x-z)\big)\psi(x)\dd x \dd \mu(z)\\
     = &\,\frac{1}{2}\int_{\R^d}\int_{|z|\geq r} \big(\phi(x+z)-\phi(x)\big)\big(\varphi(x+z)-\varphi(x)\big)\psi(x) \dd \mu(z)\dd x\\
     =&\,\int_{\R^d} \B^{\geq r}[\phi,\varphi]\psi \dd x,
\end{align*}
where we made the substitution $x\mapsto x-z$ followed by $z\mapsto -z$ and exploited the symmetry of $\dd\mu(z)$. The result follows by letting $r\to 0$.
\end{proof}

\subsection{Bounds for $\L$ and $\B$}\label{subsec:bnds}
We here provide bounds and inequalities that will be used to make sense of the expression $\L[\phi]$ and $\B[\psi,\varphi]$ in non-trivial situations. We begin with a standard estimate for smooth functions.

\begin{lemma}\label{lem: L^pBoundnessOfLvWhenVIsC2AndCompactSpport}
Let $p\in[1,\infty]$, $\phi\in C^2_c(\R^d)$, and let $D^2\phi(x)$ denote the Hessian matrix of $\phi$ at $x$. Then
\begin{align}\label{eq: L^pBoundnessOfLvWhenVIsC2AndCompactSpport}
    \norm{\L[\phi]}{L^p(\R^d)}\leq&\, \Big[2\norm{\phi}{L^p(\R^d)} + \tfrac{1}{2}\norm{D^2\phi}{L^p(\R^d)}\Big]\int_{\R^d}\Big(|z|^2\wedge 1\Big)\dd\mu(z),
\end{align}
where $\norm{D^2\phi}{L^p(\R^d)}$ denotes the $L^p$ norm of $x\mapsto |D^2\phi(x)|$, with $|D^2\phi|$ being the spectral radius of $D^2\phi$. 
\end{lemma}
Up to the choice of constants, this result is proved e.g. in \cite{EJ21}, see Lemma 2.1 and Remark 2.2.

When $\phi\in L^p(\R^d)$ for some $p\in[1,\infty]$, the previous lemma guarantees that $\L[\phi]$ makes sense as a distribution if we canonically define $\langle \L [\phi],\varphi\rangle\coloneqq \langle \phi,\L [\varphi]\rangle$ for all $\varphi\in C^2_c(\R)$. Moreover, it follows that $\L^{\geq r}[\phi]\to \L [\phi]$ in the sense of distributions as $r\to0$. However, a distributional notion of $\L[\phi]$ will often be too weak for our purpose, and so the next proposition is useful. 

\begin{proposition}[Lévy operator on compositions]\label{prop: L1ControlOfLevyOfCompositions}
Assume $h:\R\to\R$ admits a weak derivative of bounded variation and $\phi\in L^\infty(\R^d)$ has compact support. If $\L[\phi]$ is well-defined in $L^1(\R^d)$, then so is $\L[h(\phi)]$, and moreover we have 
\begin{align}\label{eq: L1ControlOfLevyOfCompositions}
     \norm{\L[h(\phi)]}{L^1(\R^d)}\leq\Big[\norm{h'}{L^\infty(\R)} +|h'|_{TV(\R)}\Big]\norm{\L [\phi]}{L^1(\R^d)}.
\end{align}
\end{proposition}

\begin{remark}
Note that $|h'|_{TV(\R)}<\infty$ implies $\norm{h'}{L^\infty(\R)}<\infty$ so that $h$ is actually Lipschitz.
\end{remark}

\begin{proof}[Proof of Proposition \ref{prop: L1ControlOfLevyOfCompositions}]
As $h'$ is of bounded variation, associate it with its left-continuous representation (so that the following pointwise estimates are unambiguous).
We consider first the case when the Lévy measure is finite $\mu(\R^d)<\infty$, so that $\L$ is a zero order operator; in particular we have $\L[\psi]=\mu\ast\psi - \mu(\R^d)\psi$ for any $\psi\in L^\infty(\R^d)$. Due to its bounded variation, $h'$ admits an essentially unique representation $h'=\b_+ - \b_-$ for two non-decreasing functions $\b_+,\b_-$ satisfying 
\begin{align*}
    |h'|_{TV(\R)}=|\b_+|_{TV(\R)}+|\b_-|_{TV(\R)} = \b_+(\infty)+\b_-(\infty)-\b_+(-\infty)-\b_-(-\infty).
\end{align*}
Introducing then the monotone function $\b(v)\coloneqq \b_+(v) + \b_-(v) - \tfrac{1}{2}|h'|_{TV(\R)}$, which has been shifted to satisfy $2\norm{\b}{L^\infty(\R)}=|h'|_{TV(\R)}$, we observe that we have the inequality
\begin{align}\label{eq: inequalityForMonotonizedH}
    |h'(v+w)-h'(v)| \leq \sgn(w)\big(\b(v+w) - \b(v)\big),
\end{align}
for any $v,w\in\R$.

In the next computation, let $\d\coloneqq \phi(x+z)-\phi(x)$ for notational simplicity. Subtracting $h'(\phi)\L [\phi]$ from $\L[h(\phi)]$ we find for a.e.~$x\in\R^d$
\begin{equation}\label{eq: inequalityForSubtractedHDerivative}
    \begin{split}
        \Big|\L[h(\phi)]-h'(\phi)\L [\phi]\Big|(x)
    =&\, \bigg|\int_{\R^d}h\big(\phi(x) + \delta\big)-h(\phi(x)) - h'(\phi(x))\delta\dd \mu(z)\bigg|\\
    \leq &\, \int_{\R^d}\int_{0}^1\big|\big(h'(\phi(x) +s\delta )-h'(\phi(x))\big)\delta\big|\dd s \dd \mu(z)\\
    \leq &\, \int_{\R^d}\int_{0}^1\big(\b(\phi(x) +s\delta )-\b(\phi(x)\big)\delta\dd s \dd \mu(z),
    \end{split}
\end{equation}
where we used \eqref{eq: inequalityForMonotonizedH}. By monotonicity of $\b$, we further have $(\b(\phi + s\d)-\b(\phi))\d\leq (\b(\phi + \d)-\b(\phi))\d$ whenever $s\in[0,1]$. And so by \eqref{eq: inequalityForSubtractedHDerivative} we infer 
\begin{align}\label{eq: almostAtLevyControlOfCompositions}
   \norm{\L[h(\phi)] - h'(\phi)\L[\phi]}{L^1(\R^d)} \leq \int_{\R^d}2\B[\beta(\phi),\phi]\dd x \leq 2\norm{\b}{L^\infty(\R^d)}\norm{\L[\phi]}{L^1(\R^d)},
\end{align}
where we shifted $\B$ over to $\L$ using Proposition \ref{prop: integrationByPartsFormula}. Thus, \eqref{eq: L1ControlOfLevyOfCompositions} follows for zero order $\L$ when combining the triangle inequality with \eqref{eq: almostAtLevyControlOfCompositions} and the bound $2\norm{\b}{L^\infty(\R^d)} = |h'|_{TV(\R)}$.

For general $\L$, we first note that $\L[h(\phi)]$ is necessarily well-defined in $L^1(\R^d)$: Let $r'>r>0$ and note that
\begin{align*}
    \limsup_{r',r\to 0}\norm{\L^{\geq r'}[h(\phi)]-\L^{\geq r}[h(\phi)]}{L^1(\R^d)} =&\,  \limsup_{r',r\to 0}\norm{\L^{r'>\cdots \geq r}[h(\phi)]}{L^1(\R^d)}\\
    \lesssim &\,  \limsup_{r',r\to 0}\norm{\L^{r'>\cdots \geq r}[\phi]}{L^1(\R^d)}=0,
\end{align*}
where we used \eqref{eq: L1ControlOfLevyOfCompositions} on the zero order Lévy operator $\L^{r'>\cdots \geq r}$ and the fact that $\L[\phi]$ is well-defined in $L^1(\R^d)$. The inequality \eqref{eq: L1ControlOfLevyOfCompositions} then follows for $\L[h(\phi)]$ by using the corresponding one for $\L^{\geq r}[h(\phi)]$ and letting $r\to 0$.
\end{proof}

We combine the previous two results to compute a local $L^1$ bound for $\L[h(\phi)]$ with $h$ as above and with $\phi$ bounded and locally $C^2$.

\begin{corollary}\label{cor: L1LocControlOfLevyOnComposition}
Let $h$ be as in Proposition \ref{prop: L1ControlOfLevyOfCompositions}, and let $\phi\in C^2_b(\R^d)$. Then for any compact $U\subset \R^d$ and any smooth cut-off function $\psi\in C_c^\infty(\R^d)$ satisfying $\psi(x)=1$ whenever $\mathrm{dist}(x,U)\leq 1$ we have 
\begin{equation}\label{eq: explicitBoundOfLevyOnCompositionOnBoundedSets}
\begin{split}
    \norm{\L[h(\phi)]}{L^1(U)}\leq&\, \Big[\norm{h'}{L^\infty(\R)} +|h'|_{TV(\R)}\Big]\times \\
    &\,  \times\Big[ 4\norm{\phi}{L^\infty(\R^d)}\norm{\psi}{L^1(\R^d)} +\tfrac{1}{2}\norm{D^2(\phi\psi)}{L^1(\R^d)}\Big]\int_{\R^d}\Big(|z|^2\wedge 1\Big)\dd\mu(z),
\end{split}
\end{equation}
where $\norm{D^2(\phi\psi)}{L^1(\R^d)}$ is as in Lemma \ref{lem: L^pBoundnessOfLvWhenVIsC2AndCompactSpport}. Thus $\L[h(\phi)]\in L^1_{\mathrm{loc}}(\R^d)$, and in particular, we have
\begin{align*}
    \lim_{r\to 0} \L^{\geq r}[h(\phi)] = \L[h(\phi)]\qquad \text{in }L^1_{\mathrm{loc}}(\R^d).
\end{align*}
\end{corollary}
\begin{proof}
To prove \eqref{eq: explicitBoundOfLevyOnCompositionOnBoundedSets}, we decompose $\L=\L^{\geq 1}+\L^{<1}$.
As $\L^{<1}[h(\phi)]=\L^{<1}[h(\phi\psi)]$ on $U$, we find
\begin{align*}
    \norm{\L^{<1}[h(\phi)]}{L^1(U)}\leq&\, \norm{\L^{<1}[h(\phi\psi)]}{L^1(\R^d)}\\
    \leq &\,\Big[\norm{h'}{L^\infty(\R)} +|h'|_{TV(\R)}\Big]\Big[ 2\norm{\phi\psi}{L^1(\R^d)} +\tfrac{1}{2}\norm{D^2(\phi\psi)}{L^1(\R^d)}\Big]\int_{\R^d}\Big(|z|^2\wedge 1\Big)\dd\mu(z),
\end{align*}
by Proposition \ref{eq: L1ControlOfLevyOfCompositions} and Lemma \ref{lem: L^pBoundnessOfLvWhenVIsC2AndCompactSpport}. And using that $\L^{\geq 1}[h(\phi)]=\psi\L^{\geq 1}[h(\phi)]$ on $U$, we compute
\begin{align*}
    \norm{\L^{\geq 1}[h(\phi)]}{L^1(U)}\leq&\,  \norm{\psi\L^{\geq 1}[h(\phi)]}{L^1(\R^d)}\leq 2\norm{\psi}{L^1(\R^d)}\norm{h'}{L^\infty(\R)}\norm{\phi}{L^\infty(\R^d)}\mu(\{|z|\geq 1\}).
\end{align*}
Putting these estimates together (plus a few coarse estimates) give  \eqref{eq: explicitBoundOfLevyOnCompositionOnBoundedSets}. The fact that $\lim_{r\to0}\L^{\geq r}[h(\phi)]=\L[h(\phi)]$ in $L^1_{\mathrm{loc}}(\R)$ follows if we can prove $L^1$ convergence on $U$ as the latter set was arbitrary. This fact follows immediately from \eqref{eq: explicitBoundOfLevyOnCompositionOnBoundedSets} with $\L^{< r}$ as the respective Lévy operator; we get
\begin{align*}
    \norm{\L[h(\phi)]-\L^{\geq r}[h(\phi)]}{L^1(U)}=\norm{\L^{< r}[h(\phi)]}{L^1(U)}\lesssim \int_{|z|<r}\Big(|z|^2\wedge 1\Big)\dd\mu(z),
\end{align*}
which tends to zero as $r\to0$.
\end{proof}

We end the section by listing a few sufficient conditions on $\phi,\psi\in L^\infty(\R^d)$ for $\B[\phi,\psi]$ to be well-defined. Here we shall say that `the integrand of $\B[\phi,\psi]$ is absolutely integrable with respect to $\dd\mu(z)\dd x$' to mean
\begin{equation}\label{eq: integrandOfBIsAbsolutelyIntegrable}
    \frac{1}{2}\int_{\R^d\times\R^d}|\phi(x+z)-\phi(x)||\psi(x+z)-\psi(x)|\dd\mu(z)\dd x<\infty.
\end{equation}
This condition ameliorates the expression $\B[\phi,\psi]$ and in particular ensures it to be well-defined in $L^1(\R^d)$; even the pointwise definition \eqref{def:blform} is then necessarily meaningful almost everywhere. 
\begin{proposition}[List of $\B$-compatible pairs]\label{prop: BCompatiblePairs}
 Let $\phi,\psi\in L^\infty(\R^d)$.  Then, the integrand of $\B[\phi,\psi]$ is absolutely integrable with respect to $\dd\mu(z)\dd x$ if either of the following conditions are satisfied:
\begin{enumerate}[{\rm (i)}]
    \item $\phi,\psi\in H^{\L}(\R^d)$. The integral in \eqref{eq: integrandOfBIsAbsolutelyIntegrable} is then bounded by 
    \begin{equation*}
        \bigg(\int_{\R^d}\B[\phi,\phi]\, \dd x\bigg)^{\frac{1}{2}}\bigg(\int_{\R^d}\B[\psi,\psi]\, \dd x\bigg)^{\frac{1}{2}}.
    \end{equation*}
    \item $\phi\in H^{\L}(\R^d)$, $\psi$ is Lipschitz continuous, and $\phi$ or $\psi$ is supported in a compact set $E\subset \R^d$. The integral in \eqref{eq: integrandOfBIsAbsolutelyIntegrable} is then bounded by
    \begin{equation*}
         \sqrt{2|E|}\bigg(\int_{\R^d}(|z|^2\wedge 1)\,d\mu(z)\bigg)^{\frac{1}{2}}\bigg(\int_{\R^d}\B[\phi,\phi]\, \dd x\bigg)^{\frac{1}{2}}\big(L_\psi \vee 2\norm{\psi}{L^\infty}\big).
    \end{equation*}
    \item $\phi$ is of bounded variation, $\psi$ is Lipschitz continuous, and $\phi$ or $\psi$ is supported in a compact set $E\subset \R^d$. The integral in \eqref{eq: integrandOfBIsAbsolutelyIntegrable} is then bounded by 
     \begin{equation*}
        \bigg(\int_{\R^d}\big(|z|^2\wedge 1\big)\,d\mu(z)\bigg)\big(|\phi|_{TV}\vee 2|E|\norm{\phi}{L^\infty}\big)\big(L_{\psi}\vee 2\norm{\psi}{L^\infty}\big)
    \end{equation*}
    \item The Lévy  measure $\mu$ is finite, and $\phi$ or $\psi$ is supported in a compact set $E\subset \R^d$. The integral in \eqref{eq: integrandOfBIsAbsolutelyIntegrable} is then bounded by 
     \begin{equation*}
       4|E|\mu(\R^d)\norm{\phi}{L^\infty}\norm{\psi}{L^\infty}.
    \end{equation*}
\end{enumerate}
Here, $|E|$ denotes the $d$-dimensional Lebesgue measure of $E$.
\end{proposition}

\begin{proof}
\noindent(i). This follows from the definition of $H^{\L}(\R^d)$ and the Cauchy–Schwarz inequality.

\smallskip
\noindent(ii). We here have
\begin{align*}
   &\, \frac{1}{2}\int_{\R^d\times\R^d}|\phi(x+z)-\phi(x)||\psi(x+z)-\psi(x)|\dd\mu(z)\dd x\\
    \leq&\, \frac{1}{2}\int_{\R^d\times\R^d}|\phi(x+z)-\phi(x)||\psi(x+z)-\psi(x)|\big(\chi_E(x) + \chi_E(x+z)\big)\dd\mu(z)\dd x,
\end{align*}
where $\chi_{E}$ is the indicator function on $E$. By splitting the integral on the right-hand side in two, one integral featuring $\chi_E(x)$ and the other $\chi_E(x+z)$, one may perform the substitution of variable $x\mapsto x-z$ followed by $z\mapsto -z$ to see that the two parts coincide. Thus, the integral is bounded by 
\begin{align*}
   &\,\int_{\R^d\times\R^d}|\phi(x+z)-\phi(x)||\psi(x+z)-\psi(x)|\chi_E(x)\dd\mu(z)\dd x\\
   \leq &\,  \int_{\R^d\times\R^d}\Big[|\phi(x+z)-\phi(x)|\Big]\Big[\big(|z|\wedge 1\big)\chi_E(x)\Big]\dd\mu(z)\dd x\big(L_\psi \vee 2 \norm{\psi}{L^\infty}\big).
\end{align*}
The result then follows from Hölder's inequality applied to the two square brackets. 

\smallskip
\noindent(iii). We again introduce $\chi_E(x)$ into the integral and get the desired bound from the calculation 
\begin{align*}
    \int_{E}|\phi(x+z)-\phi(x)||\psi(x+z)-\psi(x)|\dd x
   \leq &\, \int_{E}|\phi(x+z)-\phi(x)|\dd x \Big(L_{\psi}|z|\wedge 2\norm{\psi}{L^\infty}\Big)\\
   \leq &\, \Big(|\phi|_{TV}|z|\wedge 2|E|\norm{\phi}{L^\infty}\Big) \Big(L|z|\wedge 2\norm{\psi}{L^\infty(\R^d)}\Big)\\
   \leq&\, \big(|z|^2\wedge 1\big)\big(|\phi|_{TV}\vee 2|E|\norm{\phi}{L^\infty}\big)\big(L_{\psi}\vee 2\norm{\psi}{L^\infty}\big).
\end{align*}

\smallskip
\noindent(iv). As before, we only need to control the integral of $|\phi(x+z)-\phi(x)||\psi(x+z)-\psi(x)|\chi_E(x)\leq 4\norm{\phi}{L^\infty}\norm{\psi}{L^\infty}\chi_E(x)$. Integrating the latter quantity gives the claimed bound.
\end{proof}

We will frequently encounter the expression $\B[\cha,\varphi]$ where $\cha$ is the characteristic function on $\Omega$ and where $\varphi$ is bounded and Lipschitz continuous. Because $\partial \Omega$ has finite Hausdorff $\mathcal{H}^{d-1}$-measure (due to its regularity), we get that $|\cha|_{TV(\R^d)}=|\partial \Omega|_{\mathcal{H}^{d-1}}<\infty.$\footnote{This identity follows from the definition of $|\cdot|_{TV(\R^d)}$ and the Gauss divergence theorem.}  By the previous proposition we conclude that $\B[\cha,\varphi]$ is well-defined in $L^1(\R^d)$. 

We conclude the section with a similar, but more specific, proposition for the double variable scenario $(x,y)\in \R^d\times\R^d$. The corresponding operators are defined in Section \ref{sec:AssumptionsConceptMain}, specifically \eqref{eq: theXPlussYOperators} and \eqref{eq: theMoreMessyXPlussYOperators}.
\begin{proposition}\label{prop: absolutelyIntegrableBInTheDoubleVariableScenario}
    Let $\phi,\psi\in L^\infty(\R^d)$ and $\varphi\in C^\infty_c(\R^d\times\R^d)$ depend on $x$, $y$, and $(x,y)$ respectively, and $\B_x[\phi,\phi]$ and $\B_y[\psi,\psi]$ be well-defined in $L^1_{\loc}(\R^d)$. Then the integrands of
    \begin{align*}
        \B_{x+y}[|\phi-\psi|,\varphi],\quad \B_{x, x+y}[|\phi-\psi|,\varphi], \quad\B_{y,x+y}[|\phi-\psi|,\varphi],\quad \B_{x,y}[|\phi-\psi|,\varphi], \quad\text{and}\quad \B_{y,x}[|\phi-\psi|,\varphi],
    \end{align*}
    are all absolutely integrable with respect to $\dd \mu(z)\dd x \dd y$ in the analogous sense of \eqref{eq: integrandOfBIsAbsolutelyIntegrable}, and the resulting integral is bounded by
    \begin{align*}
        (2|E|)^{\frac{3}{2}}\Bigg[\bigg(\int_{E}\B_x[\phi,\phi]\,\dd x\bigg)^{\frac{1}{2}}+\bigg(\int_{E}\B_y[\psi,\psi]\,dy\bigg)^{\frac{1}{2}}\Bigg]\bigg(\int_{\R^d}\big(|z|^2\wedge 1\big)\,d\mu(z)\bigg)^{\frac{1}{2}}\cb\Big(L_\varphi\vee\norm{\varphi}{L^\infty}\Big)\nc,
    \end{align*}
where $E\subset\R^d$ is any compact set such that $\supp(\varphi)\subseteq E\times E$, $|E|$ denotes its $d$-dimensional \cb Lebesgue measure\nc, and $L_{\varphi}$ is a Lipschitz constant for $\varphi$.
\end{proposition}
\begin{proof}
    We only prove the $\B_{x+y}$-case, as the others can be dealt with similarly. Arguing as in part (ii) of the proof of Proposition \ref{prop: BCompatiblePairs}, we can restrict $(x,y)$ to the compact set $E\times E$, 
    and we compute  
    \begin{align*}
        &\,\int_{E\times E}\int_{\R^d}\big||\phi(x+z)-\psi(y+z)|-|\phi(x)-\psi(y)|\big|\big|\varphi(x+z,y+z)-\varphi(x,y)\big|\dd \mu(z)\dd x\dd y\\
        \leq&\,  \int_{E\times E}\int_{\R^d}\Big(|\phi(x+z)-\phi(x)| + |\psi(y+z)-\psi(y)|\Big)\Big((\sqrt{2}L_\varphi |z|)\wedge (2 \|\varphi\|_{L^\infty})\Big)\dd \mu(z)\dd x\dd y.
    \end{align*}
    The desired bound then follows from an appropriate application of  Hölder's inequality. 
\end{proof}

\section{Properties of entropy solutions}\label{sec: propertiesOfEntropySolutions}

In this section we establish some results and a priori estimates for entropy solutions of \eqref{E}. 

\subsection{Finite energy}
We will repeatedly need that entropy solutions have finite energy.

\begin{proposition}[Energy estimate]\label{prop: finiteEnergy}
Assume \eqref{Omegaassumption}--\eqref{muassumption} and $u$ is an entropy solution of \eqref{E}.
Then
\begin{align}
&\,\int_M \B\big[b(u)-b(\bu),b(u)-b(\bu)\big] \dd x\dd t\nonumber\\
    \leq &\,\int_{\Omega} H(u_0,\bu(0))\dd x - \int_{Q} \Big[(u-\bu) \bu_t + F (u,\bu)\cdot\nabla\bu\Big]b'(\bu)\dd x\dd t \label{eq: finiteEnergy}\\
   &\,+ \int_{Q} \L[b(\bu)](b(u)-b(\bu))\dd x\dd t,\nonumber
\end{align}
where $F(u,\bu)=\sgn(u-\bu)\big(f(u)-f(\bu)\big)$ and 
$
H(u,k):=\int_k^u \big(b(\xi)-b(k)\big)\dd \xi \geq 0.$ 
 Moreover, the right-hand side of \eqref{eq: finiteEnergy} is finite.
\end{proposition}

\begin{proof}
The family of entropy inequalities \eqref{eq: entropyInequalityVovelleMethod} are those corresponding to the semi-Kru\v{z}kov entropy pairs
\begin{align*}
    u\mapsto ((u-k)^\pm,F^\pm(u,k)).
\end{align*}
We shall now construct a second family of entropy inequalities corresponding to the entropy pairs $u\mapsto (H^\pm(u,k),G^\pm(u,k))$ defined by
\begin{equation}\label{eq: definitionOfEnergyEntropyPairs}
    \begin{split}
            H^+(u,k)=&\,\int_k^\infty(u-\xi)^+ b'(\xi)\dd \xi, \quad
             G^+(u,k)=\,\int_k^\infty F^+(u,\xi)b'(\xi)\dd \xi,\\
      H^-(u,k)=&\,\int_{-\infty}^k(u-\xi)^{-}b'(\xi)\dd \xi,\quad
      G^-(u,k)=\,\int_{-\infty}^kF^-(u,\xi)b'(\xi)\dd \xi.
    \end{split}
\end{equation}
We proceed by performing the calculations for the $(+)$ and $(-)$ case simultaneously, though we stress that one should consider them separate calculations.

Let $k\in\R$ and $0\leq \varphi\in C_c^\infty(M)$ be such that 
\begin{align}\label{eq: energyEstimatePositiveBoundaryCriterion}
    (b(u^c)-b(k))^\pm \varphi =0,
\end{align}
for all $(t,x)\in Q^c$. Observe that $\eqref{eq: energyEstimatePositiveBoundaryCriterion}$ is still valid when replacing $k$ with $\xi$ provided $\pm\xi\geq \pm k$. In particular, it follows that we, for all $\pm\xi\geq \pm k$, have \eqref{eq: entropyInequalityVovelleMethod} with $\xi$ replacing $k$, yielding
  \begin{equation}\label{eq: energyEstimateStepOne}
       \begin{split}
            &\,-\int_Q\big( (u- \xi)^\pm\varphi_t + F^\pm(u, \xi)\cdot\nabla\varphi\big) \dd x \dd t\\
    &\,-\int_Q \L^{\geq r}[b(u)]\sgn^\pm(u- \xi)\varphi \dd x \dd t\\
    &\,- \int_M(b(u)-b( \xi))^\pm\L^{<r}[\varphi]\dd x \dd t\\
    \leq &\, L_f\int_\Gamma (\overline{u}^c- \xi)^\pm\varphi \dd \sigma(x) \dd t,
       \end{split}
   \end{equation}
   where there is no `initial term' on the right-hand side as $\varphi(0,x)=0$.
   Multiplying \eqref{eq: energyEstimateStepOne} by the nonnegative quantity $b'(\xi)$ and integrating over $\xi\in(k,\infty)$ in the $(+)$ case and $\xi\in(-\infty,k)$ in the $(-)$ case we get
   \begin{equation}
       \label{eq: energyEstimateStepTwo}
       \begin{split}
            &\,-\int_Q \big(H^\pm(u,k)\varphi_t + G^\pm(u,k)\cdot\nabla\varphi\big) \dd x \dd t\\
    &\mp\int_M \L^{\geq r}[b(u)](b(u)- b(k))^\pm\varphi \dd x \dd t\\ &\,-\tfrac{1}{2}\int_M((b(u)-b( k))^\pm)^2\L^{<r}[\varphi]\dd x \dd t\\
    \leq &\, L_f\int_\Gamma H^\pm(\overline{u}^c,k)\varphi \dd \sigma(x) \dd t,
       \end{split}
   \end{equation}
where $H^{\pm}$ and $G^\pm$ are as in \eqref{eq: definitionOfEnergyEntropyPairs}. We have for later convenience extended the domain of integration for the second integral in \eqref{eq: energyEstimateStepTwo} from $Q$ to $M$; this is possible as $(b(u)-b(k))^\pm\varphi$ is zero a.e.~in $Q^c$. To summarize, \eqref{eq: energyEstimateStepTwo} is valid for all $k\in\R$ and $0\leq \varphi\in C^\infty_c(M)$ satisfying \eqref{eq: energyEstimatePositiveBoundaryCriterion}. 

Next, we wish to set $k=\bu(t,x)$ in \eqref{eq: energyEstimateStepTwo}, and while this cannot be done directly, we will accomplish this by a `doubling of variables' argument. For parameters $(s,y)\in M$ and $\e\in(0,1)$, we introduce the pair $(\k^\pm,\psi)\in \R\times C^\infty_c(M)$ defined as follows
\begin{equation}
\begin{split}
    \kappa^{\pm}= &\,\kappa^{\pm}_{s,y,\e}=\overline{u}^c(s,y)\pm\e L_{\overline{u}^c},\\ \psi= &\,\psi_{s,y,\e}(t,x)= \tilde\varphi(t,x)\rho_\e(t-s,x-y),
    \end{split}
\end{equation}
 where $ 0\leq \tilde \varphi\in C^\infty_c(M)$ is arbitrary, $\rho_\e$ is a standard mollifier with support in the centered ball of radius $\e$ in $\R\times\R^{d}$ and $L_{\bu}$ is a local Lipschitz constant for $\bu$ valid on $\{(t,x)\in M\colon \mathrm{dist}((t,x),\supp\tilde{\varphi})<1\}$. The support of $\rho_\e$ implies that $\psi$ is zero when $|(t-s,x-y)|\geq \e$, while $|(t-s,x-y)|< \e$ implies that $\pm \kappa^{\pm} \geq \pm \bu(t,x)$. In particular, we have
\begin{align*}
    (b(\bu(t,x))-b(\kappa^{\pm}_{s,y,\e}))^{\pm}\psi_{s,y,\e}(t,x) = 0,
\end{align*}
for all $t,s,x,y,\e$, and we may thus use the pair $(\k^\pm,\psi)$ as the constant and test function in \eqref{eq: energyEstimateStepTwo}. 
Doing so, and additionally writing $\varphi$ instead of $\tilde \varphi$ for notional simplicity, we get after integrating over $(s,y)\in M$ 
 \begin{equation}
       \label{eq: energyEstimateStepThree}
       \begin{split}
            &\,-\int_{Q_x\times M_y} \Big[H^\pm(u,\k^\pm)\varphi_t + G^\pm(u,\k^\pm)\cdot\nabla_x\varphi\Big]\rho_\e \\
    &\mp\int_{M_x\times M_y} \L^{\geq r}_x[b(u)](b(u)- b(\k^\pm))^\pm\varphi\rho_\e  \\
    &\,-\tfrac{1}{2}\int_{M_x\times M_y}((b(u)-b( \k^\pm))^\pm)^2\L^{<r}_x[\varphi\rho_\e]\\
    \leq &\,-\int_{Q_x\times M_y} \Big[H^\pm(u,\k^\pm) \partial_s\rho_\e + G^\pm(u,\k^\pm)\cdot\nabla_y\rho_\e\Big]\varphi.
       \end{split}
   \end{equation}
For brevity, we have here suppressed the differentials $\dd t\dd x \dd s\dd y$ from the integrals and added sub-indexes to the operators and the domains of integration in line with the notation from Section \ref{sec:AssumptionsConceptMain}. We also used the identities $\partial_t\rho_\e=-\partial_s\rho_\e$ and $\nabla_x\rho_\e=-\nabla_y\rho_\e$ and dropped the term on the right-hand side of \eqref{eq: energyEstimateStepTwo} since
\begin{equation*}
    H^\pm(\bu(t,x),\k^\pm)\varphi\rho_\e=0,
\end{equation*}
for all $t,x,s,y,\e$.

Next, we integrate by parts in $(s,y)$ on the right-hand side of \eqref{eq: energyEstimateStepThree}, take $\varepsilon$ so small that the boundary terms (in time) vanish, and add the integral of the quantity $\pm\L_y^{\geq r}[b(\k^\pm)](b(u)-b(\k^\pm))^\pm\varphi\rho_\e$ to each side:
\begin{equation}
       \label{eq: energyEstimateStepFour}
       \begin{split}
            &\,-\int_{Q_x\times M_y} \Big[H^\pm(u,\k^\pm)\varphi_t + G^\pm(u,\k^\pm)\cdot\nabla_x\varphi\Big]\rho_\e \\
    &\mp\int_{M_x\times M_y} \L^{\geq r}_{x+y}[b(u)-b(\k^\pm)](b(u)- b(\k^\pm))^\pm\varphi\rho_\e  \\
    &\,-\tfrac{1}{2}\int_{M_x\times M_y}((b(u)-b( \k^\pm))^\pm)^2\L^{<r}_x[\varphi\rho_\e]\\
    \leq &\, \mp \int_{Q_x\times M_y} \Big[(u-\k^\pm)^\pm \partial_s\k^\pm + F^\pm (u,\k^\pm)\cdot\nabla_y\k\Big]b'(\k^\pm)\varphi \rho_\e\\
    &\pm \int_{Q_x\times M_y} \L_y^{\geq r}[b(\k^\pm)](b(u)-b(\k^\pm))^\pm\varphi \rho_\e.
       \end{split}
   \end{equation}
With $\L_{x+y}$ as defined in \eqref{eq: theXPlussYOperators}, we here used that $\L^{\geq r}_{x+y}[w_1-w_2]=\L^{\geq r}_{x}[w_1]-\L^{\geq r}_{y}[w_2]$ when $w_1$ and $w_2$ are functions in $x$ and $y$ respectively, and we used the formulas
\begin{align*}
    \partial_s H^\pm (u,\k^\pm) =&\, \mp(u-\k^\pm)^\pm b'(\k^\pm)\partial_s\k^\pm,\\ \nabla_yG^\pm(u,\k^\pm)=&\,\mp F^\pm (u,\k^\pm)b'(\k^\pm)\nabla_y\k^\pm,
\end{align*}
which follow from \eqref{eq: definitionOfEnergyEntropyPairs}. Fixing $r'>r$ we split up the operator $\L^{\geq r}_{x+y} = \L^{\geq r'}_{x+y} + \L^{r'>\cdots\geq r}_{x+y}$.  Looking at the first inequality of Corollary \ref{cor:ConvexInequalityNonlocalOperators} for the choice $\eta(\phi)=\frac{1}{2}(\phi^{\pm})^2$, we get the pointwise inequality
\begin{align*}
    \pm\L^{r'\geq\dots\geq r}_{x+y}[b(u)-b(\k^\pm)](b(u)- b(\k^\pm))^\pm
    \leq &\, \tfrac{1}{2}\L^{r'> \cdots\geq  r}_{x+y}\Big[((b(u)-b(\k^\pm))^\pm)^2\Big].
\end{align*}
We thus conclude
\begin{align*}
    &\mp\int_{M_x\times M_y} \L^{\geq r}_{x+y}[b(u)-b(\k^\pm)](b(u)- b(\k^\pm))^\pm\varphi\rho_\e \\
    \geq &\mp\int_{M_x\times M_y} \L^{\geq r'}_{x+y}[b(u)-b(\k^\pm)](b(u)- b(\k^\pm))^\pm\varphi\rho_\e \\
    &\,-\tfrac{1}{2}\int_{M_x\times M_y}((b(u)-b(\k^\pm))^\pm)^2 \L^{r'>\cdots \geq r}_{x}[\varphi]\rho_\e,
\end{align*}
where we used the self-adjointness of $\L_{x+y}^{r'>\cdots \geq r}$ (Proposition \ref{prop: integrationByPartsFormula}) to move it over to $\varphi\rho_\e$ on which the operator satisfies $\L^{r'>\cdots \geq r}_{x+y}[\varphi\rho_\e]=\L^{r'>\cdots \geq r}_{x}[\varphi]\rho_\e$. Exploiting this in \eqref{eq: energyEstimateStepFour}, letting $r\to0$ and then renaming $r'$ as $r$ 
we get
\begin{equation}
       \label{eq: energyEstimateStepFive}
       \begin{split}
            &\,-\int_{Q_x\times M_y} \Big[H^\pm(u,\k^\pm)\varphi_t + G^\pm(u,\k^\pm)\cdot\nabla_x\varphi\Big]\rho_\e \\
    &\mp\int_{M_x\times M_y} \L^{\geq r}_{x+y}[b(u)-b(\k^\pm)](b(u)- b(\k^\pm))^\pm\varphi\rho_\e  \\
    &\,-\tfrac{1}{2}\int_{M_x\times M_y}((b(u)-b( \k^\pm))^\pm)^2\L^{<r}_x[\varphi]\rho_\e\\
    \leq &\,\mp \int_{Q_x\times M_y} \Big[(u-\k^\pm)^\pm \partial_s\k^\pm + F^\pm (u,\k^\pm)\cdot\nabla_y\k^\pm\Big]b'(\k^\pm)\varphi \rho_\e\\
    &\pm \int_{Q_x\times M_y} \L_y^{\geq r}[b(\k^\pm)](b(u)-b(\k^\pm))^\pm\varphi \rho_\e,
       \end{split}
   \end{equation}
which only differs from \eqref{eq: energyEstimateStepFour} in that $\rho_\e$ is outside of the operator $\L_x^{<r}$. 

We may now let $\e\to0$ and by the regularity of $\bu$, $b$, $H^\pm$ and $G^\pm$ we get through standard estimates
\begin{equation}
       \label{eq: energyEstimateStepSix}
       \begin{split}
            &\,-\int_{Q} \Big[H^\pm(u,\bu)\varphi_t + G^\pm(u,\bu)\cdot\nabla\varphi\Big] \\
    &\mp\int_{M} \L^{\geq r}[b(u)-b(\bu)](b(u)- b(\bu))^\pm\varphi  \\
    &\,-\tfrac{1}{2}\int_{M}((b(u)-b( \bu))^\pm)^2\L^{<r}[\varphi]\\
    \leq &\, \mp \int_{Q} \Big[(u-\bu)^\pm \bu_t + F^\pm (u,\bu)\cdot\nabla\bu\Big]b'(\bu)\varphi \\
    &\pm \int_{Q} \L^{\geq r}[b(\bu)](b(u)-b(\bu))^\pm\varphi,
       \end{split}
   \end{equation}
   where $\bu$ is here a function in $(t,x)$; now the only variables of integration.

   Set $\varphi(t,x)=\theta(t)\phi(x)$, where $0\leq \theta\in C^\infty_c((0,T))$ and $0\leq \phi\in C^\infty_c(\R^d)$, and where $\phi(x)=1$ whenever $\mathrm{dist}(x,\Omega)\leq r$ so that both $\nabla\phi$ and $\L^{<r}[\phi]$ are zero in $Q$. Then \eqref{eq: energyEstimateStepSix} collapses to
  \begin{equation}
       \label{eq: energyEstimateStepSeven}
       \begin{split}
            &\,-\int_{Q} H^\pm(u,\bu)\theta' \\
    &\mp\int_{M} \L^{\geq r}[b(u)-b(\bu)](b(u)- b(\bu))^\pm\theta \\
    \leq &\, \mp \int_{Q} \Big[(u-\bu)^\pm \bu_t + F^\pm (u,\bu)\cdot\nabla\bu\Big]b'(\bu)\theta \\
    &\pm \int_{Q} \L^{\geq r}[b(\bu)](b(u)-b(\bu))^\pm\theta,
       \end{split}
   \end{equation} 
 where we used that $b(u)- b(\bu)=0$ in $M\setminus Q$. 
Add the $(+)$ case to the $(-)$ case, and we further get
  \begin{equation}
       \label{eq: energyEstimateStepEight}
       \begin{split}
            &\,-\int_{Q} H(u,\bu)\theta' \\
    &\,+\int_{M} \B^{\geq r}[b(u)-b(\bu),b(u)- b(\bu)]\theta  \\
    \leq &\, -\int_{Q} \Big[(u-\bu) \bu_t + F (u,\bu)\cdot\nabla\bu\Big]b'(\bu)\theta \\
    &\,+ \int_{Q} \L^{\geq r}[b(\bu)](b(u)-b(\bu))\theta.
       \end{split}
   \end{equation} 
where $H=H^++H^-$, $F=F^++F^-$, and we shifted $-\L^{\geq r}$ over to $\B^{\geq r}$ using Proposition \ref{prop: integrationByPartsFormula}.

Finally, for $\epsilon>0$ let $\theta=\theta_\epsilon$ be such that $\theta_\epsilon(t)=1$ for $t\in[\epsilon,T-\epsilon]$ and $\norm{\theta'_\epsilon}{L^1(0,T)}= 2$. Letting $r,\epsilon\to0$ we may for the first term in \eqref{eq: energyEstimateStepEight} use the continuity of $u$ at $t=0$ (Lemma \ref{lem: timeContinuityAtZeroOfEntropySolution}) and the regularity of $H$ to conclude that 
\begin{align*}
    \liminf_{r,\epsilon\to0}\bigg(-\int_{Q} H(u,\bu)\theta_\epsilon'\bigg)\geq -\int_\Omega H(u_0,\bu(0)),
\end{align*}
where $\bu(0)=\bu(0,\cdot)$. For the second term, we apply Fatou's lemma to conclude
\begin{align*}
    \liminf_{r,\epsilon\to0}\int_{M} \B^{\geq r}[b(u)-b(\bu),b(u)- b(\bu)]\theta_\epsilon\geq \int_{M} \B[b(u)-b(\bu),b(u)- b(\bu)].
\end{align*}
By Corollary \ref{cor: L1LocControlOfLevyOnComposition}, the last term in \eqref{eq: energyEstimateStepEight} converges to its canonical limit, and we conclude as $r,\epsilon\to0$ 
\begin{equation}
       \label{eq: energyEstimateStepNine}
       \begin{split}
      &\, -\int_{\Omega}H(u_0,\bu(0)) \\
    &\,+ \int_{M} \B[b(u)-b(\bu),b(u)- b(\bu)]\\
    \leq &\,-\int_{Q} \Big[(u-\bu)\bu_t + F (u,\bu)\cdot\nabla\bu \Big]b'(\bu)\\
    &\,+ \int_{Q} \L[b(\bu)](b(u)-b(\bu)).
       \end{split}
   \end{equation}
The proposition is proved.
\end{proof}


\subsection{Regularity results near the boundary}\label{sec: behaviourOfEntropySolutionNearTheBoundary}

In Lemma \ref{lem: theMotherLemma} below we present a family of inequalities that follow from Definition \ref{def: entropySolutionVovelleMethod}. From this family, we shall derive appropriate boundary regularity of an entropy solution, which is vital for the uniqueness argument. The proof is similar to that of Proposition \ref{prop: finiteEnergy}, but here we perform coarser computations resulting in a less explicit inequality \eqref{eq: equationFromMotherLemma} compared to \eqref{eq: finiteEnergy}. This is due to low-regularity terms that are difficult to handle when `going to the diagonal' in the doubling of variables argument. Hence, these terms are replaced by upper bounds; while precise limits can be computed, we have instead chosen simpler and coarser calculations.  

\begin{remark}\label{rem: explainingHowTheProofOfTheMotherLemmaDiffersFromTheOneInMiVO}
The coming lemma is analogous to Lemma 4.2 in \cite{MiVo03}, but the two proofs differ in that the main steps are interchanged. In \cite{MiVo03} the (local analogue of the) below inequality is proved first for test functions compactly supported inside the domain, done through a doubling of variables argument. Afterwards, the result is extended up to the boundary using a boundary layer sequence satisfying both $\inf_{\d>0}\inf_{x\in\Omega}\Delta\z_\d(x)\geq 0$ and $\sup_{\d>0}\norm{\nabla\z_\d}{L^1(\Omega)}<\infty$. Both properties of $\z_\d$ are crucial for this extension argument, and it is not clear if an analogous boundary layer sequence always exists in our case (where $\L$ replaces $\Delta$). Thus, we instead prove the below inequality \textit{directly} for test functions support up to the boundary. This approach bypasses the need for a special boundary layer sequence, but a new difficulty arises: The admissibility condition \eqref{eq: admissibilityConditionOnConstantAndTestfunction} outside the domain hinders a straight forward doubling of variables argument, and the data $\bu$ must be `lifted' as done in the proof of Proposition \ref{eq: finiteEnergy}. This lifting leads to expressions like $\L[b(\bu \pm \e)]$, which we control by imposing regularity on $\bu$ and $b$ separately rather than just on their composition $b(\bu)$ as done in \cite{MiVo03}.
\end{remark}
\begin{lemma}\label{lem: theMotherLemma}
Assume \eqref{Omegaassumption}--\eqref{muassumption} and $u$ is an entropy solution of \eqref{E}.
Then for $0\leq \varphi\in C_c^\infty([0,T]\times\R^d)$ and $k\in\R$, 
\begin{equation}\label{eq: equationFromMotherLemma}
\begin{split}
       &\,-\int_{Q}\Big[(u-(\overline{u}^c\vee^\pm k))^{\pm} \varphi_t +F^{\pm}(u,\overline{u}^c\vee^\pm k)\cdot\nabla \varphi\Big] \dd x\dd t\\
       &\,+\int_{M} \B\big[(b(u)-b(\overline{u}^c\vee^\pm k))^\pm,\varphi\big] \dd x\dd t\\
    \leq&\, \int_{Q}\Big[(1+dL_f)L_{\overline{u}^c}+\big|\L\big[b((\bu\vee^\pm k))\big]\big|\Big]\varphi \dd x\dd t\\
    &\,+ \int_\Omega \Big[\norm{u_0}{L^\infty(\Omega)}+\norm{\overline{u}^c}{L^\infty(M)}\Big] \varphi(0,\cdot)\dd x,
    \end{split}
\end{equation} 
where $F(u,v)\coloneqq \sgn(u-v)\big(f(u)-f(v)\big)$. Moreover, the right-hand side is bounded by $C\norm{\varphi}{L^\infty}$ where $C$ is the finite constant given by
\begin{equation}\label{eq: constantBoundingMotherLemma}
    \begin{split}
          C\coloneqq&\,  |Q|\big(1+dL_f\big)L_{\bu} + |\Omega|\big(\norm{u_0}{L^\infty(\Omega)} +\norm{\bu}{L^\infty(M)}\big)\\
          &\,+  \Big(\sup_{k}\norm{\L[b(\bu\wedge k)]}{L^1(Q)}\Big)\vee\Big(\sup_{k}\norm{\L[b(\bu\vee k)]}{L^1(Q)}\Big).
    \end{split}
\end{equation}
\end{lemma}
\begin{remark}\label{rem: onWhyTheTermsInTheMotherLemmaAreWellDefined}
    The nonlocal terms in \eqref{eq: equationFromMotherLemma} and \eqref{eq: constantBoundingMotherLemma} are well-posed, but this is not obvious. The terms involving $\L$ will be discussed in 
    the proof below, so we now discuss 
    why $g^{\pm}_k(t,x)\coloneqq (b(u)-b(\bu\vee^{\pm }k))^\pm$ has compact support and finite energy (making the $\B$-pairing well-defined by Proposition \ref{prop: BCompatiblePairs} $\mathrm{(ii)}$). Compact support 
    follows  since 
    $0\leq g^{\pm}_k\leq |b(u)-b(\bu)|$ where the right-hand side is supported on $Q$.
    As for the energy, we can bound its 
    non-singular part 
    using the previous bound and Proposition \ref{prop: BCompatiblePairs} (iv):
    \begin{align*}
        \int_{M}\B^{\geq 1}[g^{\pm}_k,g^{\pm}_k]\,\dd x\dd t\leq 4|Q|\mu(\{|z|\geq r\})\norm{b(u)-b(\bu)}{L^\infty(M)}.
    \end{align*}
For the singular part we can, since $g^{\pm}_k$ is supported in $Q$, write 
$\int_{M}\B^{<1}[g^{\pm}_k,g^{\pm}_k]\dd x\dd t = \int_{U}\B^{<1}[g^{\pm}_k,g^{\pm}_k]\dd x\dd t$ 
for $U=\{(t,x)\in M\colon \mathrm{dist}(x,\Omega)<1\}$. Next, 
$b(\bu\vee^{\pm}k)=b(\bu)\vee^\pm b(k)$ by monotonicity of $b$, and so
\begin{equation*}
g^{\pm}_k=\Big(\big(b(u)-b(\bu)\big)\mp\big(b(k)-b(\bu)\big)^\pm\Big)^{\pm}.
\end{equation*}
And since $\B[\eta(\phi),\eta(\phi)]\leq \B[\phi,\phi]$ for 
$|\eta'|<1$, and 
$\B[\phi +\psi,\phi + \psi]\leq 2\B[\phi,\phi] + 2\B[\psi,\psi]$, we get from the previous identity that
    \begin{align*}
        \int_{M}\B^{<1}[g^{\pm}_k,g^{\pm}_k]\,\dd x\dd t\leq 2\int_{U}\B^{<1}[b(u)-b(\bu),b(u)-b(\bu)]\, \dd x\dd t + 2\int_{U}\B^{<1}[b(\bu),b(\bu)]\, \dd x\dd t,
    \end{align*}
where the right-hand side is finite by Proposition \ref{prop: finiteEnergy} and since $b(\bu)$ is bounded and locally Lipschitz.
\end{remark}

\begin{proof}[Proof of Lemma \ref{lem: theMotherLemma}]
We start from the definition of entropy solutions and the entropy inequality \eqref{eq: entropyInequalityVovelleMethod}. Seeking to replace $k$ by $\overline{u}^c\vee^\pm k$, we double the variables; the first part of this proof is similar to that of Proposition \ref{prop: finiteEnergy}. In particular we also here carry out the proof of the $(+)$ and $(-)$ case simultaneously.

For parameters $(s,y)\in M$, $k\in\R$ and $\e\in(0,1)$, we introduce the pair $(\kappa^{\pm},\psi)\in \R\times C^\infty_c(M)$ defined as follows 
\begin{equation}\label{eq: specialChoiceOfPhiAndK}
\begin{split}
    \kappa^{\pm}= &\,\kappa^{\pm}_{s,y,\e,k}=(\overline{u}^c(s,y)\pm\e L_{\overline{u}^c})\vee^\pm k,\\ \psi= &\,\psi_{s,y,\e}(t,x)= \varphi(t,x)\rho_\e(t-s,x-y),
    \end{split}
\end{equation}
where $ 0\leq  \varphi\in C^\infty_c(M)$ is arbitrary, $\rho_\e$ is a standard mollifier with support in the centered ball of radius $\e$ in $\R\times\R^{d}$ and $L_{\bu}$ is a Lipschitz constant for $\bu$ valid on $\{(t,x)\in M\colon \mathrm{dist}((t,x),\supp\tilde{\varphi})<1\}$. The support of $\rho_\e$ implies that $\psi$ is zero when $|(t-s,x-y)|\geq \e$, while $|(t-s,x-y)|< \e$ implies that $\pm \kappa^{\pm} \geq \pm \bu(t,x)$. In particular, we have
\begin{align}\label{eq: psiAndKappaQualifyAsTestfunctionAndConstant}
    (b(\bu(t,x))-b(\kappa^{\pm}))^{\pm}\psi(t,x) = 0,
\end{align}
for all $t,s,x,y,\e$. Thus, the pair $(\kappa^{\pm},\psi)$ qualifies for use in the entropy inequality \eqref{eq: entropyInequalityVovelleMethod} as the constant and test function (which are denoted $k,\varphi$ in said inequality, and should not to be confused with the same symbols used here). 
Inserting this pair in \eqref{eq: entropyInequalityVovelleMethod}, we get after integrating over $(s,y)\in M$ 
\begin{equation}\label{eq: firstStepAtProvingTheTraceResult}
    \begin{split}
       &\,-\int_{Q_x\times M_y}\Big[(u-\kappa^{\pm})^{\pm} \varphi_t +F^{\pm}(u,\kappa^{\pm})\cdot\nabla_x \varphi\Big]\rho_\e \\
    &\,-\int_{Q_x\times M_y} \L_{x}^{\geq r}[b(u)]\sgn(u-\kappa^{\pm})^{\pm}\varphi\rho_\e\\
    &\,- \int_{M_x\times M_y}(b(u)-b(\kappa^{\pm}))^{\pm}\L_x^{<r}[\varphi\rho_\e]\\
    \leq&\, 
    -\int_{Q_x\times M_y}\Big[(u-\kappa^{\pm})^{\pm}\partial_s\rho_\e + F^{\pm}(u,\kappa^{\pm})\cdot\nabla_y\rho_\e\Big]\varphi,
    \end{split}
\end{equation}
where we use the sub-index notation from Section \ref{sec:AssumptionsConceptMain}. We removed both the `initial term' and the `boundary term' from the right-hand-side of \eqref{eq: entropyInequalityVovelleMethod} as they are here both zero; the former because $\varphi=0$ for $t=0$, and the latter because of a similar argument that gave \eqref{eq: psiAndKappaQualifyAsTestfunctionAndConstant}. We also used the two identities $\partial_t\rho_\e=-\partial_s\rho_\e$ and $\nabla_x\rho_\e=-\nabla_y\rho_\e$. 

Integrating by parts in $(s,y)$ on the right-hand side of \eqref{eq: energyEstimateStepThree}, taking $\varepsilon$ so small that the boundary terms vanishes, and adding the integral of $\L_y^{\geq r}[b(\k^\pm)](b(u)-b(\k^\pm))^\pm\varphi\rho_\e$ over $Q_x\times M_y$ to each side we further get
\begin{equation}\label{eq: secondStepAtProvingTheTraceResult}
    \begin{split}
       &\,-\int_{Q_x\times M_y}\Big[(u-\kappa^{\pm})^{\pm} \varphi_t +F^{\pm}(u,\kappa^{\pm})\cdot\nabla_x \varphi\Big]\rho_\e \\
    &\,-\int_{Q_x\times M_y} \L_{x+y}^{\geq r}[b(u)-b(\kappa^{\pm})]\sgn(u-\kappa^{\pm})^{\pm}\varphi\rho_\e\\
    &\,-\int_{M_x\times M_y} (b(u)-b(\kappa^{\pm}))^{\pm}\L_x^{<r}[\varphi\rho_\e]\\
    \leq&\, 
    \int_{Q_x\times M_y}\Big[\partial_s\big((u-\kappa^{\pm})^\pm\big) + \mathrm{div}_y\big(F^{\pm}(u,\kappa^{\pm})\big)\Big]\varphi\rho_\e\\
    &\,  +\int_{Q_x\times M_y}\L_{y}^{\geq r}[b(\kappa^{\pm})]\sgn(u-\kappa^{\pm})^{\pm}\varphi\rho_\e.
    \end{split}
\end{equation}
With $\L_{x+y}$ as defined in \eqref{eq: theXPlussYOperators}, we here used that $\L^{\geq r}_{x+y}[w_1-w_2]=\L^{\geq r}_{x}[w_1]-\L^{\geq r}_{y}[w_2]$ when $w_1$ and $w_2$ are functions in $x$ and $y$ respectively, and we also integrated by parts in $(s,y)$ on the right-hand side. By Lipschitz continuity, it is clear that the weak derivatives on the right-hand side of \eqref{eq: secondStepAtProvingTheTraceResult} satisfy
 \begin{align*}
     |\partial_s\big((u-\kappa^{\pm})^\pm\big)| + |\mathrm{div}_y\big(F^{\pm}(u,\kappa^{\pm})\big)|\leq L_{\bu} + dL_fL_{\bu},
 \end{align*}
 since $L_{\bu}$ is a Lipschitz constant for $\k$ on $Q_y$. 

Observe next that the domain of integration for the second integral on the left-hand side of \eqref{eq: secondStepAtProvingTheTraceResult} may be extended from $Q_x\times M_y$ to $M_x\times M_y$ as $\sgn^{\pm}(u-\k)^\pm\rho_\e=0$ on $Q_x^{c}\times M_y$ (by the same argument used for \eqref{eq: psiAndKappaQualifyAsTestfunctionAndConstant}). Doing so, we may further infer
\begin{align*}
    \int_{M_x\times M_y} \L_{x+y}^{\geq r}[b(u)-b(\kappa^{\pm})]\sgn(u-\kappa^{\pm})^{\pm}\varphi\rho_\e \leq&\, \int_{M_x\times M_y} \L_{x+y}^{\geq r}[(b(u)-b(\kappa^{\pm}))^\pm]\varphi\rho_\e\\
    = &\,\int_{M_x\times M_y} (b(u)-b(\kappa^{\pm}))^\pm\L_{x}^{\geq r}[\varphi]\rho_\e,
\end{align*}
 where we used the second inequality of Corollary \ref{cor:ConvexInequalityNonlocalOperators}, the self-adjointness of $\L_{x+y}^{\geq r}$ (Proposition \ref{prop: integrationByPartsFormula}) and the identity $\L_{x+y}^{\geq r}[\varphi\rho_\e]=\rho_\e\L_{x}^{\geq r}[\varphi]$.  
Exploiting these bounds in \eqref{eq: secondStepAtProvingTheTraceResult}, we find that
\begin{equation}\label{eq: thirdStepAtProvingTheTraceResult}
    \begin{split}
       &\,-\int_{Q_x\times M_y}\Big[(u-\kappa^{\pm})^{\pm} \varphi_t +F^{\pm}(u,\kappa^{\pm})\cdot\nabla_x \varphi\Big]\rho_\e \\
    &\,-\int_{M_x\times M_y} (b(u)-b(\kappa^{\pm}))^\pm\L_{x}^{\geq r}[\varphi]\rho_\e\\
    &\,-\int_{Q_x\times M_y} (b(u)-b(\kappa^{\pm}))^{\pm}\L_x^{<r}[\varphi\rho_\e]\\
    \leq&\, 
    \int_{Q_x\times M_y}\Big[(1+dL_f)L_{\bu} + |\L_{y}^{\geq r}[b(\kappa^{\pm})]|\Big]\varphi\rho_\e.
    \end{split}
\end{equation}
Next, we let $r\to 0$. The limits on the left-hand side of \eqref{eq: thirdStepAtProvingTheTraceResult} are straight forward by Lemma \ref{lem: L^pBoundnessOfLvWhenVIsC2AndCompactSpport} and so we turn our attention to the right-hand side. It is easily verified that the function $h^{\pm}_{k,\e}(\cdot)\coloneqq b\big((\cdot\pm \e L_{\bu}) \vee^{\pm} k\big)$ satisfies
\begin{align}\label{eq: heightAndTVBoundOnHPrime}
\norm{(h^{\pm}_{k,\e})'}{L^\infty(\R)}\leq \norm{b'}{L^\infty(\R)} \quad \quad \text{and}\quad\quad |(h^{\pm}_{k,\e})'|_{TV(\R)}\leq |b'|_{TV(\R)} + \norm{b'}{L^\infty(\R)} 
\end{align}
(the global regularity of $b$ is due to Remark \ref{assumptionremark}), and that the definition \eqref{eq: specialChoiceOfPhiAndK} of $\kappa^{\pm}$ implies that
\begin{equation}\label{eq: rewritingKappaForTheTraceResult}
    h^{\pm}_{k,\e}(\bu)=b(\kappa^{\pm}).
\end{equation}
By Corollary \ref{cor: L1LocControlOfLevyOnComposition} we then get $\lim_{r\to0}\L_{y}^{\geq r}[b(\k^\pm)] = \L_{y}[b(\k^\pm)]$ in $L^1_{\mathrm{loc}}(M_y)$. Using this in \eqref{eq: thirdStepAtProvingTheTraceResult} yields
\begin{equation}\label{eq: fourthStepAtProvingTheTraceResult}
    \begin{split}
       &\,-\int_{Q_x\times M_y}\Big[(u-\kappa^{\pm})^{\pm} \varphi_t +F^{\pm}(u,\kappa^{\pm})\cdot\nabla_x \varphi\Big]\rho_\e \\
    &\,-\int_{M_x\times M_y} (b(u)-b(\kappa^{\pm}))^\pm\L_{x}[\varphi]\rho_\e\\
    \leq&\, 
    \int_{Q_x\times M_y}\Big[(1+dL_f)L_{\bu} + |\L_{y}[b(\kappa^{\pm})]|\Big]\varphi\rho_\e.
    \end{split}
\end{equation}
Finally, we wish to send $\e\to 0$. The left-hand can be dealt with using classical arguments (cf. e.g. \cite{CiJa11}) and so we again focus on the right-hand side. We shall first establish the limit
\begin{align}\label{eq: trulyVanishingSingularity}
\lim_{r\to0}\sup_{\e\in(0,1)} \int_{Q_x\times M_y}|\L_{y}^{<r}[b(\kappa^{\pm})]|\varphi\rho_\e = 0.
\end{align}
Letting $U_y$ denote the set of points in $M_y$ such that $\mathrm{dist}(y,\Omega)\leq 1$, we see that the support of $(s,y)\mapsto \rho_\e(t-s,x-y)$ lies in $U_y$ for all $(t,x)\in Q_x$ and $\e\in(0,1)$. Thus, for $\e\in(0,1)$
\begin{align*}
       \int_{Q_x\times M_y}|\L_{y}^{<r}[b(\kappa^{\pm})]|\varphi\rho_\e
         \leq &\, \norm{\varphi}{L^\infty(M)}\int_{U_y}|\L_{y}^{<r}[h^{\pm}_{k,\e}(\bu)]|,
\end{align*}
where we use the notation from \eqref{eq: rewritingKappaForTheTraceResult}. By the explicit bound \eqref{eq: explicitBoundOfLevyOnCompositionOnBoundedSets} of Corollary \ref{cor: L1LocControlOfLevyOnComposition}, the uniform bounds \eqref{eq: heightAndTVBoundOnHPrime}, and assumption \eqref{u^cassumption}, there is a finite constant $C'>0$ only depending on $U_y$ and $\bu$ such that
\begin{equation*}
\begin{split}
    \norm{\L_{y}^{<r}[h^{\pm}_{k,\e}(\bu)]}{L^1(U_y)}\leq&\, C'(\norm{b'}{L^\infty(\R)} +|b'|_{TV(\R)})
    \int_{|z|<r}(|z|^2\wedge 1)\dd\mu(z).
\end{split}
\end{equation*}
This term converges to $0$ as $r\to0$ by the dominated convergence theorem and \eqref{muassumption}.

Thus we attain \eqref{eq: trulyVanishingSingularity}, and exploiting it, we may compute
\begin{align*}
    \lim_{\e\to 0}\int_{Q_x\times M_y}|\L_{y}[b(\kappa^{\pm})]|\varphi\rho_\e = \lim_{r\to 0}\lim_{\e\to 0} \int_{Q_x\times M_y}|\L_{y}^{\geq r}[b(\kappa^{\pm})]|\varphi\rho_\e
    =  \int_{Q}|\L[b(\bu\vee^{\pm} k)]|\varphi \dd x\dd t,
\end{align*}
where $\bu$ on the right-most side is a function in $(t,x)$. Here we approximated $\L_y$ by $\L^{\geq r}_y$ uniformly in $\varepsilon$ by \eqref{eq: trulyVanishingSingularity}, used 
that $\L^{\geq r}_y$ is a zero order operator to pass to the limit in $\varepsilon$, and then concluded by the limit definition of $\L$. 
Sending $\e\to0$ in \eqref{eq: fourthStepAtProvingTheTraceResult}, we then find that 
\begin{equation}\label{eq: fifthStepAtProvingTheTraceResult}
    \begin{split}
       &\,-\int_{Q}\Big[(u-(\bu\vee^\pm k))^{\pm} \varphi_t +F^{\pm}(u,(\bu\vee^\pm k))\cdot\nabla \varphi\Big]\dd x\dd t \\
    &\,-\int_{M} (b(u)-b((\bu\vee^\pm k)))^\pm\L[\varphi]\dd x\dd t\\
    \leq&\, 
    \int_{Q}\Big[(1+dL_f)L_{\bu}+\big|\L\big[b((\bu\vee^\pm k))\big]\big|\Big]\varphi\dd x\dd t,
    \end{split}
\end{equation}
where $\bu$ now is a function in $(t,x)$. Finally, we want to extend the support of $\varphi$ to include $t=0$ and $t=T$ as we have so far assumed it to have compact support in $M=(0,T)\times \R^d$. In \eqref{eq: fifthStepAtProvingTheTraceResult} we make the substitution $\varphi\mapsto \theta_\epsilon\varphi$, where $(\theta_\epsilon)_{\epsilon>0}\subset C^\infty_c((0,T))$ is a family of nonnegative functions satisfying $\lim_{\epsilon\to0}\theta_\epsilon(t)=1$ for $t\in (0,T)$ and $\norm{\theta_\epsilon'}{L^1((0,T))}=2$ for all $\epsilon$ while the new $\varphi$ is like previous one. Exploiting the time continuity of $u$ at $t=0$, we get the inequality
\begin{align*}
    \lim_{\epsilon\to 0}\int_{Q}(u-(\bu\vee^\pm k))^{\pm}\varphi\theta_\epsilon' \dd x\dd t \leq&\,     \int_{
    \Omega}(u_0-(\bu(0,\cdot)\vee^\pm k))^{\pm}\varphi(0,\cdot)\dd x\\
    \leq&\, \int_{
    \Omega}\big(\norm{u_0}{L^\infty(\Omega)}+\norm{\bu}{L^\infty(M)}\big)\varphi(0,\cdot)\dd x.
\end{align*}
Thus, inserting for $\varphi$ in \eqref{eq: fifthStepAtProvingTheTraceResult} and letting $\epsilon\to0$ we get \eqref{eq: equationFromMotherLemma} after shifting the operator $\L$ to $\B$; this shift is justified by Proposition \ref{prop: integrationByPartsFormula}, Proposition \ref{prop: BCompatiblePairs}, and Remark \ref{rem: onWhyTheTermsInTheMotherLemmaAreWellDefined}. 

 Finally, that the nonlocal terms in \eqref{eq: constantBoundingMotherLemma} are finite (and well-defined) follows again by Corollary \ref{cor: L1LocControlOfLevyOnComposition} applied to $h_{k,0}^{\pm}(\bu)$ where $h_{k,0}^{\pm}$ is as in \eqref{eq: heightAndTVBoundOnHPrime}.
\end{proof}

\begin{corollary}\label{cor: measureRepresentationOfMotherLemma}
Under the assumptions of Lemma \ref{lem: theMotherLemma}, there exists a finite signed Borel measure $\nu_k^{\pm}$ on $[0,T]\times\R^d$, which integrates to zero\footnote{$\int_{[0,T]\times\R^d}\dd\nu_k^{\pm}=\nu_k^{\pm}([0,T]\times\R^d)=0$.} and is non-positive on $[0,T]\times(\R^d\setminus\Omega)$, and such that
\begin{equation}
\label{eq: measureRepresentationOfMotherLemma}
    \begin{split}
     \int_{[0,T]\times\R^d}\varphi\dd\nu_k^{\pm}=&\,-\int_{Q}\Big[(u-(\overline{u}^c\vee^\pm k))^{\pm} \varphi_t +F^{\pm}(u,\overline{u}^c\vee^\pm k)\cdot\nabla \varphi\Big] \dd x\dd t\\
       &\,+\int_{M} \B\big[(b(u)-b(\overline{u}^c\vee^\pm k))^\pm,\varphi\big] \dd x\dd t,
\end{split}
\end{equation}
for any $\varphi\in C^\infty_b([0,T]\times\R^d)$. Moreover, the total variation norm of $\nu_k^\pm$ admits the bound $\|\nu_k^\pm\|\leq 2C$, where $C$ is as in \eqref{eq: constantBoundingMotherLemma}.
\end{corollary}

\begin{proof}
For brevity, let $T_k^{\pm}\colon C^\infty_b([0,T]\times\R^d)\to \R$ denote the operator such that $T_k^{\pm}(\varphi)$ coincides with the right-hand side of \eqref{eq: measureRepresentationOfMotherLemma} or, equivalently, the left-hand side of \eqref{eq: equationFromMotherLemma}. Pick $0\leq\psi\in C_c^\infty(\R^d)$ such that $\psi=1$ for $|x|\leq 1$ and set $\psi_n(x)=\psi(x/n)$. Note that $\lim_{n\to\infty}T^{\pm}_k(\psi_n)=0$. For an arbitrary $ \varphi\in C^\infty_c([0,T]\times\R^d)$ satisfying $|\varphi|\leq 1$, we can exploit the linearity of $T^\pm_k$ to compute
\begin{align*}
    T^{\pm}_k(\varphi)= \lim_{n\to\infty}T^{\pm}_k(\varphi + \psi_n)\leq 2C.
\end{align*}
The inequality follows from applying \eqref{eq: equationFromMotherLemma} (possible as $\varphi+\psi_n$ is nonnegative for a sufficiently large $n$) followed by using $|\varphi + \psi_n|\leq 2$. 
The linearity of $T_k^\pm$ then implies $|T^\pm_k(\varphi)|\leq 2C\norm{\varphi}{L^\infty}$ for any $ \varphi\in C^\infty_c([0,T]\times\R^d)$ and so by the Riesz representation theorem (see e.g. \cite{Rud87}) there exists a signed Borel measure $\nu_k^\pm$, of total variation norm $\|\nu_k^\pm\|\leq 2C$, such that $T_k^\pm(\varphi)=\int_{[0,T]\times\R^d}\varphi\dd\nu_k^\pm$ for all $\varphi\in C^\infty_c([0,T]\times\R^d)$. That this relation can be extended to $\varphi\in C^\infty_b([0,T]\times \R^d)$ follows by inserting $\varphi\psi_n$, with $\psi_n$ as before, and letting $n\to\infty$; each side converges to the canonical limit.

As it is clear that $\nu_k^\pm$ integrates to zero, it remains to prove that it is non-positive on $[0,T]\times(\R^d\setminus\Omega)$. This can be seen as follows: Let $\z_\d$ denote a boundary layer sequence and observe that we for any $0\leq\varphi\in C^\infty_c([0,T]\times\R^d)$ have
\begin{align*}
    \int_{[0,T]\times(\R^d\setminus\Omega)}\varphi\dd\nu_k^\pm = \lim_{\d\to0}  \int_{[0,T]\times\R^d}(1-\z_\d)\varphi\dd\nu_k^\pm = \lim_{\d\to0} T^\pm_k\big((1-\z_\d)\varphi\big)\leq 0,
\end{align*}
where we used \eqref{eq: equationFromMotherLemma}.
\end{proof}

We now demonstrate three forms of boundary regularity which are consequences of the general identity \eqref{eq: measureRepresentationOfMotherLemma}. They will be referred to as \textit{boundary integrability}, \textit{weak trace} and \textit{boundary condition}. The two latter regularity properties have local analogs; in \cite{MiVo03} they are referred to as `weak normal trace' and `boundary condition' respectively. The former however, boundary integrability, has no local analog in neither \cite{MiVo03} nor \cite{MaPoTe02}, and is in this sense a truly nonlocal feature of the equation.

\subsubsection{Boundary integrability}
Consider \eqref{eq: measureRepresentationOfMotherLemma} when $k=-\norm{\bu}{L^\infty(M)}$ in the $(+)$ case and $k=\norm{\bu}{L^\infty(M)}$ in the $(-)$ case; if we add the two resulting identities and let $\nu$ denote the sum of the corresponding measures, we get
\begin{equation}
\label{eq: theMeasureForTheWeakTrace}
    \begin{split}
     \int_{[0,T]\times\R^d}\varphi\dd\nu=&\,-\int_{Q}\Big[|u-\overline{u}^c| \varphi_t +F(u,\overline{u}^c)\cdot\nabla \varphi\Big] \dd x\dd t\\
       &\,+\int_{M} \B\big[|b(u)-b(\overline{u}^c)|,\varphi\big] \dd x\dd t,
\end{split}
\end{equation}
for all $\varphi\in C^\infty_b([0,T]\times \R^d)$. Here $F\coloneqq F^++F^-$ so that
$$
F(u,\bu)=\sgn(u-\bu)(f(u)-f(\bu))
$$ 
and $\nu$ is a finite signed Borel measure which integrates to zero and whose total variation norm satisfies $\|\nu\|\leq 4C$ with $C$ as in \eqref{eq: constantBoundingMotherLemma}. This identity puts a restriction on how much $b(u)$ can differ from $b(\bu)$ close to $\Gamma$ as shown in the following proposition.

\begin{proposition}[Boundary integrability]\label{prop: boundaryIntegrability}
Assume \eqref{Omegaassumption}--\eqref{muassumption} and $u$ is an entropy solution of \eqref{E}.
Then 
\begin{equation}\label{eq: theEquationOfBoundaryIntegrability}
    \begin{split}
\,\int_{0}^T\int_{\substack{x\in\Omega\\ x+z\in 
            \Omega^c}}|b(u(t,x))-b(\bu(t,x))|\dd \mu(z)\dd x\dd t
            \leq&\,
    C,
    \end{split}
\end{equation}
where $C$ is finite and given by \eqref{eq: constantBoundingMotherLemma}.
\end{proposition}

\begin{proof}
Let $(\overline{\z}_{\d})_{\d>0}$ be an outer boundary layer sequence (cf. \eqref{exteriorBoundaryLayerSequence}). Replacing $\varphi$ by $\overline{\z}_\d$ in \eqref{eq: theMeasureForTheWeakTrace}, we get
\begin{equation}\label{eq: stepOneOfProvingBoundaryIntegrability}
\begin{split}
       \int_{M} \B\big[|b(u)-b(\overline{u}^c)|,\overline{\z}_\d\big] \dd x\dd t= \int_{[0,T]\times\R^d}\overline{\z}_\d\dd\nu\leq 2C,
    \end{split}
\end{equation}
where the inequality follows as the positive part of $\nu$ is bounded by $2C$. For brevity, we now set $g(t,x)\coloneqq |b(u(t,x))-b(\bu(t,x))|$ for $(t,x)\in M$. Writing out the $\B$ term, we find
\begin{align}\label{eq: explicitBOfGAndExteriorBoundaryLayerSequence}
    \int_{M} \B\big[g,\overline{\z}_\d\big] \dd x\dd t
    =&\, \int_{M}\int_{\R^d}\Big(g(t,x+z)-g(t,x)\Big)\Big(\overline{\z}_\d(x+z)-\overline{\z}_\d(x)\Big)\dd\mu(z)\dd x\dd t.
\end{align}
As $g$ has finite energy (it is dominated by that of $b(u)-b(\bu)$) and $\overline{\z}_\d\in C^\infty_c(\R^d)$, the integrand on the right-hand side is absolutely integrable, and so we need not worry about the order of integration. Exploiting that $\overline{\z}_\d=1$ in $\overline{\Omega}$ and that $g(t,x)=0$ for 
$x\in \Omega^c$, the integral in \eqref{eq: explicitBOfGAndExteriorBoundaryLayerSequence} reduces to 
\begin{align*}
    &\, \int_0^T\int_{\substack{x\in\Omega \\
    x+z\in 
    {\Omega}^c}}g(t,x)(1-\overline{\z}_\d(x+z))\dd\mu(z)\dd x\dd t + \int_0^T\int_{\substack{x+z\in\Omega \\
    x\in 
    {\Omega}^c}}g(t,x+z)(1-\overline{\z}_\d(x))\dd\mu(z)\dd x\dd t \\
    =&\,2\int_0^T\int_{\substack{x\in\Omega \\
    x+z\in \Omega^c}}g(t,x)(1-\overline{\z}_\d(x+z))\dd\mu(z)\dd x\dd t\to 2\int_0^T\int_{\substack{x\in\Omega \\
    x+z\in 
    {\Omega}^c}}g(t,x)(1-\chi_{\overline{\Omega}}(x+z))\dd\mu(z)\dd x\dd t,
\end{align*}
as $\d\to0$: The equality holds as we in the second integral can perform the change of variables $x\mapsto x-z$ followed by $z\mapsto-z$, and the limit is valid by monotonicity and Fubini's theorem. Inserting this limit on the left-hand side of \eqref{eq: stepOneOfProvingBoundaryIntegrability}, we get \eqref{eq: theEquationOfBoundaryIntegrability} save for the weight $1-\chi_{\overline\Omega}(x+z)$ in the integrand. On the domain of integration, $x\in\Omega$ and $x+z\in \Omega^c$, this weight coincides with $1-\chi_{\partial\Omega}(x+z)$. By \eqref{Omegaassumption}, we infer that $\chi_{\partial\Omega}(x)=0$ for a.e.~$x\in\R^d$, and so $x\mapsto\int_{\R^d}\chi_{\partial\Omega}(x+z)\dd\mu(z)$ is zero for a.e.~$x\in\R^d$; see (b) from Remark \ref{rem: defentsolnremark} (while $\mu$ is not finite here, we can use finite approximations $\mu_n(E)\coloneqq \mu(E\setminus \{0<|z|<n^{-1}\})$ to get the same conclusion by monotonicity). That is, the weight $(1-\chi_{\overline{\Omega}}(x+z))$ can safely be removed from the integral without changing its value. 
\end{proof}

A corollary of the previous proposition is the following $\B$-pairing property of $b(u)-b(\bu)$.

\begin{corollary}\label{cor: extraRegularityOfEntropySolution}
Under the assumptions of Proposition \ref{prop: boundaryIntegrability}, the integrand of $\B\big[b(u)-b(\overline{u}^c),\varphi\cha\big]$ is absolutely integrable with respect to $\dd\mu(z)\dd x\dd t$ for all $\varphi\in C^\infty([0,T]\times\R^d)$.
 \end{corollary}
 \begin{proof}
Similarly to the previous proof, we set $g(t,x)\coloneqq b(u(t,x))-b(\bu(t,x))$ for brevity, and compute
 \begin{align*}
     &\,\int_{0}^T\int_{\R^d\times\R^d}|g(t,x+z)-g(t,x)||(\varphi\cha)(t,x+z)-(\varphi\cha)(t,x)|\dd\mu(z)\dd x\dd t\\
     =&\,\int_{0}^T\int_{\substack{x\in\Omega\\x+z\in\Omega}}|g(t,x+z)-g(t,x)||\varphi(t,x+z)-\varphi(t,x)|\dd\mu(z)\dd x\dd t\\
     &\,+\int_{0}^T\int_{\substack{x\in\Omega\\x+z\in\Omega^c}}|g(t,x)||\varphi(t,x)|\dd\mu(z)\dd x\dd t\\
     &\,+\int_{0}^T\int_{\substack{x\in\Omega^c\\x+z\in\Omega}}|g(t,x+z)||\varphi(t,x+z)|\dd\mu(z)\dd x\dd t,
     \end{align*}
where we used that both $g$ and $\varphi\cha$ are supported in $Q$.
The first integral on the right-hand side is finite as $g$ is of finite energy (Proposition \ref{prop: finiteEnergy}) and $\varphi$ is smooth, while the two latter integrals (which by a change of variables can be seen to coincide) are finite due to Proposition \ref{prop: boundaryIntegrability}.
\end{proof}

\subsubsection{Weak trace}
We next establish a weak trace identity resulting from \eqref{eq: theMeasureForTheWeakTrace}.

\begin{proposition}[Weak trace]\label{prop: weakTrace}
 Assume \eqref{Omegaassumption}--\eqref{muassumption} and $u$ is an entropy solution of \eqref{E}.
 Let $\nu$ be the finite Borel measure from \eqref{eq: theMeasureForTheWeakTrace}. Then, for any boundary layer sequence $(\z_\d)_{\d>0}$, $r>0$, and $\varphi\in C^\infty_b([0,T]\times \R^d)$,  we have
\begin{equation}\label{eq: weakTrace}
\begin{split}
       &\,\lim_{\d\to0}\Bigg(\int_Q F(u,\overline{u}^c)\cdot (\nabla \z_\d) \varphi \dd x\dd t-\int_{M}\B^{<r}[ |b(u)-b(\overline{u}^c)|,\z_\d]\varphi \dd x \dd t\Bigg)\\
     = &\,\int_{\overline{\Gamma}}(1-\z)\varphi \dd \nu  - \int_M \B^{<r}\big[|b(u)-b(\overline{u}^c)|,\cha\big]\varphi \dd x \dd t,
\end{split}
\end{equation}
where $\overline{\Gamma}= [0,T]\times\partial\Omega$, $F(u,\bu)=\sgn(u-\bu)(f(u)-f(\bu))$, and $\z$ is the limit of $\z_\d$ on $\po$. The right-hand side of \eqref{eq: weakTrace} side is non-positive, and the expression inside the $\lim$ on the left-hand side 
is
bounded by $\tilde{C}\norm{\varphi}{C^1(M)}$ for a $\tilde{C}>0$ independent of $\d,\varphi$.
\end{proposition}

\begin{proof}[Proof of Proposition \ref{prop: weakTrace}]
Choosing $(1-\z_\d)\varphi$ as the test function in \eqref{eq: theMeasureForTheWeakTrace}, we get after a little rewriting
\begin{equation}\label{eq: stepOneForWeakTrace}
    \begin{split}
              &\,\int_{Q}F(u,\bu)\cdot(\nabla\z_\d)\varphi\dd x\dd t - \int_{M} \B^{<r}\big[|b(u)-b(\overline{u}^c)|,\z_\d\big] \varphi\dd x\dd t \\
       =&\,\int_{[0,T]\times \R^d}(1-\z_\d)\varphi \dd \nu +\int_{Q}\Big[|u-\overline{u}^c| \varphi_t +F(u,\overline{u}^c)\cdot\nabla \varphi\Big](1-\z_\d) \dd x\dd t  \\
    &\,-\int_{M} \B^{< r}\big[|b(u)-b(\overline{u}^c)|,\varphi\big](1-\z_\d) \dd x\dd t-\int_{M} \B^{\geq r}\big[|b(u)-b(\overline{u}^c)|,(1-\z_\d)\varphi\big] \dd x\dd t,
    \end{split}
\end{equation}
where we decomposed $\B=\B^{\geq r}+\B^{<r}$ for some fixed $r>0$ and used the product rule for $\B^{<r}$ (Proposition \ref{prop: productRuleForB}). Note that the very last part of the proposition follows from \eqref{eq: stepOneForWeakTrace}, the uniform bound $|\z_\d|\leq 1$, the total variation bound $\|\nu\|\leq 4C$ with $C$ as in \eqref{eq: constantBoundingMotherLemma}, the finite energy and the compact support of $|b(u)-b(\bu)|$; the integrals involving $\B^{<r}$ and $\B^{\geq r}$ on the right-hand side are bounded by $\mathrm{(ii)}$ and $\mathrm{(iv)}$ from Proposition \ref{prop: BCompatiblePairs}.

Letting $\d\to 0$ in \eqref{eq: stepOneForWeakTrace} we obtain
\begin{equation}\label{eq: stepTwoForWeakTrace}
    \begin{split}
            &\,\lim_{\d\to0}\Bigg(\int_Q F(u,\overline{u}^c)\cdot (\nabla \z_\d) \varphi \dd x\dd t-\int_{M}\B^{<r}\big[ |b(u)-b(\overline{u}^c)|,\z_\d\big]\varphi \dd x \dd t\Bigg)\\
    =&\, \int_{\overline{\Gamma}}(1-\z)\varphi \dd \nu + \int_{[0,T]\times(\R^d\setminus\overline{\Omega})}\varphi \dd \nu -\int_{M} \B^{< r}\big[|b(u)-b(\overline{u}^c)|,\varphi\big](1-\cha) \dd x\dd t\\
    &\,-\int_{M} \B^{\geq r}\big[|b(u)-b(\overline{u}^c)|,(1-\cha)\varphi\big] \dd x\dd t.
    \end{split}
\end{equation}
It remains to show that the right-hand side of \eqref{eq: stepTwoForWeakTrace} coincides with that of \eqref{eq: weakTrace}. Consider the special case of an outer boundary layer sequence $(\overline{\z}_\d)_{\d>0}$ (cf. \eqref{exteriorBoundaryLayerSequence}). Using that $\overline{\z}_\d=1$ on $\overline{\Omega}$, \eqref{eq: stepTwoForWeakTrace} reads
\begin{equation}\label{eq: stepThreeForWeakTrace}
    \begin{split}
            &\,-\lim_{\d\to0}\int_{M}\B^{<r}\big[ |b(u)-b(\overline{u}^c)|,\overline{\z}_\d\big]\varphi \dd x \dd t\\
    =&\, \int_{[0,T]\times(\R^d\setminus\overline{\Omega})}\varphi \dd \nu -\int_{M} \B^{< r}\big[|b(u)-b(\overline{u}^c)|,\varphi\big](1-\cha) \dd x\dd t\\
    &\,-\int_{M} \B^{\geq r}\big[|b(u)-b(\overline{u}^c)|,(1-\cha)\varphi\big] \dd x\dd t.
    \end{split}
\end{equation}
But in the case of an outer boundary layer sequence, the limit on the left-hand side of \eqref{eq: stepThreeForWeakTrace} may be computed explicitly giving
\begin{equation}\label{eq: explicitLimitOfBWhenWeUseExteriorBoundaryLayerSequence}
    -\lim_{\d\to 0}\int_M \B^{<r}\big[|b(u)-b(\bu)|,\overline{\z}_\d\big]\varphi\dd x\dd t = -\int_M \B^{<r}\big[|b(u)-b(\bu)|,\cha\big]\varphi\dd x\dd t.
\end{equation}
For brevity we skip this computation as it is almost identical to the proof of Proposition \ref{prop: boundaryIntegrability}. In conclusion, the three last terms on the right-hand side of \eqref{eq: stepTwoForWeakTrace} may be replaced by the right-hand side of \eqref{eq: explicitLimitOfBWhenWeUseExteriorBoundaryLayerSequence} and so we attain \eqref{eq: weakTrace}. Finally, that the right-hand side of \eqref{eq: weakTrace} is non-positive is a consequence of Corollary \ref{cor: measureRepresentationOfMotherLemma} and that $\B^{<r}[|b(u)-b(\bu)|,\cha]\geq 0$ (as can be seen by writing it out).
\end{proof}

\subsubsection{Boundary condition}
Recall that $F=F^++F^-$, and define the quantities
    \begin{equation}\label{eq: KruzkovBoundaryEntropy-EntropyFlux}
    \begin{split}
        \mathcal{F}= \mathcal{F}(u,\overline{u}^c,k) &\coloneqq F(u,\overline{u}^c)+F(u,k)-F(\overline{u}^c,k),\\
          \Sigma=\Sigma(u,\overline{u}^c,k) &\coloneqq |b(u)-b(\overline{u}^c)| + |b(u)-b(k)|-|b(\overline{u}^c)-b(k)|.
          \end{split}
    \end{equation}
As was observed in \cite{MiVo03}, these quantities may alternatively be written as
    \begin{equation}\label{eq: identitySatisfiedByKruzkovBoundaryEntropy-EntropyFlux}
    \begin{split}
\mathcal{F}/2 = &\, F^{+}(u,
       \overline{u}^c \vee^+ k)+F^{-}(u,
       \overline{u}^c \vee^- k),\\
  \Sigma/2 = &\,(b(u)-b(\overline{u}^c \vee^+ k))^{+}+(b(u)-b(\overline{u}^c \vee^- k))^{-},
          \end{split}
    \end{equation}
which is verified by considering the three cases $u< \bu\vee^{-}k$, $u> \bu\vee^{+}k$, and $u\in[\bu\vee^{-}k,\bu\vee^{+}k]$. In particular, we infer from \eqref{eq: identitySatisfiedByKruzkovBoundaryEntropy-EntropyFlux} that $|\mathcal{F}|\leq 2L_f|u-\bu|$ and $0\leq \Sigma\leq 2|b(u)-b(\bu)|$, and that $\Sigma$ has finite energy by Remark \ref{rem: onWhyTheTermsInTheMotherLemmaAreWellDefined}.

\begin{proposition}[Boundary condition]\label{prop: boundaryConditionEntropySolution}
Assume \eqref{Omegaassumption}--\eqref{muassumption} and $u$ is an entropy solution of \eqref{E}.
Then, for any inner boundary layer sequence $(\z_\d)_{\d>0}$, $k\in\R$, $r>0$, and $0\leq \varphi\in C^\infty_b([0,T]\times\R^d)$, 

\begin{align}\label{eq: theUsefulBoundaryCondition}
    \lim_{\d\to 0}&\Bigg( \int_Q \big(\mathcal{F}\cdot\nabla\z_\d\big)\varphi \dd x\dd t- \int_M \B^{<r}\big[\Sigma,\z_\d\big]\varphi \dd x\dd t\Bigg)\leq 0.
\end{align}
Moreover, the expression inside the $\lim$ 
is
bounded by $\tilde{C}\norm{\varphi}{C^1(M)}$ for a $\tilde{C}>0$ independent of $\d,k,\varphi$.
\end{proposition}

\begin{proof}
For a fixed $k$, the $(+)$ and $(-)$ case of \eqref{eq: measureRepresentationOfMotherLemma} results in two identities, each featuring their own measure $\nu_k^+$ and $\nu_k^-$. Setting $\nu_k\coloneqq\nu_k^++\nu_k^-$, and adding these two identities together, we get for any $0\leq \varphi\in C^\infty_c([0,T]\times\R^d)$
\begin{equation}\label{eq: stepOneInProvingTheBoundaryCondition}
\begin{split}
       \int_{[0,T]\times\R^d}\varphi\dd\nu_k=&\,-\int_{Q}\Big[(u-(\overline{u}^c\vee^+ k))^{+}+(u-(\overline{u}^c\vee^- k))^{-}\Big] \varphi_t \dd x\dd t\\
       &\,-\frac{1}{2}\int_{Q}\mathcal{F}\cdot\nabla \varphi \dd x\dd t+\frac{1}{2}\int_{M} \B[\Sigma,\varphi] \dd x\dd t,
    \end{split}
\end{equation}
where we used the formula \eqref{eq: identitySatisfiedByKruzkovBoundaryEntropy-EntropyFlux}. Next, we substitute $\varphi\mapsto (\overline{\z}_\d-\z_\d)\varphi$ in \eqref{eq: stepOneInProvingTheBoundaryCondition}, where the new $\varphi$ is as before, while $\z_\d$ and $\overline{\z}_\d$ represents an inner and outer boundary layer sequence respectively (cf. \eqref{interiorBoundaryLayerSequence} and \eqref{exteriorBoundaryLayerSequence}). By the product rules for $\nabla$ and $\B$ we find
\begin{equation*}
\begin{split}
      -\int_{Q}\mathcal{F}\cdot\nabla \big((\overline{\z}_\d-\z_\d)\varphi \big)\dd x\dd t =&\,\int_{Q}\big(\mathcal{F}\cdot\nabla \z_\d\big)\varphi\dd x\dd t - \int_{Q}\big(\mathcal{F}\cdot\nabla \varphi\big)(1-\z_\d)\dd x\dd t,\\
      \int_{M} \B\big[\Sigma,(\overline{\z}_\d-\z_\d)\varphi\big] \dd x\dd t =&\,\int_{M} \B\big[\Sigma,\overline{\z}_\d\big]\varphi \dd x\dd t-\int_{M} \B\big[\Sigma,\z_\d\big]\varphi \dd x\dd t+\int_{M} \B\big[\Sigma,\varphi\big](\overline{\z}_\d-\z_\d) \dd x\dd t,
    \end{split}
\end{equation*}
where we, for the first identity, used that $\overline{\z}_{\d}=1$ in $\overline{\Omega}$. Using these in \eqref{eq: stepOneInProvingTheBoundaryCondition} (after the substitution $\varphi\mapsto (\overline{\z}_\d-\z_\d)\varphi$) together with some rewriting yields
\begin{equation}\label{eq: stepTwoInProvingTheBoundaryCondition}
\begin{split}
&\,\int_{Q}\big(\mathcal{F}\cdot\nabla \z_\d\big)\varphi\dd x\dd t-\int_{M} \B^{<r}\big[\Sigma,\z_\d\big]\varphi \dd x\dd t\\
           =&\,2\int_{[0,T]\times\R^d}(\overline{\z}_\d-\z_\d)\varphi \dd\nu_k+ \int_{M} \B^{\geq r}\big[\Sigma,\z_\d\big]\varphi \dd x\dd t-\int_{M} \B\big[\Sigma,\overline\z_\d\big]\varphi \dd x\dd t\\
           &\,+ 2\int_{Q}\Big[(u-(\overline{u}^c\vee^+ k))^{+}+(u-(\overline{u}^c\vee^- k))^{-}\Big] (\overline{\z}_\d-\z_\d)\varphi_t \dd x\dd t\\
           &\,+ \int_{Q}\big(\mathcal{F}\cdot\nabla \varphi\big)(1-\z_\d)\dd x\dd t - \int_{M} \B\big[\Sigma,\varphi\big](\overline{\z}_\d-\z_\d) \dd x\dd t,
    \end{split}
\end{equation}
where we decomposed $\B[\Sigma,\z_\d]=\B^{\geq r}[\Sigma,\z_\d] + \B^{<r}[\Sigma,\z_\d]$ for some fixed $r>0$. We pause here to note that the last part of the proposition follows from \eqref{eq: stepTwoInProvingTheBoundaryCondition}. Indeed, one has the uniform bounds $|\z_\d|,|\overline{\z}_\d|\leq 1$, the total variation bound $\|\nu_k\|\leq 4C$ with $C$ as in \eqref{eq: constantBoundingMotherLemma}, the bounds $|\mathcal{F}|\leq 2L_f|u-\bu|$ and $0\leq \Sigma\leq 2|b(u)-b(\bu)|$, and finite energy of $\Sigma$ (as explained below \eqref{eq: identitySatisfiedByKruzkovBoundaryEntropy-EntropyFlux}). In particular, the two integrals involving $\B^{\geq r}[\Sigma,\z_\d]$ and $\B[\Sigma,\varphi]$ can be bounded by $\mathrm{(ii)}$ and $\mathrm{(iv)}$ from Proposition \ref{prop: BCompatiblePairs}, while the integral involving $\B[\Sigma,\overline{\z}_\d]$ is bounded by Corollary \ref{cor: extraRegularityOfEntropySolution}; this latter claim is not so obvious and so we explain it in detail: Since $\Sigma(t,x)=0$ for $x\in \Omega^c$ and $\overline{\z}_\d=1$ for $x\in\Omega$, we get that $2\B[\Sigma,\overline{\z}_\d]$ may be written
\begin{equation*}
\int_{\R^d}\Big(\Sigma(t,x+z)-\Sigma(t,x)\Big)\Big(\overline{\z}_\d(x+z)-\overline{\z}_\d(x)\Big)\dd\mu(z)=\begin{cases}
     \int_{x+z\in\Omega^c}\Sigma(t,x)(1-\overline{\z}_\d(x+z))\dd\mu(z),\text{ if } x\in\Omega,\\
      \int_{x+z\in\Omega}\Sigma(t,x+z)(1-\overline{\z}_\d(x))\dd\mu(z),\text{ if }x\in\Omega^c.
  \end{cases} 
\end{equation*}
Using $0\leq \Sigma \leq 2|b(u)-b(\bu)|$ and $0\leq (1-\overline{\z}_\d)\leq 1$ in this last expression, we obtain the bound
\begin{equation}\label{eq: theBoundOnSigmaPairedWithOuterBoundaryLayerSequence}
    0\leq \B[\Sigma,\overline{\z}_\d]\leq 2\B[|b(u)-b(\bu)|,\cha],
\end{equation}
for a.e.~$(t,x)\in M$. By Corollary \ref{cor: extraRegularityOfEntropySolution}, the right-hand side of \eqref{eq: theBoundOnSigmaPairedWithOuterBoundaryLayerSequence} is integrable over $M$.

Next, we let $\d\to 0$ in \eqref{eq: stepTwoInProvingTheBoundaryCondition}: We have the pointwise limit $(\overline{\z}_\d-\z_\d)\to\chi_{\partial\Omega}$ and the two $L^1(M)$-limits $\B^{\geq r}\big[\Sigma,\z_\d\big]\to \B^{\geq r}\big[\Sigma,\cha\big]$ and $\B\big[\Sigma,\overline{\z}_\d\big]\to \B\big[\Sigma,\cha\big]$ where the latter is a consequence of dominated convergence (we have \eqref{eq: theBoundOnSigmaPairedWithOuterBoundaryLayerSequence}) and the fact that $\cha$ and $\chi_{\overline{\Omega}}$ differ on a Lebesgue null set (see the discussion at the end of the proof of Proposition \ref{prop: boundaryIntegrability}). Thus, when $\d\to0$, we obtain from \eqref{eq: stepTwoInProvingTheBoundaryCondition}
\begin{equation*}
       \lim_{\d\to 0}\Bigg(\int_{Q}\big(\mathcal{F}\cdot\nabla \z_\d\big)\varphi \dd x\dd t-\int_{M} \B^{<r}\big[\Sigma,\z_\d\big]\varphi \dd x\dd t\Bigg) = 2\int_{\overline{\Gamma}}\varphi \dd\nu_k - \int_{M} \B^{<r}\big[\Sigma,\cha\big]\varphi \dd x\dd t,
\end{equation*}
where $\overline{\Gamma}\coloneqq [0,T]\times\partial\Omega$. The proposition is thus proved, because $\nu_k$ is non-positive on $\overline{\Gamma}$ (by Corollary \ref{cor: measureRepresentationOfMotherLemma}) and because $\B^{<r}[\Sigma,\cha]\geq 0$ (which is proved similarly as \eqref{eq: theBoundOnSigmaPairedWithOuterBoundaryLayerSequence}).
\end{proof}


\section{Uniqueness of entropy solutions}\label{sec: uniqueness}
This section is devoted to proving the uniqueness of entropy solutions of \eqref{E}. The proof deploys the classical Kruskov's `doubling of variables' device as developed for nonlocal problems in \cite{CiJa11,Ali07}, and follows a similar path as the one developed in the work of Otto \cite{Ott96}, Michel and Vovelle \cite{MiVo03} and Mascia et al. \cite{MaPoTe02}. Its length is due to the many non-trivial limits that need to be computed. To get a rough understanding of the coming proof, we give a formal overview of its steps:

For two entropy solutions $u,v$ with the same exterior data $\bu=\bv$ the proof starts off classically by doubling the variables, $u=u(t,x)$ and $v=v(s,y)$, followed by combining the two entropy inequalities for a nonnegative test function in the variables $(t,x,s,y)$ which is compactly supported in $Q\times Q$. The test function is chosen as a product of boundary layer sequences (for compact support) and a standard mollifier in the variables $t-s$ and $x-y$.

Just as in the classical hyperbolic case \cite{Kru70} the mollifier commutes with the symmetric operators $(\partial_t+\partial_s)$ and $(\nabla_x + \nabla_y)$, which is necessary for avoiding blow up when we later let the mollifier approximate $\d(t-s,x-y)$. In the degenerate parabolic-hyperbolic case, the $\Delta$-term poses some challenge at this step since $\Delta_x + \Delta_y$ does \textit{not} commute with the mollifier. This difficulty was overcome by Carrillo who in \cite{Car99} showed that the necessary `cross-terms' could be safely introduced into the entropy inequality, replacing this operator by the more appropriate $(\nabla_x+\nabla_y)\cdot(\nabla_x+\nabla_y)$. In our nonlocal case, a similar difficulty arises for $\L$, but only for its singular part (which has been subjected to a convex inequality \cite{CiJa11,Ali07}). Of course, this singular part vanishes as $r\to 0$, which is how the proof proceeds. The remaining `good' part of $\L$ is then written in the appropriate $(x,y)$-symmetric form using Proposition \ref{prop: CarrilloStyleCross-Terms} and partially shifted on to the test function to avoid blow up as $r\to 0$. 

We then let the support of the test function expand up to the boundary. Here we introduce a new $r>0$ to once again divide $\L$ into a singular and non-singular part. As the boundary layers converge to the characteristic function on $\Omega$, we avoid potential blow up of the convection term and the singular part of $\L$ by resorting to the boundary condition of Proposition \ref{prop: boundaryConditionEntropySolution}. Essentially, this condition lets us exchange cumbersome differences of $u$ and $v$ with more manageable ones, such as a difference of $u$ and $\bu$ which admits a weak trace by Proposition \ref{prop: weakTrace}. 

Finally, we `go to the diagonal' by letting the mollifier tend to $\d(t-s,x-y)$ in an appropriate sense. Here we again deploy the weak trace result of Proposition \ref{prop: weakTrace} to compute some of the limits and, conveniently, these new trace-terms cancel out the ones from before, which could not else be dropped due to having a `bad' sign. After the completion of this process, the desired conclusion is reached modulo some debris from the singular part of $\L$. But these undesirable terms are small by Proposition \ref{prop: boundaryIntegrability} (boundary integrability) and vanish when we, once more, let $r\to0$ completing the proof.

\begin{theorem}[Uniqueness]\label{thm: uniquenessEntropySolutions}
Assume \eqref{Omegaassumption}--\eqref{muassumption}. Let $u,v$ be entropy solutions of \eqref{E} in the sense of Definition \ref{def: entropySolutionVovelleMethod} with initial data $u_0,v_0$ and exterior data $u^c=v^c$. Then for a.e.~$t\in (0,T)$ we have
\begin{align*}
    \int_{\Omega}|u(t,x)-v(t,x)|\dd x \leq \int_{\Omega}|u_0(x)-v_0(x)|\dd x.
\end{align*}
In particular, if also $u_0=v_0$, then $u=v$ a.e.~in $(0,T)\times\R^d$.
\end{theorem}

\begin{remark}
We can pick the \textit{same} $C^2$-extension for both $u^c$ and $v^c$, i.e., $\overline{u}^c=\overline{v}^c$. This is unproblematic, as all the a priori calculations of Section \ref{sec: propertiesOfEntropySolutions} can be carried out for any smooth extension.
\end{remark}

\begin{proof}[Proof of Theorem \ref{thm: uniquenessEntropySolutions}]

\noindent\textbf{1)} \emph{Doubling the variables.} 
We begin with the standard approach of doubling the variables: Consider $u$ and $v$ functions in $(t,x)$ and $(s,y)$ respectively, and let $0\leq \varphi\in C_c^\infty(Q_x\times Q_y)$, where the sub-index notation is as introduced in Section \ref{sec:AssumptionsConceptMain}. Considering $(s,y)\in M_y$ fixed, we may in the entropy inequalities \eqref{eq: entropyInequalityVovelleMethod} for $u$ let the constant be replaced with $v(s,y)$ and the test function replaced with $(t,x)\mapsto \varphi(t,x,s,y)$. Combining the $(+)$ and $(-)$ case, we obtain 
\begin{equation}\label{eq: uniquenessEntropyInequlaityInX}
\begin{split}
         -&\,\int_{Q_x} |u-v|\partial_t\varphi + F(u,v)\cdot\nabla_x \varphi\\ 
         - &\,\int_{M_x}\L^{\geq r}_{x}[b(u)]\sgn(u-v)\varphi+|b(u)-b(v)|\L^{<r}_x[\varphi]\leq 0,
\end{split}
\end{equation}
where we for brevity suppress the differentials $\dd t\dd x$ and have added sub-indexes to clarify in which variables the operators act. As \eqref{eq: uniquenessEntropyInequlaityInX} holds for all $(s,y)\in M_y$, it can be integrated over this set, so that we obtain
\begin{equation}\label{eq: uniquenessEntropyInequlaityInXIntegrated}
\begin{split}
         -&\,\int_{Q_x\times Q_y} |u-v|\partial_t\varphi + F(u,v)\cdot\nabla_x \varphi\\ 
         - &\,\int_{M_x\times M_y}\L^{\geq r}_{x}[b(u)]\sgn(u-v)\varphi+|b(u)-b(v)|\L^{<r}_x[\varphi]\leq 0,
\end{split}
\end{equation}
where we for the first integral used that $\varphi$ is supported in $Q_x\times Q_y$. Swapping the role of $u$ and $v$ (and the role of $(t,x)$ and $(s,y)$) we get the analogous inequality
\begin{equation}\label{eq: uniquenessEntropyInequlaityInYIntegrated}
\begin{split}
         -&\,\int_{Q_x\times Q_y} |v-u|\partial_s\varphi + F(v,u)\cdot\nabla_y \varphi\\ 
         - &\,\int_{M_x\times M_y}\L^{\geq r}_{y}[b(v)]\sgn(v-u)\varphi+|b(v)-b(u)|\L^{<r}_y[\varphi]\leq 0.
\end{split}
\end{equation}
With $\L_{x+y}$ as defined in \eqref{eq: theXPlussYOperators}, we have $\L_{x+y}^{\geq r}[b(u)-b(v)]=\L^{\geq r}_x[b(u)]-\L^{\geq r}_y[b(v)]$, and so by adding \eqref{eq: uniquenessEntropyInequlaityInXIntegrated} to \eqref{eq: uniquenessEntropyInequlaityInYIntegrated} we obtain
\begin{equation}\label{eq: uniquenessSumOfIntegratedEntropyInequalities}
\begin{split}
         -&\,\int_{Q_x\times Q_y} |u-v|(\partial_t+\partial_s)\varphi + F(u,v)\cdot(\nabla_x+\nabla_y) \varphi\\ 
         - &\,\int_{M_x\times M_y}\L^{\geq r}_{x+y}[b(u)-b(v)]\sgn(u-v)\varphi+|b(u)-b(v)|(\L^{<r}_{x}+\L^{<r}_{y})[\varphi]\leq 0,
\end{split}
\end{equation}
where we used that $F(u,v)=F(v,u)$ and that $\sgn(u-v)=-\sgn(v-u)$. Exploiting the second inequality of Corollary \ref{cor:ConvexInequalityNonlocalOperators} with the self-adjointness of $\L^{\geq r}_{x+y}$ (Proposition \ref{prop: integrationByPartsFormula}) we see that
\begin{align*}
    \int_{M_x\times M_y}\L^{\geq r}_{x+y}\big[b(u)-b(v)\big]\sgn(u-v)\varphi\leq&\, \int_{M_x\times M_y}\L^{\geq r}_{x+y}\big[|b(u)-b(v)|\big]\varphi\\
    =&\, \int_{M_x\times M_y}|b(u)-b(v)|\L^{\geq r}_{x+y}[\varphi].
\end{align*}
Using this in \eqref{eq: uniquenessSumOfIntegratedEntropyInequalities} and further letting $r\to0$, we conclude
\begin{equation}\label{eq: uniquenessSumOfIntegratedEntropyInequalities++}
\begin{split}
         -\int_{Q_x\times Q_y} \Bigg(|u-v|(\partial_t+\partial_s)\varphi + F(u,v)\cdot(\nabla_x+\nabla_y) \varphi\Bigg)
         - \int_{M_x\times M_y}\Bigg(|b(u)-b(v)|\L_{x+y}[\varphi]\Bigg)\leq 0.
\end{split}
\end{equation}
\smallskip
\noindent\textbf{2)} \emph{Going to the boundary.} 
With a slight abuse of notation, we let $r>0$ represent a new arbitrary constant so that we may (again) split up the operator $\L_{x+y}= \L^{\geq r}_{x+y}+\L^{< r}_{x+y}$. By Proposition \ref{prop: integrationByPartsFormula} we may shift $\L_{x+y}^{<r}$ to $\B^{<r}_{x+y}$ so to get
\begin{align*}
   - \int_{M_x\times M_y}|b(u)-b(v)|\L^{<r}_{x+y}[\varphi] = \int_{M_x\times M_y} \B^{<r}_{x+y}\Big[|b(u)-b(v)|,\varphi\Big].
\end{align*}
The integral on the right-hand side is well-defined because of Proposition \ref{prop: absolutelyIntegrableBInTheDoubleVariableScenario} and the local energy of $b(u)$ and $b(v)$ (see Remark \ref{remark: energyEstimateImpliesLocalBoundedEnergyOfBU}). Using this relation in \eqref{eq: uniquenessSumOfIntegratedEntropyInequalities++} and rearranging, we get
\begin{equation}\label{eq: uniquenessSumOfIntegratedEntropyInequalitiesRearranged}
\begin{split}
         &\,-\int_{Q_x\times Q_y} |u-v|(\partial_t+\partial_s)\varphi - \int_{M_x\times M_y}|b(u)-b(v)|\L_{x+y}^{\geq r}[\varphi]\\ 
         \leq &\,\int_{Q_x\times Q_y} F(u,v)\cdot(\nabla_x+\nabla_y) \varphi -\int_{M_x\times M_y}\B^{<r}_{x+y}\Big[|b(u)-b(v)|,\varphi\Big].
\end{split}
\end{equation}
Next, we set
\begin{align*}
    \varphi(t,x,s,y)=\z_\d(x)\z_\d(y)\rho(t-s,x-y)\a(t),
\end{align*}
where $(\z_\d)_{\d>0}$ is an inner boundary layer sequence satisfying $\sup_{\d>0}\|\nabla\z_\d\|_{L^1(\Omega)}<\infty$ (cf. \eqref{eq: generalizedBoundaryLayerSequence}), and where both $\rho\in C_c^\infty(\R\times\R^d)$ and $\a\in C_c^\infty((0,T))$ are nonnegative. We assume that the support of $\rho$ is contained in a sufficiently small ball around zero so that $\varphi$ is indeed compactly supported in $Q_x\times Q_y$. For brevity, we shall write $\z_\d^x$, $\z_\d^{y}$, $\rho$, $\alpha$ to mean the expressions $\z_\d(x)$, $\z_\d(y)$, $\rho(t-s,x-y)$, $\a(t)$ respectively. Inserting for $\varphi$ in \eqref{eq: uniquenessSumOfIntegratedEntropyInequalitiesRearranged} we get
\begin{equation}\label{eq: uniquenessFirstInequality}
    \begin{split}
         &\,-\int_{Q_x\times Q_y} |u-v|\z_\d^x\z_\d^y\rho\a' - \int_{M_x\times M_y}|b(u)-b(v)|\L^{\geq r}_{x+y}\Big[\z_\d^x\z_\d^y\Big]\rho\a\\
    \leq&\, \int_{Q_x\times Q_y}F(u,v)\cdot(\nabla_x\z_\d^x)\z_\d^y \rho\a -\int_{M_x\times M_y}\B^{<r}_{x,x+y}\Big[|b(u)-b(v)|,\z_\d^x\z_\d^y\Big]\rho\a\\
    &\,+\int_{Q_x\times Q_y}F(u,v)\cdot(\nabla_y\z_\d^y)\z_\d^x\rho\a -\int_{M_x\times M_y}\B^{<r}_{y,x+y}\Big[|b(u)-b(v)|,\z_\d^x\z_\d^y\Big]\rho\a,
    \end{split} 
\end{equation}
where we used the cross-terms formula for $\B^{<r}_{x+y}$ from Proposition \ref{prop: CarrilloStyleCross-Terms} (the resulting integrals are, again,  well defined by Proposition \ref{prop: absolutelyIntegrableBInTheDoubleVariableScenario}), and moved $\rho$ outside the nonlocal operators since it is invariant under translations $(x,y)\mapsto(x+z,y+z)$.


We now wish to let $\d\to0$; observe that the left-hand side of \eqref{eq: uniquenessFirstInequality} has a well-defined limit as $\L_{x+y}^{\geq r}$ is a zero order operator. Thus, letting $\d\to 0$ we obtain
\begin{equation}\label{eq: uniquenessSecondInequality}
    \begin{split}
            &\,-\int_{Q_x\times Q_y} |u-v|\rho\a' - \int_{M_x\times M_y}|b(u)-b(v)|\L^{\geq r}_{x+y}\Big[\cha^x\cha^y\Big]\rho\a\leq \lim_{\d\to 0}(I_u + I_v),
    \end{split}
\end{equation}
where $\cha^x$ and $\cha^y$ denote the characteristic function $\cha$ as a function in $x$ and $y$ respectively and where $I_u$ and $I_v$ denote the first and second line on the right-hand side of \eqref{eq: uniquenessFirstInequality}. That $I_u+I_v$ admits a limit as $\d\to0$ will be proved by breaking it up into eighth appropriate terms. 

We begin by exploiting the product rule given by Lemma \ref{aSpecialProductRule} to see that $I_u$ may be rewritten
\begin{align*}
    I_u= &\,\int_{Q_x\times Q_y}F(u,v)\cdot(\nabla_x\z_\d^x) \z_\d^y\rho\alpha -\int_{M_x\times M_y}\B_x^{<r}\Big[|b(u)-b(v)|,\z_\d^x\Big]\z_\d^y\rho\alpha\\
     &\,-\int_{M_x\times M_y}|b(u)-b(v)|\Big(\B_y^{<r}\Big[\z_\d^y,\rho\Big]\z_\d^x-\B_{x,y}^{<r}\Big[\z_\d^x,\z_\d^y\Big]\rho\Big)\alpha,
\end{align*}
and by adding and subtracting terms we further obtain
\begin{equation}\label{eq: rewritingIX}
    \begin{split}
        I_u= &\,\int_{Q_x\times Q_y}\big(\F_u\cdot\nabla_x\z_\d^x\big) \z_\d^y\rho\alpha- \int_{M_x\times M_y}\B^{<r}_x\Big[\Sigma_u,\z_\d^x\Big]\z_\d^y\rho\alpha\\
    &\,-\int_{M_x\times M_y}|b(u)-b(v)|\Big(\B_y^{<r}\Big[\z_\d^y,\rho\Big]\z_\d^x-\B_{x,y}^{<r}\Big[\z_\d^x,\z_\d^y\Big]\rho\Big)\alpha\\
    &\,-\int_{Q_x\times Q_y}\big(F(u,\overline{u}^c)-F(v,\overline{u}^c)\big)\cdot\big(\nabla_x\z_\d^x\big) \z_\d^y\rho\alpha\\
    &\,+\int_{M_x\times M_y}\B^{<r}_x\Big[|b(u)-b(\overline{u}^c)|-|b(v)-b(\overline{u}^c)|,\z_\d^x\Big]\z_\d^y\rho\alpha,
    \end{split}
\end{equation}
where $\F_u = \F(u(t,x),v(s,y),\overline{u}^c(t,x))$, and $\Sigma_u = \Sigma(u(t,x),v(s,y),\overline{u}^c(t,x))$, and with $\Sigma,\F$ as in \eqref{eq: KruzkovBoundaryEntropy-EntropyFlux}. Analogously, we may write $I_v$ as
\begin{equation}\label{eq: rewritingIY}
    \begin{split}
        I_v= &\,\int_{Q_x\times Q_y}\big(\F_v\cdot\nabla_y\z_\d^y\big) \z_\d^x\rho\alpha- \int_{M_x\times M_y}\B^{<r}_y\Big[\Sigma_v,\z_\d^y\Big]\z_\d^x\rho\alpha\\
    &\,-\int_{M_x\times M_y}|b(u)-b(v)|\Big(\B_x^{<r}\Big[\z_\d^x,\rho\Big]\z_\d^y-\B_{x,y}^{<r}\Big[\z_\d^x,\z_\d^y\Big]\rho\Big)\alpha\\
    &\,-\int_{Q_x\times Q_y}\big(F(v,\overline{v}^c)-F(u,\overline{v}^c)\big)\cdot\big(\nabla_y\z_\d^y\big) \z_\d^x\rho\alpha\\
    &\,+\int_{M_x\times M_y}\B^{<r}_y\Big[|b(v)-b(\overline{v}^c)|-|b(u)-b(\overline{v}^c)|,\z_\d^y\Big]\z_\d^x\rho\alpha,
    \end{split}
\end{equation}
where $\F_v = \F(v(s,y),u(t,x),\overline{v}^c(s,y))$, and $\Sigma_v = \Sigma(v(s,y),u(t,x),\overline{v}^c(s,y))$. Note that, although $\overline{u}^c$ and $\overline{v}^c$ coincides as functions, they depend on the variables $(t,x)$ and $(s,y)$ respectively. Finally, by grouping terms from \eqref{eq: rewritingIX} and \eqref{eq: rewritingIY} we can further write
\begin{align}\label{eq: sumOfIXandIYisSumOfAllJs}
    I_u+I_v=J_u^0+J_v^0+J_u^1+J_v^1+J_u^2+J_v^2+J_u^3+J_v^3,
\end{align}
where we have introduced
\begin{align*}
    J_u^0&\coloneqq \int_{Q_x\times Q_y}\big(\F_u\cdot\nabla_x\z_\d^x\big) \z_\d^y\rho\alpha- \int_{M_x\times M_y}\B^{<r}_x\Big[\Sigma_u,\z_\d^x\Big]\z_\d^y\rho\alpha,\\
    J_v^0&\coloneqq \int_{Q_x\times Q_y}\big(\F_v\cdot\nabla_y\z_\d^y\big) \z_\d^x\rho\alpha- \int_{M_x\times M_y}\B^{<r}_y\Big[\Sigma_v,\z_\d^y\Big]\z_\d^x\rho\alpha,\\
    J_u^1&\coloneqq -\int_{Q_x\times Q_y}\big(F(u,\overline{u}^c)\cdot\nabla_x\z_\d^x\big) \z_\d^y\rho\alpha+ \int_{M_x\times M_y}\B^{<r}_x\Big[|b(u)-b(\overline{u}^c)|,\z_\d^x\Big]\z_\d^y\rho\alpha,\\
    J_v^1&\coloneqq -\int_{Q_x\times Q_y}\big(F(v,\overline{v}^c)\cdot\nabla_y\z_\d^y\big)\z_\d^x\rho\alpha + \int_{M_x\times M_y}\B^{<r}_y\Big[|b(v)-b(\overline{v}^c)|,\z_\d^y\Big]\z_\d^x\rho\alpha,\\
     J_u^2&\coloneqq \int_{Q_x\times Q_y}\big(F(u,\overline{v}^c)\cdot\nabla_y\z_\d^y\big)\z_\d^x\rho\alpha-\int_{M_x\times M_y}|b(u)-b(v)|\B_y^{<r}\Big[\z_\d^y,\rho\Big]\z_\d^x\alpha,\\
    J_v^2&\coloneqq \int_{Q_x\times Q_y}\big(F(v,\overline{u}^c)\cdot\nabla_x\z_\d^x\big)\z_\d^y\rho\alpha - \int_{M_x\times M_y}|b(u)-b(v)|\B_x^{<r}\Big[\z_\d^x,\rho\Big]\z_\d^y\alpha,\\
    J_u^3&\coloneqq \int_{M_x\times M_y}\Big(|b(u)-b(v)|\B_{x,y}^{<r}\Big[\z_\d^x,\z_\d^y\Big]-\B^{<r}_y\Big[|b(u)-b(\overline{v}^c)|,\z_\d^y\Big]\z_\d^x\Big)\rho\alpha,\\ 
    J_v^3&\coloneqq \int_{M_x\times M_y}\Big(|b(u)-b(v)|\B_{x,y}^{<r}\Big[\z_\d^x,\z_\d^y\Big]-\B^{<r}_x\Big[|b(v)-b(\overline{u}^c)|,\z_\d^x\Big]\z_\d^y\Big)\rho\alpha.
\end{align*}
The limit as $\d\to0$ of each of these terms will now be examined.

\smallskip
\noindent\textbf{3)} \emph{The limit of $J_u^0$ and $J_v^0$.} 
We will not compute the exact value of these limits (though they can be derived from the proof of Proposition \ref{prop: boundaryConditionEntropySolution}). We only demonstrate that they are non-positive which suffices for our purpose. Starting with $J_u^0$, we observe that it can be written
\begin{align*}
   J_u^0=\int_{Q_y} \Bigg[\int_{Q_x}\big(\F_u\cdot\nabla_x\z_\d^x\big)\rho\alpha- \int_{M_x}\B^{<r}_x\Big[\Sigma_u,\z_\d^x\Big]\rho\alpha\Bigg]\z_\d^y.
\end{align*}
By the boundary condition for $u$ (Proposition \ref{prop: boundaryConditionEntropySolution}), the expression inside the $\lim$ is uniformly bounded on $Q_y$ as $\d\to 0$ and it admits a nonnegative limit. Applying dominated convergence, we conclude that
\begin{align*}
    \lim_{\d\to0}J_u^0 \leq 0.
\end{align*}
An analogous calculation gives the corresponding result $\lim_{\d\to0}J_v^0\leq 0$.

\smallskip
\noindent\textbf{4)} \emph{The limit of $J_u^1$ and $J_v^1$.} 
Starting with $J_u^1$, we again do a little rewriting to see that 
\begin{align*}
   J_u^1= \int_{Q_y}\Bigg[\int_{M_x}\Big(-F(u,\overline{u}^c)\cdot\nabla_x\z_\d^x + \B^{<r}_x\Big[|b(u)-b(\overline{u}^c)|,\z_\d^x\Big]\Big)\rho\a \Bigg]\z_\d^y ,
\end{align*}
By the weak trace result (Proposition \ref{prop: weakTrace}), the expression inside the $\lim$ is uniformly bounded on $Q_y$ as $\d\to0$ and it admits a limit given by said proposition. Thus, dominated convergence gives us
\begin{align*}
\lim_{\d\to0} J^1_u= L_u^1\coloneqq &\,-\int_{\Gamma_x\times Q_y} \rho\a \dd \nu_u   + \int_{M_x\times Q_y} \B_x^{<r}\Big[|b(u)-b(\overline{u}^c)|,\cha^x\Big]\rho\a,
\end{align*}
where $\nu_u$ is a finite Borel measure on $\overline{\Gamma}_x$, but we integrate over $\Gamma_x$ as $\a$ is compactly supported in $(0,T)$. As before, we write $\cha^x$ to stress the $x$-dependence of the characteristic function.

Analogously, we also obtain 
\begin{align*}
\lim_{\d\to0} J_v^1= L_v^1\coloneqq  &\,-\int_{Q_x\times\Gamma_y} \rho\a \dd \nu_v   + \int_{Q_x\times M_y} \B_y^{<r}\Big[|b(v)-b(\overline{v}^c)|,\cha^y\Big]\rho\alpha.
\end{align*}

\smallskip 
\noindent\textbf{5)} \emph{The limit of $J_u^2$ and $J_v^2$.} As before, we start with $J_u^2$ and rewrite it as
\begin{align*}
 J_u^2 =&\,\int_{Q_x} \bigg(\int_{Q_y} \big(F(u,\overline{v}^c)\cdot\nabla_y\z_\d^y\big)\rho\a\bigg)\z_\d^x\\
 &\,-\int_{|z|<r}\Bigg( \int_{M_x\times M_y}|b(u)-b(v)|\big(\z_\d^y(y+z)-\z_\d^y(y)\big)\big(\rho(x-y-z)-\rho(x-y)\big)\a\z_\d^x\Bigg)\dd\mu(z),
\end{align*}
where we for the second integral used Fubini's theorem and wrote out only the arguments featuring translations in $z$.
For the first integral, we exploit the regularity of $\overline{v}^c$, $\rho$, $\a$ and $f$ to get
\begin{align*}
\lim_{\d\to0}\int_{Q_x} \bigg(\int_{Q_y} \big(F(u,\overline{v}^c)\cdot\nabla_y\z_\d^y\big)\rho\a\bigg)\z_\d^x = &\, -\int_{Q_x\times\Gamma_y} (F(u,\overline{v}^c)\cdot \hat{n})\rho\a,
\end{align*}
where $\hat{n}$ is an outward-pointing normal vector on $\po$. Here we used dominated convergence; this is justified as the inner integral is uniformly bounded over $Q_x$ and $\d>0$ (we recall that $\sup_{\d>0}\norm{\nabla\z_\d}{L^1(\Omega)}<\infty$). For the second part of $J_u^2$ we again wish to deploy dominated convergence; the inner integral has a canonical limit for fixed $z$, and for all $\d>0$ it is dominated by
\begin{align*}
    \norm{b(u)-b(v)}{L^\infty(M_x\times M_y)}\sup_{\d>0}\Big(\norm{\nabla\z_\d^y}{L^1(M_y)}|z|\wedge 2\norm{\z_\d^y}{L^1(M_y)}\Big)\Big(L_\rho|z|\wedge2\norm{\rho}{L^\infty(M_x\times M_y)}\Big)\norm{\a\z_\d^x}{L^1(M_x)},
\end{align*}
where $L_\rho$ is a Lipschitz constant for $\rho$ on $M_x\times M_y$. As this bound is less than $C(|z|^2\wedge 1)$, for an appropriate $C$, we infer from dominated convergence that the second part of $J_u^2$ admits the limit
\begin{align*}
    -\int_{|z|<r}\Bigg( \int_{M_x\times M_y}|b(u)-b(v)|\big(\cha^y(y+z)-\cha^y(y)\big)\big(\rho(x-y-z)-\rho(x-y)\big)\a\cha^x\Bigg)\dd\mu(z),
\end{align*}
as $\d\to0$. Fubini's theorem may again be applied to move the integral in $z$ inside, and so we conclude
\begin{align*}
   \lim_{\d\to0} J_u^2= L_u^2\coloneqq -\int_{Q_x\times \Gamma_y}\big(F(u,\overline{v}^c)\cdot \hat{n}\big)\rho\a - \int_{Q_x\times M_y}|b(u)-b(v)|\B^{<r}_{y}\Big[\cha^y,\rho\Big]\a,
\end{align*}
and by an analogous argument we also obtain
\begin{align*}
   \lim_{\d\to0} J_v^2= L_v^2\coloneqq -\int_{\Gamma_x\times Q_y}\big(F(v,\overline{u}^c)\cdot \hat{n}\big)\rho\a - \int_{M_x\times Q_y}|b(u)-b(v)|\B^{<r}_{x}\Big[\cha^x,\rho\Big]\a.
\end{align*}

\smallskip
\noindent\textbf{6)} \emph{The limit of $J_u^3$ and $J_v^3$.}
We begin by studying $J_u^3$. By Fubini's theorem, we rewrite $J_u^3$ as
\begin{align*}
   J_u^3 = &\,\int_{|z|<r}\Bigg(\int_{M_x\times M_y}\bigg[|b(u)-b(v)|\Big(\z_\d^x(x+z)-\z_\d^x(x)\Big)\Big(\z_\d^y(y+z)-\z_\d^y(y)\Big)\\ &\quad\quad\quad-\Big(|b(u)-b(\overline{v}^c(y+z))|-|b(u)-b(\overline{v}^c(y))|\Big)\Big(\z_\d^y(y+z)-\z_\d^y(y)\Big)\z_\d^x\bigg]\rho\a\Bigg) \dd\m(z),
\end{align*}
where only the arguments featuring translations in $z$ have been written out. For fixed $z$ it is obvious that the inner integral converges to its canonical limit as $\d\to0$. Moreover, the inner integral admits the following bound 
\begin{equation}\label{eq: boundOnInnerIntegralForJu3}
\begin{split}
     &|z|^2\Big(\norm{b(u)-b(v)}{L^\infty(M_x\times M_y)}\norm{\nabla\z_\d^x}{L^1(\R^d_x)}\norm{\nabla\z_\d^y}{L^1(\R^d_y)}\\
     &\quad\quad\quad+ L_{b(\overline{v}^c)}\norm{\nabla\z_\d^y}{L^1(\R^d_y)}\norm{\z_\d^x}{L^1(\R^d_x)}\Big)\norm{\rho\a}{L^\infty(M_x\times M_y)} \leq C|z|^2,
\end{split}
\end{equation}
where $L_{b(\overline{v}^c)}$ is a local Lipschitz constant of $b(\overline{v}^c)$ and where $C$ is some large constant independent of $\d>0$ as $\sup_\d\norm{\nabla\z_\d}{L^1(\R^d)}<\infty$ and $\sup_\d\norm{\z_\d}{L^1(\R^d)}\leq |\Omega|$. As this last bound is integrable on $|z|<r$ with respect to $\dd \mu$, it follows by dominated convergence that 
\begin{align*}
   \lim_{\d\to0} J_u^3 = L_u^3\coloneqq \int_{M_x\times M_y}\Big(|b(u)-b(v)|\B_{x,y}^{<r}\Big[\cha^x,\cha^y\Big]-\B^{<r}_y\Big[|b(u)-b(\overline{v}^c)|,\cha^y\Big]\cha^x\Big)\rho\a.
\end{align*}
And by an analogous argument, we also obtain
\begin{align*}
   \lim_{\d\to0} J_v^3 = L_v^3\coloneqq \int_{M_x\times M_y}\Big(|b(u)-b(v)|\B_{x,y}^{<r}\Big[\cha^x,\cha^y\Big]-\B^{<r}_x\Big[|b(v)-b(\overline{u}^c)|,\cha^x\Big]\cha^y\Big)\rho\a.
\end{align*}

\smallskip
\noindent\textbf{7)} \emph{Going to the diagonal.}
In summary, we have so far established the following inequality
\begin{equation}\label{eq: uniquenessThirdInequality}
    \begin{split}
            &\,-\int_{Q_x\times Q_y} |u-v|\rho\a' - \int_{M_x\times M_y}|b(u)-b(v)|\L^{\geq r}_{x+y}[\cha^x\cha^y]\rho\a\\
            \leq&\, L_u^1+L_v^1+L_u^2+L_v^2+L_u^3+L_v^3,
    \end{split}
\end{equation}
which follows from \eqref{eq: uniquenessSecondInequality} and \eqref{eq: sumOfIXandIYisSumOfAllJs} together with the above established limits. Observe that we have chosen to omit $L_u^0$ and $L_v^0$ from the right-hand side as these terms are non-positive. We now wish to insert a more explicit expression for $\rho$. For two small fixed parameters $\epsilon,\e>0$ we set
\begin{align*}
    \rho(t-s,x-y)=\theta_\epsilon(t-s)\tilde{\rho}_\e(x-y),
\end{align*}
where $\theta_\epsilon$ and $\rho_\e$ are standard mollifiers in $\R$ and $\R^d$ respectively. More precisely, we have $\theta_\epsilon=\epsilon^{-1}\theta(\cdot/\epsilon)$ and $\tilde{\rho}_\e=\e^{-d}\tilde{\rho}(\cdot/\e)$, where $\theta$ and $\tilde{\rho}$ are positive, smooth and symmetric functions supported in the unit ball of $\R$ and $\R^d$ that additionally satisfy $\int_\R\theta \dd t =\int_{\R^d}\tilde{\rho} \dd x =1$. For notational simplicity, we shall write $\rho$ instead of $\tilde{\rho}$, and the abbreviations $\theta_\epsilon$ and $\rho_\e$ to mean the expressions $\theta_\epsilon(t-s)$ and $\rho_\e(x-y)$. If we on the left-hand side of \eqref{eq: uniquenessThirdInequality} replace the (old) $\rho$ with the new expression $\theta_\epsilon\rho_\e$ we get 
\begin{align}\label{eq: theLefthandSideOfTheThirdInequality}
    -\int_{Q_x\times Q_y} |u(t,x)-v(s,y)|\theta_\epsilon\rho_\e\a' - \int_{M_x\times M_y}|b(u(t,x))-b(v(s,y))|\L^{\geq r}_{x+y}[\cha^x\cha^y]\theta_\epsilon\rho_\e\a,
\end{align}
where the arguments of $u$ and $v$ have been included to clarify the next computation.
We wish to \textit{go to the diagonal}, i.e.~to let $\e,\epsilon\to0$. This task is fairly straightforward for the above expression. Indeed, subtracting
\begin{align}\label{eq: theLimitOfTheLeftHandSideOfTheThirdInequality}
    -\int_{Q_x\times Q_y}|u(t,x)-v(t,x)|\theta_\epsilon\rho_\e\a' - \int_{M_{x}\times M_y}|b(u(t,x))-b(v(t,x))|\L_{x}^{\geq r}[\cha^x]\theta_\epsilon\rho_\e\a,
\end{align}
from \eqref{eq: theLefthandSideOfTheThirdInequality} and making the substitution of variables $s=t + \tau$ and $y=x+\xi$, a (mostly) standard calculation shows that the result is bounded by
\begin{align*}
    &\sup_{\substack{|\tau|<\epsilon\\
    |\xi|<\e}}\norm{v(\cdot +\tau,\cdot+\xi)-v}{L^1(Q_y)}\norm{\a'}{L^\infty((0,T))}\\ 
   + &\sup_{\substack{|\tau|<\epsilon\\
    |\xi|<\e}}\norm{b(v)(\cdot +\tau,\cdot+\xi)-b(v)}{L^1(Q_y)}2C_r\norm{\cha^x}{L^\infty(M_x)}\norm{\alpha}{L^\infty(M_x)}\\
   + &\norm{b(u)-b(v)}{L^\infty(M_x\times M_y)}2C_r\sup_{\substack{
    |\xi|<\e}}\norm{\cha^x(\cdot)(\cha^x(\cdot+\xi)-1)}{L^1(M_x)}\norm{\alpha}{L^\infty((0,T))},
\end{align*}
where $C_r=\mu(\{|z|\geq r\})$. By translation regularity of $L_{\textup{loc}}^1$-functions, this last bound tends to zero as $\e,\epsilon\to0$. In particular, the limit of \eqref{eq: theLefthandSideOfTheThirdInequality} as $\e,\epsilon\to 0$ coincides with that of \eqref{eq: theLimitOfTheLeftHandSideOfTheThirdInequality} which is given by
\begin{align}\label{eq: cleanerFormOfLimitOfLefthandSideOfThirdInequality}
    -\int_{Q} |u(t,x)-v(t,x)|\a'(t)\dd x\dd t - \int_{M}|b(u(t,x))-b(v(t,x))|\L^{\geq r}[\cha](x)\a(t)\dd x\dd t.
\end{align}
Observe next that the second integral in \eqref{eq: cleanerFormOfLimitOfLefthandSideOfThirdInequality} is non-positive. This follows as $|b(u(t,x))-b(v(t,x))|$ is nonnegative and supported in $Q$ (since $u^c=v^c$ a.e.~in $Q^c$), while $\L^{\geq r}[\cha]$ is non-positive in $Q$ due to $\cha$ taking its maximum in $Q$. We conclude then from \eqref{eq: uniquenessThirdInequality} and the above analysis that
\begin{equation}\label{eq: uniquenessFourthInequality}
    \begin{split}
            &\,-\int_{Q} |u(t,x)-v(t,x)|\a'(t)\dd x\dd t\leq  \lim_{r\to0}\limsup_{\e\to 0}\lim_{\epsilon\to0}\bigg[L_u^1+L_v^1+L_u^2+L_v^2+L_u^3+L_v^3\bigg].
    \end{split}
\end{equation}
where the choice of taking $\limsup$ in $\e$ is for brevity sake: While the intermediate $\e$-limit does in fact exist for all the coming terms, several of these limits further vanish as $r\to0$. Thus, it will often be sufficient (and much simpler) to instead obtain appropriate bounds on $\limsup_{\e\to0}$ for these terms. 

The remainder of the proof is to show that the right-hand side of \eqref{eq: uniquenessFourthInequality} is non-positive.
 For clarity, we restate the definitions of the right-hand-side terms, with the old $\rho$ replaced by $\theta_\epsilon\rho_\e$. They are
\begin{align*}
L_u^1= &\,-\int_{\Gamma_x\times Q_y} \theta_\epsilon\rho_\e\a \dd \nu_u   + \int_{M_x\times Q_y} \B_x^{<r}\Big[|b(u)-b(\overline{u}^c)|,\cha^x\Big]\theta_\epsilon\rho_\e\a,\\
L_v^1= &\,-\int_{Q_x\times \Gamma_y} \theta_\epsilon\rho_\e \a\dd \nu_v   + \int_{Q_x\times M_y} \B_y^{<r}\Big[|b(v)-b(\overline{v}^c)|,\cha^y\Big]\theta_\epsilon\rho_\e\a,\\
L_u^2=&\,-\int_{Q_x\times \Gamma_y}\big(F(u,\overline{v}^c)\cdot \hat{n}\big)\theta_\epsilon\rho_\e\a - \int_{Q_x\times M_y}|b(u)-b(v)|\B^{<r}_{y}\Big[\cha^y,\rho_\e\Big]\theta_\epsilon\a,\\
 L_v^2=&\,-\int_{\Gamma_x\times Q_y}\big(F(v,\overline{u}^c)\cdot \hat{n}\big)\theta_\epsilon\rho_\e\a - \int_{M_x\times Q_y}|b(u)-b(v)|\B^{<r}_{x}\Big[\cha^x,\rho_\e\Big]\theta_\epsilon\a,\\
L_u^3= &\,\int_{M_x\times M_y}\Big(|b(u)-b(v)|\B_{x,y}^{<r}\Big[\cha^x,\cha^y\Big]-\B^{<r}_y\Big[|b(u)-b(\overline{v}^c)|,\cha^y\Big]\cha^x\Big)\theta_\epsilon\rho_\e\a,\\
    L_v^3= &\,\int_{M_x\times M_y}\Big(|b(u)-b(v)|\B_{x,y}^{<r}\Big[\cha^x,\cha^y\Big]-\B^{<r}_x\Big[|b(v)-b(\overline{u}^c)|,\cha^x\Big]\cha^y\Big)\theta_\epsilon\rho_\e\a.
\end{align*}
Because the right-hand side of \eqref{eq: uniquenessFourthInequality} is bounded by the sum of the term-wise limits, that is
\begin{equation}\label{eq: justifyingTheIndividualLimitsOfTheLTerms}
    \lim_{r\to0}\limsup_{\e\to 0}\lim_{\epsilon\to0}\bigg[L_u^1+L_v^1+L_u^2+L_v^2+L_u^3+L_v^3\bigg]\leq \sum_{\substack{k=1,2,3\\
    w=u,v}}\lim_{r\to0}\limsup_{\e\to 0}\lim_{\epsilon\to0}L_w^k,
\end{equation}
we proceed by studying the limit of each $L$-terms, as was similarly done for the $J$-terms. 

\smallskip
\noindent\textbf{8)} \emph{The limit of $L_u^1$ and $L_v^1$.}
Starting with $L_u^1$, we introduce
\begin{align}\label{eq: definitionOfVarrho}
     \varrho_\e(x) &\coloneqq \int_{\Omega}\rho_\e(x-y)\dd y.
\end{align}
Note that $\varrho_\e$ is smooth, has compact support, and is bounded between $0$ and $1$ on $\R^d$. Moreover, it satisfies
\begin{align}\label{eq: varrhoTendsToAHalfAtTheBoundary}
    \lim_{\e\to0}\varrho_\e(x)=\begin{cases}
    1,\quad x\in\Omega,\\
    \tfrac{1}{2},\quad x\in\po,\\
    0,\quad x\in\Omega^c\setminus \po,
    \end{cases}
\end{align}
by the regularity of the boundary. With this notation, we let $\epsilon\to0$ and find that $L_u^1$ can be written
\begin{align}\label{eq: cleanerWayToWriteLu1}
    \lim_{\epsilon\to0}L_u^1=-\int_{\Gamma_x}\varrho_\e \a\dd \nu_u   + \int_{M_x} \B_x^{<r}\Big[|b(u)-b(\overline{u}^c)|,\cha^x\Big]\varrho_\e\a.
\end{align}
where we used that $\a$ has compact support in $(0,T)$  and that 
$\lim_{\epsilon\to0}\int_0^T\theta_\epsilon(t-s)\dd s =1$ for all $t\in(0,T)$.  By Corollary \ref{cor: extraRegularityOfEntropySolution}, the expression $\B^{<r}_x\big[|b(u)-b(\overline{u}^c)|,\cha^x\big]$ is well defined in $L^1(M_x)$ and must consequently vanish (in $L^1$ sense) as $r\to 0$.  We conclude from \eqref{eq: varrhoTendsToAHalfAtTheBoundary}, \eqref{eq: cleanerWayToWriteLu1} and dominated convergence that 
\begin{align*}
    \lim_{r\to0}\limsup_{\e\to0}\lim_{\epsilon\to0} L_u^1 = -\frac{1}{2}\int_{\Gamma}\a \dd \nu_u.
\end{align*}
For $L_v^1$ the analysis is analogous, except for the small difference that $\a$ goes from being a function in $t$ to one in $s$ as $\epsilon\to 0$. This leads to the analogous conclusion
\begin{align*}
    \lim_{r\to0}\limsup_{\e\to0}\lim_{\epsilon\to0} L_v^1 = -\frac{1}{2}\int_{\Gamma}\a \dd \nu_v.
\end{align*}

\smallskip
\noindent\textbf{9)} \emph{The limit of $L_u^2$ and $L_v^2$.}
 As we shall see, these two limits are the opposite of (and thus cancel out) the previous two limits. However, they are far more laborious to compute.
 
 Starting with $L_u^2$, we first note that the limit $\epsilon\to0$ is standard, giving
\begin{equation}\label{eq: firstLimitOfL2}
    \begin{split}
            \lim_{\epsilon\to0}L_u^2
            = &\,-\int_0^T\int_{\Omega_x\times \po_y}\big(F(u,\overline{v}^c)\cdot \hat{n}\big)\a\rho_\e \dd x\dd \sigma(y)\dd t\\
            &\,- \int_0^T\int_{\Omega_x\times \R^d_y}|b(u)-b(v)|\B^{<r}_{y}\Big[\cha^y,\rho_\e\Big]\a \dd x\dd y\dd t.
    \end{split}
\end{equation}
where $v$ and $\bv$ on the right-hand side of \eqref{eq: firstLimitOfL2} are functions in $(t,y)$. We wish to replace $F(u(t,x),\bv(t,y))$ with $F(u(t,x),\bu(t,x))$ in the first integral on the right-hand side of \eqref{eq: firstLimitOfL2}. To do so, we recall that $\bv$ is the same function as $\bu$ (but in different variables) so we may compute
\begin{equation}\label{eq: differenceOfFsIsOfSizeEpsilon}
    \begin{split}
            &\,\bigg|\int_0^T\int_{\Omega_x\times\po_y} \big(F(u,\overline{v}^c)-F(u,\overline{u}^c)\big)\cdot\hat{n} \a\rho_\e\dd x\dd \sigma(y) \dd t\bigg|\\
    \leq&\,
    L_f\int_0^T\int_{\Omega_x\times\po_y} |\bv(t,y)-\bu(t,x)|\a(t)\rho_\e(x-y)\dd x\dd \sigma(y) \dd t\\
    \leq&\,L_fL_{\bu}T\norm{\a}{L^\infty((0,T))}\int_{\Omega_x\times\po_y} |x-y|\rho_\e(x-y)\dd x\dd \sigma(y),
    \end{split}
\end{equation}
where $L_f$ and $L_{\bu}$ are local Lipschitz constants of $f$ and $\bu$. Using that $\int_{\R^d}|x-y|\rho_\e(x-y)\dd x \leq \e$ we note that the right-most side of \eqref{eq: differenceOfFsIsOfSizeEpsilon} is of size $\lesssim \e$. Moreover, by the divergence theorem we have the identity
\begin{align*}
    -\int_{\po_y}\hat{n}\rho_{\e} \dd \sigma(y)= \int_{\Omega_y} -\nabla_y \rho_\e\dd y= \int_{\Omega_y} \nabla_x \rho_\e\dd y=\nabla_x \varrho_\e,
\end{align*}
where $\varrho_\e$ is as in \eqref{eq: definitionOfVarrho}. Summarizing, the first part of \eqref{eq: firstLimitOfL2} can be written
\begin{align}\label{eq: firstPartOfL2CanBeTurnedIntoTrace}
    -\int_0^T\int_{\Omega_x\times\po_y}\big(F(u,\overline{v}^c)\cdot\hat{n}\big)\a\rho_\e \dd x\dd \sigma(y)\dd t= \int_{Q_x}F(u,\overline{u}^c)\cdot\nabla_x\varrho_\e\a\dd x\dd t + O(\e).
\end{align}
Next, we shall find an analogous representation for the second integral in \eqref{eq: firstLimitOfL2}. For brevity, we will in the coming analysis ignore $\a$ and the integral in $t$ and suppress the $t$-dependence of the remaining expressions. We begin by observing that
\begin{equation}\label{eq: secondPartOfL2CanBeTurnedIntoTrace}
    \begin{split}
          &\Bigg|\int_{\Omega_x\times\R^d_y}\Big(|b(u)-b(v)|-|b(u)-b(\overline{u}^c)|\Big)\B^{<r}_{y}\Big[\cha^y,\rho_\e\Big]\dd x\dd y\Bigg|\\
    \leq &\, \int_{\Omega_x\times\R^d_y}\Big|\big(b(v)-b(\overline{v}^c)\big)\B^{<r}_{y}\Big[\cha^y,\rho_\e\Big]\Big|\dd x\dd y\\
    &\quad+ \int_{\Omega_x\times\R^d_y}\Big|\big(b(\overline{u}^c)-b(\overline{v}^c)\big)\B^{<r}_{y}\Big[\cha^y,\rho_\e\Big]\Big|\dd x\dd y.
    \end{split}
\end{equation}
We show that both expressions on the right-hand side of \eqref{eq: secondPartOfL2CanBeTurnedIntoTrace} vanish uniformly in $\e$ as $r\to0$. For the first one, we exploit that $|b(v)-b(\overline{v}^c)|$ is zero for $y\in\Omega^c$ and obtain 
\begin{align*}
      &\,\int_{\Omega_x\times\R^d_y}\Big|\big(b(v)-b(\overline{v}^c)\big)\B^{<r}_{y}\Big[\cha^y,\rho_\e\Big]\Big|\dd x\dd y\\ \leq&\,\int_{\substack{x,y\in\Omega\\y+z\in\Omega^c\\ |z|<r}}\frac{1}{2}|b(v(y))-b(\overline{v}^c(y))||\rho_\e(x-y-z)-\rho_\e(x-y)|\dd x\dd y\dd \mu(z)\\
      \leq&\,\int_{\substack{y\in\Omega\\y+z\in\Omega^c\\ |z|<r}}|b(v(y))-b(\overline{v}^c(y))|\dd y\dd \mu(z)\to 0,\quad r\to0,
\end{align*}
where the limit holds by the boundary integrability of Proposition \ref{prop: boundaryIntegrability}. The second expression on the right-hand side of \eqref{eq: secondPartOfL2CanBeTurnedIntoTrace} is more tricky. First, since $\bu$ and $\bv$ is the same function and differ only in their arguments (they now depend on $(t,x)$ and $(t,y)$ respectively) we compute   
\begin{equation}\label{eq: secondLeftOverWhenTurningL2IntoTrace}
    \begin{split}
            &\,\int_{\Omega_x\times\R^d_y}\Big|\big(b(\overline{u}^c)-b(\overline{v}^c)\big)\B^{<r}_{y}\Big[\cha^y,\rho_\e\Big]\Big|\dd x\dd y\\
     \leq&\, \frac{L_{b(\overline{u}^c)}}{2}\int_{\substack{x\in\Omega,\\ y\in\R^d,\\
    |z|<r}}|x-y||\cha(y+z)-\cha(y)||\rho_\e(x-y-z)-\rho_\e(x-y)|\dd x\dd y\dd \mu(z),
    \end{split}
\end{equation}
where $L_{b(\overline{u}^c)}$ is a local Lipschitz constant of $b(\overline{u}^c)$.  As the integrand of the previous integral is globally nonnegative, we expand the domain of integration to attain an upper bound. Ignoring the factor $L_{b(\overline{u}^c)}/2$, the previous integral is then bounded by 
\begin{align*}
    &\int_{\substack{x,y\in\R^d,\\
    |z|<r}}|x-y||\cha(y+z)-\cha(y)||\rho_\e(x-y-z)-\rho_\e(x-y)|\dd x\dd y\dd \mu(z)\\
    = &\, \int_{\substack{x, y\in\R^d,\\ |z|<r}}|x||\cha(y+z)-\cha(y)||\rho_\e(x-z)-\rho_\e(x)|\dd x\dd y\dd \mu(z)\\
    \leq &\,|\cha|_{TV(\R^d)}\int_{\substack{x\in\R^d,\\ |z|<r}}|x||\rho_\e(x-z)-\rho_\e(x)||z|\dd x\dd \mu(z),
\end{align*}
where we recall that $|\cha|_{TV(\R^d)}=|\partial\Omega|_{\mathcal{H}^{d-1}}<\infty$ (see comment after Proposition \ref{prop: BCompatiblePairs}). For the quantity $|\rho_\e(x-z)-\rho_\e(x)|$ we have the estimate 
\begin{align*}
    |\rho_\e(x-z)-\rho_\e(x)|\leq \frac{1}{\e^{d}}\Big(2\norm{\rho}{L^\infty(\R^d)}\wedge\tfrac{|z|}{\e}\norm{\nabla\rho}{L^\infty(\R^d)}\Big)\lesssim \frac{|z|}{\e^{d}(|z|+\e)},
\end{align*}
where the last bound exploits that $(1\wedge p)\leq 2p/(1+p)$ when $p\geq 0$, giving the desired conclusion for $p=|z|/\e$.
Moreover, the support of said quantity is contained in $\{(x,z)\colon |x|\wedge|x-z|<\e\}$; for fixed $z$, this set is of measure $\simeq \e^{d}$, and any $x$ in it satisfies $|x|<|z|+\e$. Putting this together yields
\begin{align*}
    \int_{\substack{x\in\R^d}}|x||\rho_\e(x-z)-\rho_\e(x)|\dd x \lesssim   \e^d(|z|+\e) \frac{|z|}{\e^d(|z|+\e)}=|z|,
\end{align*}
so that we may conclude
\begin{align*}
\int_{\substack{x\in\R^d,\\ |z|<r}}|x||\rho_\e(x-z)-\rho_\e(x)||z|\dd x\dd \mu(z)
\lesssim &\int_{|z|<r}|z|^2 \dd \mu(z)\to 0,\quad r\to0.
\end{align*}
To summarize, we have just demonstrated that both terms on the right-hand side of \eqref{eq: secondPartOfL2CanBeTurnedIntoTrace} vanish when taking $\lim_{r\to0}\limsup_{\e\to0}$. For the term we subtracted on the left-hand side of \eqref{eq: secondPartOfL2CanBeTurnedIntoTrace} we use the relation between $\B^{<r}$ and $\L^{<r}$ from Proposition \ref{prop: integrationByPartsFormula} to compute
\begin{align*}
    \int_{\Omega_x\times\R^d_y}|b(u)-b(\overline{u}^c)|\B^{<r}_y\Big[\cha^y,\rho_\e\Big]\dd x\dd y=&\,-\int_{\Omega_x\times \Omega_y}|b(u)-b(\overline{u}^c)|\L^{<r}_y\big[\rho_\e\big]\dd x\dd y\\
     =&\,-\int_{\Omega_x}|b(u)-b(\overline{u}^c)|\L^{<r}_x\big[\varrho_\e\big]\dd x\\
     =&\,\int_{\R^d_x}\B^{<r}_x\Big[|b(u)-b(\overline{u}^c)|,\varrho_\e\Big]\dd x,
\end{align*}
where $\varrho_\e$ is as defined in \eqref{eq: definitionOfVarrho} and where the compact support of $|b(u)-b(\overline{u}^c)|$ justified expanding the domain of integration from $\Omega_x$ to $\R^d_x$ (so to shift $\L^{<r}_x$ over to $\B^{<r}_x$).
All in all, analogous to \eqref{eq: firstPartOfL2CanBeTurnedIntoTrace}, we see that the second integral in \eqref{eq: firstLimitOfL2} can be written
\begin{align}\label{eq: eq: secondPartOfL2CanDefinitelyBeTurnedIntoTrace}
    - \int_0^T\int_{\Omega_x\times \R^d_y}|b(u)-b(v)|\B^{<r}_{y}[\cha^y,\rho_\e]\a \dd x\dd y\dd t = \int_{M_x}\B^{<r}_x\Big[|b(u)-b(\overline{u}^c)|,\varrho_\e\Big]\a\dd x\dd t + o(1),
\end{align}
where the $o$-term is such that $\lim_{r\to0}\limsup_{\e\to0}o(1)=0$. Finally, combining \eqref{eq: firstLimitOfL2} with \eqref{eq: firstPartOfL2CanBeTurnedIntoTrace} and \eqref{eq: eq: secondPartOfL2CanDefinitelyBeTurnedIntoTrace} we conclude that 
\begin{equation}\label{eq: rewritingTheLimOfL2}
    \begin{split}
         &\,\lim_{r\to0}\limsup_{\e\to0}\lim_{\epsilon\to0} L_u^2\\
    = &\,  \lim_{r\to0}\limsup_{\e\to0} \bigg(\int_{Q_x}F(u,\overline{u}^c)\cdot\nabla_x\varrho_\e\a - \int_{M_x}\B^{<r}_{x}\Big[|b(u)-b(\overline{u}^c)|,\varrho_\e\Big]\a\bigg).
    \end{split}
\end{equation}
Observe that $(\varrho_\e)_{\e>0}$ is a boundary layer sequence satisfying $\lim_{\e\to0}\varrho_\e=\frac{1}{2}$ pointwise on $\po$. By the weak trace result of Proposition \ref{prop: weakTrace} we conclude 
\begin{align*}
     &\,\lim_{r\to0}\limsup_{\e\to0}\lim_{\epsilon\to0} L_u^2= \frac{1}{2}\int_\Gamma \a \dd \nu_u,
\end{align*}
where we used that $\lim_{r\to0}\B^{<r}[|b(u)-b(\bu)|,\cha]= 0$ in $L^1(M)$ as follows from Corollary \ref{cor: extraRegularityOfEntropySolution}.
By analogous arguments, we also find that
\begin{align*}
     &\,\lim_{r\to0}\limsup_{\e\to0}\lim_{\epsilon\to0} L_v^2= \frac{1}{2}\int_\Gamma \a \dd \nu_v.
\end{align*}
\smallskip
\noindent\textbf{10)} \emph{The limit of $L_u^3$ and $L_v^3$.}
We will demonstrate that both limits vanish. Starting with $L_u^3$, we begin as before by going to the diagonal in the time variable
\begin{equation}\label{eq: firstLimitOfL3}
    \begin{split}
        \lim_{\epsilon\to0} L_u^3= &\,\int_0^T \int_{\R^d_x\times \R^d_y}|b(u)-b(v)|\B_{x,y}^{<r}\Big[\cha^x,\cha^y\Big]\rho_\e\a\dd x\dd y\dd t\\
        &\,-\int_0^T\int_{\R^d_x\times \R^d_y}\B^{<r}_y\Big[|b(u)-b(\overline{v}^c)|,\cha^y\Big]\cha^x\rho_\e\a\dd x\dd y \dd t,
    \end{split}
\end{equation}
where the time dependence of $v$ and $\overline{v}^c$ in now in $t$. For brevity, we again ignore $\a$, the integral in $t$ and arguments in $t$.

The second integral in \eqref{eq: firstLimitOfL3} is easily seen to vanish uniformly in $\e$ as $r\to0$. Indeed, as $b(u)$ depends on $x$ and not $y$ it follows by a triangle argument that
\begin{equation}\label{eq: theSecondPartOfL3Vanish}
    \begin{split}
         &\bigg|\int_{\R^d_x\times \R^d_y}\B^{<r}_y\Big[|b(u)-b(\overline{v}^c)|,\cha^y\Big]\cha^x\rho_\e\dd x\dd y\bigg|\\
    \leq &\,  \frac{1}{2}\int_{\substack{x,y\in\R^d\\ |z|<r}}\big|b(\overline{v}^c(y+z))-b(\overline{v}^c(y))\big|\big|\cha(y+z)-\cha(y)\big|\rho_\e(x-y)\dd x\dd y\dd \mu(z)\\
   \leq &\,  \frac{L_{b(\overline{v}^c)}}{2} \int_{\substack{y\in\R^d\\ |z|<r}}|z|\big|\cha(y+z)-\cha(y)\big|\dd y\dd \mu(z)\\
    \leq &\, \frac{L_{b(\overline{v}^c)}}{2}|\cha|_{TV} \int_{|z|<r}|z|^2\dd \mu(z)\to0, \qquad\textup{$r\to0$},
    \end{split}
\end{equation}
where $L_{b(\overline{v}^c)}$ denotes a local Lipschitz constant of $\overline{v}^c$. Thus, we turn our attention to the first integral in \eqref{eq: firstLimitOfL3}. We begin by a similar argument as was used for the second integral in \eqref{eq: firstLimitOfL2}: Observe that 
\begin{equation}\label{eq: splittingTheFirstPartOfL3}
    \begin{split}
         &\bigg|\int_{\R^d_x\times\R^d_y}\Big(|b(u)-b(v)|-|b(\bu)-b(\bv )|\Big)\B_{x,y}^{<r}\Big[\cha^x,\cha^y\Big]\rho_\e\dd x\dd y\bigg|\\
    \leq&\int_{\R^d_x\times\R^d_y}\Big(|b(u)-b(\overline{u}^c)|+|b(v)-b(\overline{v}^c)|\Big)\Big|\B_{x,y}^{<r}\Big[\cha^x,\cha^y\Big]\Big|\rho_\e \dd x\dd y.
    \end{split}
\end{equation}
We now show that the right-hand side vanish uniformly in $\e$ as $r\to0$. Focusing on the $|b(u)-b(\bu)|$ term (which is supported in $\Omega_x$)  we have
\begin{align*}
    &\,\int_{\R^d_x\times\R^d_y}|b(u)-b(\overline{u}^c)|\Big|\B_{x,y}^{<r}\Big[\cha^x,\cha^y\Big]\Big|\rho_\e\dd x\dd y\\
    = &\,\frac{1}{2}\int_{\substack{x\in\Omega\\ x+z\in\Omega^c\\ |z|<r}} |b(u(x))-b(\overline{u}^c(x))|\Bigg(\int_{\R^d_y} |\cha(y+z)-\cha(y)|\rho_\e\dd y\Bigg) \dd x\dd \mu(z)\\
    \leq &\, \int_{\substack{x\in\Omega\\ x+z\in\Omega^c\\ |z|<r}} |b(u(x))-b(\overline{u}^c(x))| \dd x\dd \mu(z)\to0, \qquad\textup{$r\to0$},
\end{align*}
where the limit holds by the boundary integrability of Proposition \ref{prop: boundaryIntegrability}. As a corresponding computation can be carried out for the $|b(v)-b(\bv)|$ term, we conclude that the right-hand side of \eqref{eq: splittingTheFirstPartOfL3} vanish when taking $\lim_{r\to0}\limsup_{\e\to0}$. Thus, we may replace $|b(u)-b(v)|$ with the more regular $|b(\bu)-b(\bv)|$ when seeking the limit of the first integral in \eqref{eq: firstLimitOfL3}. Moreover, by a similar argument used to prove Proposition \ref{prop: integrationByPartsFormula} we may shift $\B^{<r}_{x,y}[\cha^x,\cdot]$ from $\cha^y$ over to $|b(\bu)-b(\bv)|\rho_\e$ giving 
\begin{align*}
    &\int_{\R^d_x\times \R^d_y}|b(\overline{u}^c)-b(\overline{v}^c)|\B_{x,y}^{<r}\Big[\cha^x,\cha^y\Big]\rho_\e\dd x\dd y
    =\int_{\R^d_x\times \R^d_y}\B_{x,y}^{<r}\Big[\cha^x,|b(\overline{u}^c)-b(\overline{v}^c)|\rho_\e\Big]\cha^y\dd x\dd y.
\end{align*}
And by a triangle inequality argument, we find
\begin{align*}
    &\int_{\R^d_x\times \Omega_y}\Big|\B_{x,y}^{<r}\Big[\cha^x,|b(\overline{u}^c)-b(\overline{v}^c)|\rho_\e\Big]\Big|\dd x\dd y\\
    \leq&\,\int_{\R^d_x\times \Omega_y}\int_{|z|<r}|\cha(x+z)-\cha(x)||b(\overline{v}^c(y+z))-b(\bv(y))|\rho_\e(x-y-z)\dd\mu(z) \dd x\dd y\\
    &\,+\int_{\R^d_x\times \Omega_y}\int_{|z|<r}|\cha(x+z)-\cha(x)||b(\overline{u}^c)-b(\overline{v}^c)||\rho_\e(x-y-z)-\rho_\e(x-y)|\dd\mu(z) \dd x\dd y.
\end{align*}
The terms on the right-hand side are of similar form as the ones from \eqref{eq: theSecondPartOfL3Vanish} and \eqref{eq: secondLeftOverWhenTurningL2IntoTrace} respectively; we infer that both vanish uniformly in $\e$ as $r\to0$. 

In conclusion, both integrals on the right-hand side of \eqref{eq: firstLimitOfL3} vanish when taking $\lim_{r\to0}\limsup_{\e\to0}$, and so
\begin{align*}
    \lim_{r\to0}\limsup_{\e\to0}\lim_{\epsilon\to0} L_u^3=0.
\end{align*}
By analogous arguments, we also obtain
\begin{align*}
    \lim_{r\to0}\limsup_{\e\to0}\lim_{\epsilon\to0} L_v^3=0.
\end{align*}

\smallskip
\noindent\textbf{11)} \emph{Concluding.} Together with \eqref{eq: uniquenessFourthInequality} and \eqref{eq: justifyingTheIndividualLimitsOfTheLTerms}, these limits imply that
\begin{align}\label{eq: uniquenessFifthInequality}
     &\,-\int_{Q} |u(t,x)-v(t,x)|\a'(t)\dd x\dd t \leq 0.
\end{align}
This in turn gives for a.e.~$\tau\in(0,T)$ the $L^1$-contraction
\begin{align*}
    &\,\int_{\Omega} |u(\tau,x)-v(\tau,x)|\dd x \leq \int_{\Omega}|u_0(x)-v_0(x)|\dd x.
\end{align*}
Indeed, this standard implication follows by letting $\a\to \chi_{(0,\tau)}$ pointwise while $\norm{\a'}{L^1((0,T))}\leq 2$; the initial $L^1$-continuity of $u$ and $v$ (Lemma \ref{lem: timeContinuityAtZeroOfEntropySolution}) and the Lebesgue differentiation theorem then gives the above inequality for a.e.~$\tau\in(0,T)$. With this, uniqueness is established.
\end{proof}


\subsection*{Acknowledgments}
We would like to thank the anonymous referees for a thorough reading and insightful suggestions that helped us improve the paper.\\
This research has been supported by the Research Council of Norway through different grants: All authors were supported by the Toppforsk (research excellence) grant agreement no. 250070 `Waves and Nonlinear Phenomena (WaNP)', OM is supported by grant agreement no. 301538 `Irregularity and Noise In Continuity Equations' (INICE), JE was supported by the MSCA-TOPP-UT grant agreement no. 312021, and ERJ is supported by grant agreement no. 325114 `IMod. Partial differential equations, statistics and data: An interdisciplinary approach to data-based modelling'. In addition, JE received funding from the European Union’s Horizon 2020 research and innovation programme under the Marie Sk{\l}odowska-Curie grant agreement no. 839749 `Novel techniques for quantitative behavior of convection-diffusion equations (techFRONT)'. 


\appendix

\section{Entropy solutions are a generalized solution concept}
\label{sec:ClassicalSolutionsAreEntropySolutions}

Entropy solutions in the sense of Definition Definition \ref{def: entropySolutionVovelleMethod} naturally contain classical solutions, and they are more restrictive than very weak solutions.

\begin{lemma}\label{lem:ClassicalSolutionsAreEntropySolutions}
Assume \eqref{Omegaassumption}--\eqref{muassumption}. Then:
\begin{enumerate}[{\rm (a)}]
\item Classical solutions of \eqref{E} are entropy solutions in the sense of Definition \ref{def: entropySolutionVovelleMethod}.
\item Entropy solutions in the sense of Definition \ref{def: entropySolutionVovelleMethod} satisfy
\begin{equation*}
\int_M u\partial_t\varphi+f(u)\cdot \nabla\varphi + b(u)\L[\varphi]\dd x \dd t=0,
\end{equation*}
for all $\varphi\in C_c^\infty(Q)$, i.e., they are (very) weak solutions.
\end{enumerate}

\end{lemma}

\begin{proof}
\noindent(a) Let us assume that $u,b(u)\in C^2_{b}([0,T)\times\R^d)$ and $u$ solves \eqref{E} classically. It remains to check part (a) of Definition \ref{def: entropySolutionVovelleMethod}. To that end, multiply the PDE by $\sgn_{\epsilon}^{\pm}(u-k)\varphi$, where
\begin{align*}
    \sgn_{\epsilon}^{\pm}(s)=\begin{cases}0,\quad \pm s\leq 0,\\
    \frac{s}{\epsilon},\quad 0\leq\pm s\leq \epsilon,\\
    \pm 1,\quad \epsilon\leq \pm s,
    \end{cases}
\end{align*}
and $0\leq \varphi\in C_c^{\infty}([0,T)\times \R^d)$ satisfies \eqref{eq: admissibilityConditionOnConstantAndTestfunction}. With the short-hand notation
$$
(\psi)_\epsilon^\pm =\int_0^{\psi}\sgn^\pm_\epsilon(\xi)\dd \xi,
$$
for (real valued) expressions $\psi$, we get after integrating the PDE over $Q$ (and performing several chain rules, integrations by parts, and using \eqref{eq: admissibilityConditionOnConstantAndTestfunction})
\begin{align*}
&\,-\int_{Q} \Big((u-k)_{\epsilon}^{\pm}\dell_t \varphi +\sgn_{\epsilon}^{\pm}(u-k)\big(f(u)-f(k)\big)\cdot\nabla \varphi\Big)\dd x\dd t\\
&\,-\int_Q \L^{\geq r}[b(u)]\sgn_\epsilon^{\pm}(u-k)\varphi\dd x\dd t -\int_M (b(u)-b(k))_\epsilon^{\pm}\L^{<r}[\varphi]\dd x\dd t\\
= &\, \int_\Omega(u_0-k)_\epsilon^{\pm}\varphi(0,\cdot)\dd x + L_f\int_\Gamma (\overline{u}^c-k)_\epsilon^\pm\varphi\dd \sigma(x)\dd t\\
&\quad- \int_\Gamma\Big(\sgn_\epsilon^\pm(u-k)(f(u)-f(k))\cdot \hat{n}+L_f(u-k)_{\epsilon}^{\pm}\Big)\varphi\dd \sigma(x)\dd t\\
&\quad + \int_M \Big(\L^{<r}[b(u)]\sgn_\epsilon^\pm(u-k)\cha-\L^{<r}[(b(u)-b(k))_{\epsilon}^\pm]\Big)\varphi\dd x\dd t\\
&\quad +\int_Q (f(u)-f(k))\cdot(\nabla u) (\sgn_\epsilon^{\pm})'(u-k)\varphi \dd x\dd t.\\
\end{align*}
We then 
send $\epsilon\to 0$ to obtain
\begin{align*}
&\,-\int_{Q}\Big((u-k)^{\pm} \dell_t\varphi +F^{\pm}(u,k)\cdot\nabla \varphi\Big)\dd x\dd t \\
&\,-\int_M \L^{\geq r}[b(u)-b(k)]\sgn^{\pm}(u-k)\varphi\cha\dd x\dd t -\int_M (b(u)-b(k))^{\pm}\L^{<r}[\varphi]\dd x \dd t\\
= &\, \int_\Omega(u_0-k)^{\pm}\varphi(0,\cdot)\dd x + L_f\int_\Gamma (\overline{u}^c-k)^\pm\varphi\dd\sigma(x)\dd t - \int_M \varphi \dd \mu^{\pm,r},\\
\end{align*}
where
\begin{align*}
\mu^{\pm,r} &\coloneqq\Big(\hat{n}\cdot F^{\pm}(u,k)+L_f(u-k)^{\pm}\Big)\delta_{(t,x)\in\Gamma}\\
&\quad +\Big(\L^{<r}[(b(u)-b(k))^\pm]-\L^{<r}[b(u)-b(k)]\sgn^\pm(u-k)\cha\Big).
\end{align*}\cb
The proof is complete if we can demonstrate that $\int_M \varphi \dd \mu^{\pm,r}\geq 0$. Clearly, the first part of $\mu^{\pm,r}$ is nonnegative since $|\hat{n}\cdot F^{\pm}(u,k)|\leq L_f (u-k)^\pm$, and since $\varphi\geq 0$ we thus infer that
\begin{equation*}
\begin{split}
\int_M \varphi \dd \mu^{\pm,r}\geq &\,\int_M \Big(\L^{<r}[(b(u)-b(k))^\pm]-\L^{<r}[b(u)-b(k)]\sgn^\pm(b(u)-b(k))\cha\Big)\varphi\dd x\dd t\\
=&\,\int_{Q} \Big(\L^{<r}[(b(u)-b(k))^\pm]-\L^{<r}[b(u)-b(k)]\sgn^\pm(b(u)-b(k))\Big)\varphi\dd x\dd t\\
&\, + \int_{M\setminus Q}\L^{<r}[(b(u)-b(k))^\pm]\varphi\dd x\dd t.
\end{split}
\end{equation*}
The first term on the right-hand side is nonnegative because of Corollary \ref{cor:ConvexInequalityNonlocalOperators} (specifically its second part), while the second term is nonnegative since \eqref{eq: admissibilityConditionOnConstantAndTestfunction} gives us the identity
\begin{align*}
    \int_{M\setminus Q}\L^{<r}[(b(u)-b(k))^\pm]\varphi\dd x\dd t=\int_{M\setminus Q}\int_{|z|<r}\big(b(u(t,x+z))-b(k)\big)^\pm\varphi(t,x)\dd\mu(z)\dd x\dd t\geq0,
\end{align*}
and so we are done. \nc

\smallskip
\noindent(b) The proof follows by Definition \ref{def: entropySolutionVovelleMethod}, the convex inequality
$\L^{\geq r}[\psi]\sgn^\pm\psi \leq  \L^{\geq r}[\psi^\pm ]$ (cf. Corollary \ref{cor:ConvexInequalityNonlocalOperators}), and considering $(u-k^\mp)^\pm=\pm(u-k^\mp)$ where
$$
k^{\pm}=\pm\big(\norm{u_0}{L^\infty(\Omega)}+\norm{u^c}{L^\infty(Q^c)})
$$
and 
$k^-\leq u\leq k^+$ by Lemma \ref{lem: maximumsPrinciple}. For $0\leq\varphi\in C_c^\infty(Q)$ we then conclude from \eqref{eq: entropyInequalityVovelleMethod} that
\begin{equation*}
\begin{split}
&\mp\int_Q \Big( u\partial_t\varphi + f(u)\cdot\nabla\varphi \Big)+\int_M b(u)\L[\varphi]
\leq \mp\int_Q \Big( k^{\mp}\partial_t\varphi + f(k^{\mp})\cdot\nabla\varphi \Big),
\end{split}
\end{equation*}
where the non-singular operator $\L^{\geq r}$ has been shifted over to $\varphi$ using that $\L^{\geq r}$ is self adjoint (Proposition \ref{prop: integrationByPartsFormula}) followed by letting $r\to 0$.
As $\varphi$ has compact support in $Q$, the right-hand side is zero and thus
\begin{equation}\label{eq:EntropySolutionsAreWeakSolutions}
-\int_Q\Big(u\partial_t\varphi+f(u)\cdot \nabla\varphi\Big)\dd x \dd t+\int_Mb(u)\L[\varphi]\dd x \dd t=0.
\end{equation}
The general case when $\varphi\in C_c^\infty(Q)$ holds by picking $0\leq \psi\in C_c^\infty(Q)$ such that $0\leq \phi\coloneqq \psi-\varphi$ and using that \eqref{eq:EntropySolutionsAreWeakSolutions} holds for $\phi$ and $\psi$.
\end{proof}


\section{Proofs of basic properties of entropy solutions}
\label{app:div_pfs}

We will prove the maximum principle and time continuity at $t=0$ for entropy solutions.

\begin{proof}[Proof of Lemma \ref{lem: maximumsPrinciple}]
In the entropy inequality \eqref{eq: entropyInequalityVovelleMethod} for $u$ we set $k$ in the $(+)$ case to $k^+\coloneqq \sup_{\Omega} u_0\vee  \sup_{ Q^c} u^c$ and in the $(-)$ case to $k^-\coloneqq \inf_{\Omega} u_0\wedge \inf_{Q^c} u^c$. Then the condition \eqref{eq: admissibilityConditionOnConstantAndTestfunction} is automatically satisfied for every $0\leq \varphi\in C_c^\infty([0,T)\times \R^d)$. We then set $\varphi(t,x)=\theta(t)\phi(x)$ with $0\leq \theta\in C^\infty_c([0,T))$ and with $0\leq \phi\in C^\infty_c(\R^d)$, and where $\phi=1$ in $\Omega$ and $\phi\leq 1$ in $\Omega^c$. By sending $r\to\infty$ in \eqref{eq: entropyInequalityVovelleMethod} we find that
\begin{equation*}
\begin{split}
&\,-\int_Q  (u-k^\pm)^\pm\theta'\dd x\dd t - \int_M(b(u)-b(k^\pm))^\pm\L[\phi]\theta\dd x\dd t\leq 0.
\end{split}
\end{equation*}
As $(b(u)-b(k^\pm))^{\pm}=0$ a.e.~in $Q^c$ and $\L[\phi]\leq 0$ in $Q$ ($\phi$ attains its global maximum on every point in $Q$), the integrand of the second integral is clearly non-positive and we conclude that
\begin{align*}
    -\int_Q (u-k^\pm)^\pm\theta' \dd x\dd t\leq 0, \quad \implies \quad \int_Q (u-k^\pm)^\pm \dd x\dd t\leq 0 ,
\end{align*}
where the last implication follows from setting $\theta(t)=T-t$, which can be approximated by nonnegative elements of $C_c^\infty([0,T))$. This concludes the proof.
\end{proof}

\begin{proof}[Proof of Lemma \ref{lem: timeContinuityAtZeroOfEntropySolution}]
We follow Lemma 7.41 in \cite{MaNeRoRu96}. Let $(\theta_\epsilon)_{\epsilon>0}\subset C^\infty_c([0,T))$ be a family of functions satisfying $0\leq\theta_\epsilon(t)\leq\theta_\epsilon(0)=1$, $\supp(\theta_\epsilon)\subset[0,\epsilon)$, and $\theta_\epsilon(t)-\theta_{\epsilon'}(t)\geq 0$ when $\epsilon\geq \epsilon'>0$. For any $0\leq \b\in C^\infty_c(\Omega)$, we set $\varphi(t,x)=(\theta_\epsilon-\theta_{\epsilon'})(t)\b(x)$ in \eqref{eq: entropyInequalityVovelleMethod}, add the inequalities for $(\cdot)^+$ and $(\cdot)^-$, and since all terms are bounded and $\theta_\epsilon\to 0$ as $\epsilon\to0$, we observe that 
\begin{align}
\label{eq: nonDecreasingSequenceInTheLimit}
    \limsup_{\epsilon\to0}\sup_{\epsilon'\in(0,\epsilon)}-\iint_{Q} |u(t,x)-k| \b(x)  (\theta_{\epsilon}'-\theta_{\epsilon'}')(t) \dd x\dd t \leq 0,
\end{align}
or equivalently, $\underset{\epsilon\to0}{\limsup}\iint_{Q} |u(t,x)-k| \b(x)  (-\theta_{\epsilon}')(t) \dd x\dd t \leq    \underset{\epsilon'\to0}{\liminf}\iint_{Q} |u(t,x)-k| \b(x)  (-\theta_{\epsilon'}')(t) \dd x\dd t$.
In the same way, but with $\varphi(t,x)=\theta_\epsilon(t)\b(x)$, we find that
\begin{align}\label{int-lim}
    \lim_{\epsilon\to0}\int_0^\epsilon\Bigg(\int_{\Omega} |u(t,x)-k| \b(x) \dd x \Bigg) (-\theta_\epsilon'(t))\dd t \leq \int_\Omega |u_0(x)-k|\b(x) \dd x,
\end{align}
where the limit on the left-hand side exists by \eqref{eq: nonDecreasingSequenceInTheLimit}. It then follows that 
\begin{align}\label{lim2}
    \esslimsup_{t\to0+}\int_{\Omega} |u(t,x)-k| \b(x) \dd x \leq\int_\Omega |u_0(x)-k|\b(x) \dd x.
\end{align}

Indeed, assume \eqref{lim2} is false: There is an $r>0$ and a set $E\subset [0,T]$ with $\textup{meas} \{E\cap (0,\delta)\}>0$ for all $\delta>0$, such that $\int_{\Omega} |u(t,x)-k| \b(x) \dd x> \int_\Omega |u_0(x)-k|\b(x) \dd x+r$ for $t\in E$. Let  $\theta_\epsilon(t)=1-\int_0^t\psi(s)\dd s$ and $\psi(s)=\frac{1}{\textup{meas}\,E_\epsilon}1_{E_\epsilon}*\rho_\delta(t)$ where $\delta\in(0,\epsilon)$, $E_\epsilon:=E\cap (0,\epsilon)$, $\rho_\delta=\frac1\delta\rho(\frac t\delta)$, and $0\leq\rho\in C^\infty_c((0,T))$ with $\int_0^T\rho \dd t=1$. 
Integration in time against $\theta'_\epsilon=-\psi$ and sending $\delta\to0$ then gives for all $\epsilon>0$,
\begin{align*}
    \lim_{\delta\to0}\int_0^\epsilon\Bigg(\int_{\Omega} |u(t,x)-k| \b(x) \dd x \Bigg) (-\theta_\epsilon'(t))\dd t 
    = &\,\frac1{\textup{meas}\,E_\epsilon}\int_{E_\epsilon}\int_{\Omega} |u(t,x)-k| \b(x) \dd x  \dd t \\
    &> \int_\Omega |u_0(x)-k|\b(x) \dd x +r,
\end{align*}
which contradicts \eqref{int-lim}.

By approximation, \eqref{lim2} holds also when $\beta$ is a characteristic function. 
Since $k$ is arbitrary, it follows by linearity of integrals and density of 
simple functions in $L^1(\Omega)$, that for any $v_0\in L^\infty(\Omega)\subset L^1(\Omega)$,
\begin{align*} \esslimsup_{t\to0+}\int_{\Omega} |u(t,x)-v_0(x)| \dd x \leq\int_\Omega |u_0(x)-v_0(x)| \dd x.
\end{align*}
Taking $v_0=u_0$ and using $0\geq \esslimsup [\dots] \geq \essliminf [\dots]\geq 0$, the proof is complete.
%
\end{proof}


\section{Convex inequality for nonlocal operators}\label{sec:ConvexInequalityNonlocalOperators}
The following lemma is well-known, and so we omit the proof.
\begin{lemma}\label{lem:ConvexFunctions}
Assume $\eta:\R\to\R$ is convex. Then, for any real numbers $v,w$, we have
$$
\eta(v)-\eta(w)\geq \eta'(w)(v-w),
$$
where $\eta'$ is a subderivative of $\eta$, that is, $\eta'(w)\in [\eta'(w-),\eta'(w+)]$.
\end{lemma}
\begin{corollary}\label{cor:ConvexInequalityNonlocalOperators}
Let $\L$ be a Lévy operator and $\eta\colon \R\to\R$ be convex. For $\phi\in L^\infty(\R^d)$ we have
$$
\L[\phi]\eta'(\phi)\leq\L[\eta(\phi)]\qquad\text{a.e.}\qquad \text{in $\R^d$},
$$
provided the expressions are well-defined in $L^1_{\loc}(\R^d)$.
 In the special case $\eta=(\cdot)^{\pm}$, and with $h\colon \R\to\R$ a non-decreasing function, then for any $\phi,\psi\in L^\infty(\R^d)$ we have
\begin{align*}
    \L[h(\phi)-h(\psi)]\sgn^{\pm}(\phi-\psi)\leq \L\big[\big(h(\phi)-h(\psi)\big)^{\pm}\big]\qquad\text{a.e.}\qquad \text{in $\R^d$},
\end{align*}
again, provided that the expressions are well-defined in $L^1_{\loc}(\R^d)$.
\end{corollary}

\begin{proof}\cb
It suffices to prove the results for zero order Lévy operators, where $\mu(\R^d)<\infty$, since the general case then follows from a truncation argument (by assumption, $\lim_{r\to 0}\L^{\geq r}=\L$ for the involved expressions). Thus, we assume $\L$ is of zero order. 

The first inequality is a direct application of Lemma \ref{lem:ConvexFunctions} and so we focus on the second one: One gets the desired bound after writing out $\L[h(\phi)-h(\psi)](x)\sgn^{\pm}(\phi-\psi)(x)$ in full, and then use the following bound and identity
\begin{equation}\label{eq: twoUsefulIdentitiesForTheConcaveEstimate}
    \begin{split}
            \Big(h(\phi(x+z))-h(\psi(x+z))\Big)\sgn^\pm\big(\phi(x)-\psi(x)\big)\leq&\, \Big(h(\phi(x+z))-h(\psi(x+z))\Big)^{\pm},\\
    \Big(h(\phi(x))-h(\psi(x))\Big)\sgn^\pm\big(\phi(x)-\psi(x)\big)=&\, \Big(h(\phi(x))-h(\psi(x))\Big)^{\pm}.
    \end{split}
\end{equation}
The first line in \eqref{eq: twoUsefulIdentitiesForTheConcaveEstimate} follows from the two trivial implications 
\begin{equation*}
    w\in[0,1]\implies vw\leq (v)^+, \qquad  w\in[-1,0]\implies vw\leq (v)^-,
\end{equation*}
valid for all $v\in\R$, while the second line in \eqref{eq: twoUsefulIdentitiesForTheConcaveEstimate} follows from the fact that $h$ is non-decreasing.\nc
\end{proof}


\section{Additional identities for the bilinear operator}

We here show that the expression $\B_{x+y}[\phi,\psi]$ is somewhat analogous to the expression  
$$
(\nabla_x\phi+\nabla_y\phi)\cdot(\nabla_x\psi+\nabla_y\psi),
$$
and that we, in particular, need not introduce ``cross-terms'' for the diffusion in the doubling of variables argument as is needed in the local case (cf. \cite{Car99, KaRi03, MaPoTe02, MiVo03}). 

\begin{proposition}[Cross-terms formula]\label{prop: CarrilloStyleCross-Terms}
Assume \eqref{muassumption}. Let $\phi,\psi\in L^\infty(\R^d\times\R^d)$ be functions in $(x,y)$ and let $\psi$ have compact support. Using the operator notation from Section \ref{sec:AssumptionsConceptMain}, we then have
\begin{align*}
    &\,\int_{\R^{d}\times\R^d}\B_{x+y}[\phi,\psi]\dd x \dd y= \int_{\R^{d}\times\R^d}\B_{x,x+y}[\phi,\psi]+\B_{y,x+y}[\phi,\psi]\dd x \dd y,
\end{align*}
provided the expressions in the integrals are well defined in $L^1(\R^d\times\R^d)$.
\end{proposition}

\begin{proof}
Fix $r>0$. Observe that the integrands in the following argument are absolutely integrable with respect to $\dd x\dd y\dd\mu(z)$ due to the compact support of $\psi$ and the truncation of our operators; in particular we are free to deploy Fubini's theorem. We compute
\begin{align*}
    &\,2\int_{\R^d\times\R^d}\B_{x,x+y}^{\geq r}[\phi,\psi]\dd x\dd y\\
    = &\, \int_{\R^d\times\R^d}\int_{|z|\geq r}\big(\phi(x+z,y)-\phi(x,y)\big)\big(\psi(x+z,y+z)-\psi(x,y)\big)\dd \mu(z)\dd x\dd y\\
    = &\,\int_{|z|\geq r}\int_{\R^d\times\R^d}\big(\phi(x,y-z)-\phi(x-z,y-z)\big)\psi(x,y)\dd x\dd y\dd \mu(z)\\
    &\quad-\int_{\R^d\times\R^d}\int_{|z|\geq r}\big(\phi(x+z,y)-\phi(x,y)\big)\psi(x,y)\dd \mu(z)\dd x\dd y\\
    = &\,\int_{\R^d\times\R^d}\int_{|z|\geq r}\big(\phi(x,y)-\phi(x-z,y-z)\big)\psi(x,y)\dd\mu(z)\dd x\dd y\\
    &\quad + \int_{\R^d\times\R^d}\int_{|z|\geq r}\big(\phi(x,y-z)-\phi(x+z,y)\big)\psi(x,y)\dd \mu(z)\dd x\dd y.
\end{align*}
By the change of variables $z\mapsto -z$ and the symmetry of $\dd\mu(z)$ we see that the first integral on the right-hand side coincides with
\begin{align*}
    -\int_{\R^{d}\times \R^d}\L_{x+y}^{\geq r}[\phi]\psi\dd x\dd y =  \int_{\R^{d}\times \R^d}\B_{x+y}^{\geq r}[\phi,\psi]\dd x\dd y ,
\end{align*}
where the identity follows from Proposition \ref{prop: integrationByPartsFormula}. By a similar computation for the $\B_{y,x+y}^{\geq r}$ operator, we conclude
\begin{align*}
   &\,\int_{\R^d\times\R^d}\Big(\B_{x,x+y}^{\geq r}+ \B_{y,x+y}^{\geq r}\Big)[\phi,\psi] \dd x\dd y\\
   = &\, \int_{\R^d\times\R^d}\B_{x+y}^{\geq r}[\phi,\psi]\dd x\dd y \\
     &\,+ \frac{1}{2}\int_{\R^d\times\R^d}\int_{|z|\geq r}\big(\phi(x,y-z)-\phi(x+z,y)\big)\psi(x,y)\dd \mu(z)\dd x\dd y\\
    &\, + \frac{1}{2}\int_{\R^d\times\R^d}\int_{|z|\geq r}\big(\phi(x-z,y) - \phi(x,y+z)\big)\psi(x,y)\dd \mu(z)\dd x\dd y.
\end{align*}
The two last integrals cancel out, 
as can be seen by substituting $z\mapsto -z$ in either, and so the result follows by letting $r\to0$.
\end{proof}

We also introduce a specific product rule, which is useful in the uniqueness proof.
\begin{lemma}[An auxiliary product rule]\label{aSpecialProductRule}
Assume \eqref{muassumption}. Let $\phi\in L^\infty(\R^d\times\R^d)$ be a function in $(x,y)$, and let $\zeta^x,\zeta^y,\rho\in L^\infty(\R^d)$ be functions in $x$, $y$, and $x-y$ respectively. Let further $\z^x$ and $\z^y$ have compact support. Using the operator notation from Section \ref{sec:AssumptionsConceptMain}, we then have 
\begin{align*}
       &\,\int_{\R^d\times\R^d}\B_{x,x+y}\Big[\phi,\z^x\z^y\Big]\rho\dd x\dd y\\
   = &\,\int_{\R^d\times\R^d}\B_x\Big[\phi(\cdot,y),\z^x\Big]\z^y\rho\dd x \dd y + \int_{\R^d\times\R^d}\phi\Big(\B_y\Big[\z^y,\rho(x-\cdot)\Big] \z^x -\B_{x,y}\Big[\z^x,\z^y\Big]\rho\Big)\dd x \dd y,
\end{align*}
provided the expressions in the integrals are well-defined in $L^1(\R^d\times \R^d)$.
\end{lemma}

\begin{proof}
Fix $r>0$. The compact support of $\z^x,\z^y$ ensures that the following integrals are well-defined. Adding and subtracting $\z^x(x+z)\z^y(y)$ to the integrand, and performing a few changes of variables, we compute
\begin{align*}   &\,\int_{\R^d\times\R^d}\B_{x,x+y}^{\geq r}\Big[\phi,\z^x\z^y\Big]\rho\dd x\dd y\\
    = &\,\frac{1}{2}\int_{\R^d\times\R^d}\int_{|z|\geq r}\big(\phi(x+z,y)-\phi(x,y)\big)\big(\z^x(x+z)\z^y(y+z) - \z^x(x)\z^y(y)\big)\rho(x-y)\dd \mu(z)\dd x\dd y\\
    = &\,\frac{1}{2}\int_{\R^d\times\R^d}\int_{|z|\geq r}\big(\phi(x+z,y)-\phi(x,y)\big)\big(\z^x(x+z) - \z^x(x)\big)\z^y(y)\rho(x-y)\dd \mu(z)\dd x\dd y\\
    &+\frac{1}{2}\int_{\R^d\times\R^d}\int_{|z|\geq r}\big(\phi(x+z,y)-\phi(x,y)\big)\big(\z^y(y+z) - \z^y(y)\big)\z^x(x+z)\rho(x-y)\dd \mu(z)\dd x\dd y.
\end{align*}
Note that the first integral on the right-hand side coincides with $\int_{\R^d\times\R^d}\B^{\geq r}_{x}[\phi(\cdot,y),\z^x]\z^y\rho\dd x\dd y$. For the second integral, we split it up and perform a change of variables giving  
\begin{align*}
&\, \frac{1}{2}\int_{\R^d\times\R^d}\int_{|z|\geq r}\big(\phi(x+z,y)-\phi(x,y)\big)\big(\z^y(y+z) - \z^y(y)\big)\z^x(x+z)\rho(x-y)\dd \mu(z)\dd x\dd y\\
        = &\,\frac{1}{2}\int_{\R^d\times\R^d}\int_{|z|\geq r}\phi(x,y)\big(\z^y(y+z) - \z^y(y)\big)\z^x(x) \rho(x-z-y) \dd \mu(z)\dd x\dd y\\
    &\quad-\frac{1}{2}\int_{\R^d\times\R^d}\int_{|z|\geq r}\phi(x,y)\big(\z^y(y+z) - \z^y(y)\big)\z^x(x+z)\rho (x-y)\dd \mu(z)\dd x\dd y,
\end{align*}
If we further add and subtract the integral of $\phi(x,y)\big(\z^y(y+z)-\z^y(y)\big)\z^x(x)\rho(x-y)$ to this right-hand side we get
\begin{align*}   \int_{\R^d\times\R^d}\phi\B_y^{\geq r}\Big[\z^y,\rho(x-\cdot)\Big] \z^x\dd x\dd y - \int_{\R^d\times\R^d}\phi\B_{x,y}^{\geq r}\Big[\z^x,\z^y\Big]\rho\dd x \dd y,
\end{align*}
where we used that $\rho(x-z-y)=\rho(x-(y+z))$. Letting $r\to 0$ gives the result.
\end{proof}




\end{document}